\documentclass[a4paper,10pt,leqno]{article}
\title{Hochschild and cyclic homology of Yang-Mills algebras}
\author{Estanislao Herscovich and Andrea Solotar
\thanks{This work has been supported by the projects UBACYTX212, Mathamsud NOCOMALRET, PIP-CONICET 112-200801-00487 and PICT 2007-02182.
The first author is a CONICET fellow.
The second author is a research member of CONICET (Argentina) and a
Regular Associate of ICTP Associate Scheme.}}
\date{}

\input{xy}
\xyoption{all}
\xyoption{poly}
\usepackage[all]{xy}
\usepackage{amsmath,amsthm,amsfonts,amssymb,stmaryrd,fancyhdr}
\usepackage[fleqn,tbtags]{mathtools}
\usepackage{latexsym}
\usepackage{color}
\usepackage{graphicx}
\usepackage{float}  
\usepackage{fullpage}
\usepackage[small, bf, margin=90pt, tableposition=bottom]{caption}
\usepackage[T1]{fontenc}
\usepackage{palatino}

\usepackage[
  breaklinks,
  naturalnames,
  ]{hyperref}

\newtheorem{theorem}{Theorem}[section]
\newtheorem{theorem*}{Theorem}
\newtheorem{corollary}[theorem]{Corollary}
\newtheorem{lemma}[theorem]{Lemma}
\newtheorem{proposition}[theorem]{Proposition}
\theoremstyle{definition}        
\newtheorem{remark}[theorem]{Remark}
\numberwithin{equation}{section}                    

\newcommand\cl[1]{{\langle#1\rangle}}
\def\id{{\mathrm{id}}}
\newcommand\CC{{\mathbb{C}}}
\newcommand\RR{{\mathbb{R}}}

\newcommand\ZZ{{\mathbb{Z}}}
\newcommand\NN{{\mathbb{N}}}
\newcommand\PP{{\mathbb{P}}}
\newcommand\SSS{{\mathbb{S}}}

\newcommand\so{{\mathfrak{so}}}

\newcommand\SO{{\mathrm{SO}}}

\def\B{{\mathcal B}}
\def\C{{\mathcal C}}

\def\I{{\mathcal I}}
\def\J{{\mathcal J}}

\def\O{{\mathcal O}}

\def\T{{\cal T}}
\def\U{{\mathcal U}}

\def\Z{{\mathcal Z}}

\def\YM{{\mathrm {YM}}}
\def\TYM{{\mathrm {TYM}}}

\def\ad{{\mathrm {ad}}}

\def\Ker{{\mathrm {Ker}}}
\def\Coker{{\mathrm {Coker}}}
\def\Hom{{\mathrm {Hom}}}

\def\End{{\mathrm {End}}}
\def\tr{{\mathrm {tr}}}

\def\Der{{\mathrm {Der}}}

\def\Ext{{\mathrm {Ext}}}
\def\Tor{{\mathrm {Tor}}}

\def\g{{\mathfrak g}}
\def\h{{\mathfrak h}}

\def\ym{{\mathfrak{ym}}}
\def\f{{\mathfrak f}}

\def\tym{{\mathfrak{tym}}}
\def\p{{\mathfrak{p}}}

\def\place{{-}}

\begin{document}

\maketitle
\begin{abstract}
The aim of this article is to present a detailed algebraic computation of the Hochschild and cyclic homology groups of the Yang-Mills algebras 
$\YM(n)$ ($n \in \NN_{\geq 2}$) defined by A. Connes and M. Dubois-Violette in \cite{CD1}, 
continuing thus the study of these algebras that we have initiated in \cite{HS1}.
The computation involves the use of a spectral sequence associated to the natural filtration on the universal enveloping algebra 
$\YM(n)$ provided by a Lie ideal $\tym(n)$ in $\ym(n)$ which is free as Lie algebra. 
As a corollary, we describe the Lie structure of the first Hochschild cohomology group.
\end{abstract}



\section{Introduction}
\label{intro}

The homological invariants of an algebra are closely related to the category of its representations. 
Some properties of the category of representations of Yang-Mills algebras have been already analyzed in \cite{HS1}, 
exploiting the Kirillov orbit method.
It is well known that Hochschild cohomology measures the existence of deformations (see \cite{Ge1}). 
The main purpose of this article is to describe in detail Hochschild homology and cohomology of Yang-Mills algebras, 
as well as their cyclic homology.

Let us recall the definition of Yang-Mills algebras given by A. Connes and M. Dubois-Violette in \cite{CD1}.
For a positive integer $n \geq 2$, the Lie Yang-Mills algebra over an algebraically closed field $k$ of characteristic zero is
\[     \ym(n) = \f(n)/\cl{\{ \sum_{i=1}^{n} [x_{i},[x_{i},x_{j}]] : j= 1, \dots, n \}},     \]
where $\f(n)$ is the free Lie algebra on $n$ generators $x_{1}, \dots, x_{n}$. 
We also define $V(n)$ as the $k$-vector space spanned by them. 
The associative enveloping algebra $\U(\ym(n))$ will be denoted $\YM(n)$.
It is a cubic Koszul algebra of global dimension $3$ with Poincar\'e duality and the Calabi-Yau property (see \cite{BT1}, Example 5.1).

The first instance of Yang-Mills theory in physics is through Maxwell's equations for the charge
free situation which gives a representation of the Yang-Mills algebra. 

In general, the Yang-Mills equations we consider are equations for covariant derivatives on bundles over the affine space $\RR^{n}$
endowed with a pseudo-Riemannian metric $g$. 
Any complex vector bundle of rank $m$ over $\RR^{n}$ is trivial and every connection on such a bundle is given by a $M_{m}(\CC)$-valued 
$1$-form $\sum_{i=1}^{n} A_{i} dx^{i}$  with the
corresponding covariant derivative given by $\nabla_{i} = \partial_{i} + A_{i}$. 
The Yang-Mills equations for the covariant derivative are
\[     \sum_{i,j=1}^{n} g^{i,j} [\nabla_{i},[\nabla_{j},\nabla_{k}]] = 0,    \]
where $g^{-1}=(g^{i,j})$ .
Yang-Mills equations have also been recently studied due to their applications to the gauge theory of $D$-branes 
and open string theory (see \cite{Ne1}, \cite{Mov1} and \cite{Doug1}).


Despite the fact that several properties of these algebras can be expressed in a geometrical
language, the arguments we use in the proofs appearing in this article are homological. 
In fact, the main proofs only require a detailed knowledge of the complexes involved. 

It is important to notice that the behaviour of the Yang-Mills algebra with two generators $\YM(2)$ 
is completely different from the other cases (\textit{i.e.} $n \geq 3$). The algebra $\YM(2)$ is isomorphic to the 
enveloping algebra of the Heisenberg Lie algebra. 
The computation of its Hochschild homology and cohomology can be easily done
using the Koszul complex and we recover in this way  
the results obtained by P. Nuss in \cite{Nu1}. 

On the other hand, if the Yang-Mills algebra has a number of generators greater than or equal to three, 
we provide detailed computations in order to describe the zeroth and first Hochschild cohomology groups, together with the Lie structure of the latter. 
Then, using this description, we compute the other Hochschild cohomology groups and all the Hochschild and cyclic homology groups.  
We also recover the results announced by M. Movshev in the preprint \cite{Mov1} and in his article \cite{Mov2}. 
The proofs sketched there use geometrical properties of Yang-Mills algebras and they are in general not self-contained and sometimes not complete.
We hope that our approach will contribute to the understanding of the subject. 
In fact, we think that the algebraic point of view is clearer. 
An example of this is the study of the relations between the homology of Yang-Mills algebra and the 
homology over the abelian Lie algebra of generators (e.g. the considerations before Proposition \ref{prop:annhyangmills}, 
which allow us to provide simpler proofs of that proposition, Propositions \ref{prop:annh3yangmills} and \ref{prop:wni} and 
Corollaries \ref{coro:anulacion} and \ref{coro:anulacion2}, among others). 

In spite of these good homological properties, the computation of the Hochschild cohomology is rather
difficult and technical. 
Concerning the zeroth and first Hochschild cohomology groups, we prove the following result.
\begin{theorem*}
If $n \geq 3$, the center of the Yang-Mills algebra $\YM(n)$ is $k$ and 
there is an isomorphism 
\[     HH^{1}(\YM(n)) \simeq k \oplus V(n)[2] \oplus \Lambda^{2} (V(n)[1]).     \]
\end{theorem*}

In particular, using the theorem, we may interpret the non-trivial infinitesimal symmetries of the Yang-Mills 
algebra as dilations, translations and rotations. 
Our interest in the first cohomology group comes from the fact that, in the noncommutative geometrical setting,
there is a one-to-one correspondence between the classes of noncommutative vector fields and derivations of an algebra, 
as it appears in \cite{Ko1} or \cite{KR1}.
We also describe explicitely the Lie bracket on the first cohomology group.

Making use of the Koszul property of these algebras, 
of a corollary of  Goodwillie's Theorem obtained by M. Vigu\'e-Poirrier, 
and of the fact that the graded Euler-Poincar\'e characteristic of the cyclic homology of a multigraded algebra is known (see \cite{I1}), 
we describe not only the Hilbert series of the other cohomology groups, but also the ones for the cyclic homology groups. 

Our main result may be formulated as follows:
\begin{theorem}
\label{teo:hhhcym>3}
If $n \geq 3$, then the Hilbert series for the Hochschild homology are given by 
\begin{align*}
   HH_{\bullet}(\YM(n))(t) &= 0, &\text{if $\bullet \geq 4$,}
   \\
   HH_{3}(\YM(n))(t) &= t^{4},
   \\
   HH_{2}(\YM(n))(t) &= \Big(\frac{n(n-1)}{2} + 1\Big) t^{4} + n t^{3},
   \\
   HH_{1}(\YM(n))(t) &= - \sum_{l \geq 1} \frac{\varphi(l)}{l} \log(1 - n t^{l} + n t^{3l} - t^{4l}) + (n(n-1) - 1) t^{4} + 2 n t^{3},
   \\
   HH_{0}(\YM(n))(t) &= - \sum_{l \geq 1} \frac{\varphi(l)}{l} \log(1 - n t^{l} + n t^{3l} - t^{4l}) 
   + \Big(\frac{n(n-1)}{2} - 1\Big) t^{4} + n t^{3} + 1,
\end{align*}
where $\varphi$ denotes the Euler function. 

Also, $HH^{\bullet}(\YM(n))(t) = 0$, if $\bullet \geq 4$ and $HH^{\bullet}(\YM(n))(t) = t^{-4} HH_{3-\bullet}(\YM(n))(t)$. 
\medskip

On the other hand, the Hilbert series for the cyclic homology are
\begin{align*}
   &HC_{4+2\bullet}(\YM(n))(t) = 1, &\text{if $\bullet \geq 0$,}
   \\
   &HC_{3+2\bullet}(\YM(n))(t) = 0, &\text{if $\bullet \geq 0$,}
   \\
   &HC_{2}(\YM(n))(t) = 1 + t^{4},
   \\
   &HC_{1}(\YM(n))(t) = \frac{n(n-1)}{2} t^{4} + n t^{3},
   \\
   &HC_{0}(\YM(n))(t) = - \sum_{l \geq 1} \frac{\varphi(l)}{l} \log(1 - n t^{l} + n t^{3l} - t^{4l}) + \Big(\frac{n(n-1)}{2} - 1\Big) t^{4} 
   + n t^{3} + 1.
\end{align*}
\end{theorem}

The key ingredient of the proof of this theorem is the analysis of the spectral sequence associated to the natural filtration 
on $\YM(n)$ provided by a Lie ideal $\tym(n)$ in $\ym(n)$ which is free as Lie algebra \cite{HS1}. 
Notice that since  $HH^{2}(\YM(n))$ and $HH^{3}(\YM(n))$ are not zero, deformations 
of $\YM(n)$ may exist, but up to the present we do not know whether they are obstructed or not. 

\bigskip

The contents of the article are as follows.
In Section \ref{sec:generalities} we recall the definition of Yang-Mills algebras $\YM(n)$ 
and some of their homological properties and we  
compute, using the Koszul complex for this algebra, the complete Hochschild homology of 
$\YM(2)$, recovering results by P. Nuss \cite{Nu1} and giving explicit bases. 

Section \ref{sec:Wn} is devoted to the study of the space of generators $W(n)$ of the free Lie algebra $\tym(n)$, which 
will play an important role in the computation of the Hochschild cohomology of the 
Yang-Mills algebras.
We provide a complete description of the homological properties of this space considered as both a left $S(V(n))$-module and a left $\YM(n)$-module 
(see Theorem \ref{teo:exacta}), which will be useful in the sequel to describe the corresponding 
homological properties of the enveloping algebra $\U(\tym(n))^{\mathrm{ad}}$ considered as a left $\YM(n)$-module with the adjoint action. 
Using the fact that $\tym(n)$ is free and nonabelian, we compute the zeroth Hochschild cohomology group of $\U(\ym(n))$. 

In Section \ref{sec:calculoh1tymn}, we study the cohomology of the Yang-Mills algebra $\ym(n)$ 
with coefficients in the augmentation ideal of $\U(\tym(n))^{\mathrm{ad}}$. 
The results obtained throughout this section lead to Theorem \ref{teo:homi}.
Its proof involves the analysis of a spectral sequence associated to the filtration on $\U(\tym(n))$ by powers of the augmentation ideal. 

The aim of Section \ref{sec:hh1} is the description of the outer derivations of the Yang-Mills algebra $\YM(n)$ (see Theorem \ref{teo:hh1}). 
This is done by using a spectral sequence associated to the filtration given by powers of the ideal generated by 
$\tym(n)$ in $\YM(n)$ and homological information obtained in Section \ref{sec:Wn}.

Finally, in Section \ref{sec:hhhc} we collect all previous results to prove our main result, Theorem \ref{teo:hhhcym>3}. 

Throughout this article $k$ will denote an algebraically closed field of characteristic zero and all unadorned tensor products $\otimes$ 
will be over the field $k$, \textit{i.e.} $\otimes = \otimes_{k}$. 
All morphisms will be $k$-linear homogeneous of degree zero, unless the contrary is stated. 

We would like to thank Jacques Alev, Michel Dubois-Violette and Jorge Vargas for useful comments and remarks.
We also thank the referee for his careful reading of the manuscript. 
We are indebted to Mariano Su\'arez-\'Alvarez for many suggestions and improvements. 

\section{Definition and homological properties}
\label{sec:generalities}

In this first section we fix notations and recall some elementary properties of the Yang-Mills algebras. 
As reference we suggest \cite{CD1} and \cite{HS1}.

\subsection{Generalities}

Let $n$ be a positive integer such that $n \geq 2$ and let $\f(n)$ be the free Lie algebra on $n$ generators $\{ x_{1}, \dots, x_{n} \}$.
This Lie algebra is provided with a locally finite dimensional $\NN$-grading.

Following \cite{CD1}, the quotient Lie algebra
\[     \ym(n) = \f(n)/\cl{\{ \sum_{i=1}^{n} [x_{i},[x_{i},x_{j}]] : 1 \leq j \leq n \}}     \]
is called the \emph{Yang-Mills algebra with $n$ generators}.

The $\NN$-grading of $\f(n)$ induces an $\NN$-grading of $\ym(n)$, which is also locally finite dimensional.
We denote by $\ym(n)_{j}$ the $j$-th homogeneous component and
\begin{equation}
\label{eq:gradym}
\ym(n)^{l} = \bigoplus_{j = 1}^{l} \ym(n)_{j}.
\end{equation}
The Lie ideal
\begin{equation}
\label{eq:tym}
   \tym(n) = \bigoplus_{j \geq 2} \ym(n)_{j}
\end{equation}
will be of considerable importance in the sequel.

The enveloping algebra $\U(\ym(n))$ will be denoted $\YM(n)$.
It is the \emph{(associative) Yang-Mills algebra on $n$ generators}.
If $V(n) = \text{span}_{k} \cl{\{ x_{1} , \dots , x_{n} \}}$, we see that
\[     \YM(n) \simeq T(V(n))/\cl{R(n)},     \]
where $T(V(n))$ denotes the tensor $k$-algebra on $V(n))$ and 
\begin{equation}
\label{eq:relym}
     R(n) = \mathrm{span}_{k} \cl{\{ \sum_{i=1}^{n} [x_{i},[x_{i},x_{j}]] : 1 \leq j \leq n \}} \subseteq V(n)^{\otimes 3}.
\end{equation}
We shall also consider the enveloping algebra of the Lie ideal $\tym(n)$, which will be denoted $\TYM(n)$.
Occasionally, we will omit the index $n$ from  the notation if it is clear from the context.

We shall make use of the previous grading on $\ym(n)$, it will be called the \emph{usual grading} of the Yang-Mills algebra $\ym(n)$. 
However, we would like to mention the \emph{special grading} of the Yang-Mills algebra $\ym(n)$,
for which it is a graded Lie algebra concentrated in even degrees with each homogeneous space $\ym(n)_{j}$ in degree $2j$.
These gradings induce respectively the \emph{usual grading} and the \emph{special grading} on the associative algebra
$\YM(n)$.
The last one corresponds to taking the graded enveloping algebra of the graded Lie algebra $\ym(n)$.
We will be mainly concerned with the usual grading and shall only briefly mention the special one. 

As noted in \cite{HS1}, the algebra $\ym(n)$ is nilpotent if $n=2$, in which case it is also finite dimensional
(see \cite{HS1}, Example 2.1).
When $n \geq 3$, $\ym(n)$ is neither finite dimensional nor nilpotent (see \cite{HS1}, Remark 3.13).
Also, the algebra $\YM(n)$ is a domain for any $n \in \NN$, since it is the enveloping algebra
of a Lie algebra.

There is an important collection of symmetries acting on the Yang-Mills algebra, which we now describe. 
The reader who does not want to work in full generality may simplify its attention to the case $k = \CC$, 
even though the arguments we use here are also valid for any algebraically closed field $k$ of characteristic zero (\textit{cf.}~\cite{Hum2}). 
The representation of the algebraic group $\SO(n)$ on $V(n)$ given by the standard action of matrices induces
a representation of the Lie algebra $\so(n)$.
Furthermore, given $j \in \NN$, $V(n)^{\otimes j}$ is a representation of $\SO(n)$, and then of $\so(n)$, with the diagonal action.
There is then an action by algebra automorphisms of $\SO(n)$ on $T(V(n))$, which induces in turn an action by derivations of $\so(n)$
on $T(V(n))$.
Both actions are homogeneous of degree $0$.

It is readily verified that these actions on $T(V(n))$ preserve the ideal $\cl{R(n)}$.
So, $\SO(n)$ acts by algebra automorphisms on $\YM(n)$ and $\so(n)$ acts by derivations.
The latter induces in turn an action by derivations of $\so(n)$ on $\ym(n)$.
As before, all these actions are homogeneous of degree $0$.

We summarize these facts in the following proposition.
\begin{proposition}
\label{prop:accionlieyangmills}
The standard action of $\SO(n)$ on $V(n)$ induces an action by automorphisms of graded algebras on $\YM(n)$ and
an action by 
derivations of the Lie algebra $\so(n)$ on $\YM(n)$.
\end{proposition}

Let $Y$ be a graded left $\YM(n)$-module provided with an action of $\so(n)$ which is homogeneous of degree $0$.
We shall say that the action of the Lie algebra $\so(n)$ is \emph{compatible} with the action of the Yang-Mills algebra if 
the structure morphism $\YM(n) \rightarrow \End_{k}(Y)$ is $\so(n)$-linear, where $\End_{k}(Y)$ is an $\so(n)$-module with 
the action induced by that of $Y$, or equivalently, 
\[     x\cdot (z y) = (x\cdot z) y + z (x\cdot y)     \]
for all $x \in \so(n)$, $z \in \YM(n)$ and $y \in Y$.
The dot $\cdot$ indicates both the action of $\so(n)$ on $\YM(n)$ and on $Y$. 
Examples of such modules are $\ym(n)$ with the adjoint action and also the symmetric algebra $S(\ym(n))$ with the induced action.

Since it will be useful to combine the complete collection of symmetries that are available, from now on we shall consider
the category of graded left modules over the graded Yang-Mills algebra $\YM(n)$ with the usual grading
provided with a compatible action of $\so(n)$
and call them \emph{equivariant} left $\YM(n)$-modules, without making explicit reference to the grading or the action of $\so(n)$,
unless we write the contrary.
Furthermore, a homogeneous left $\YM(n)$-linear morphism of degree $0$ which is $\so(n)$-equivariant between two
equivariant left  $\YM(n)$-modules will also be called \emph{equivariant}.
The previous definitions apply as well for the category of $\YM(n)$-bimodules and right $\YM(n)$-modules.

\begin{remark}
\label{rem:svn}
Since $V(n)$, considered as an abelian Lie algebra, is a quotient of $\ym(n)$ by the Lie ideal $\tym(n)$, any graded left 
module $Y$ over $S(V(n))$ becomes a graded left module over $\YM(n)$.
If $Y$ is provided with a homogeneous action of $\so(n)$ of degree $0$, such that
\[     x\cdot (z y) = (x\cdot z) y + z (x\cdot y),     \]
for all $x \in \so(n)$, $z \in V(n)$ and $y \in Y$, then it is an equivariant left $\YM(n)$-module.
In this case, we shall also say that $Y$ is an \emph{equivariant} left  $S(V(n))$-module.
The same applies to right modules.
\qed
\end{remark}

Since $\YM(n)$ and $SV(n)$ are universal enveloping algebras of Lie algebras, any left module $Y$ over $\YM(n)$ (resp. $S(V(n))$) is also a right module over $\YM(n)$ (resp. $S(V(n))$) in the usual manner. 

An example of an equivariant $\YM(n)$-module is $S(V(n))$, where the generators
$\{ x_{1} , \dots, x_{n} \}$ of $\ym(n)$ act by multiplication on $S(V(n))$, the action of $\tym(n)$ is trivial 
and the action of $\so(n)$ is induced by the standard action of $\so(n)$
on $V(n)$.

\subsection{Homology and cohomology}
\label{subsec:homandcohom}

We recall (see \cite{Ber1}) that if $A$ is an $N$-homogeneous $k$-algebra ($N \geq 2$) given by $A = TV/\cl{R}$, where $R \subseteq V^{\otimes N}$,
the $N$-homogeneous dual algebra $A^{!}$ is defined as the quotient $T(V^{*})/\cl{R^{\perp}}$,
where $R^{\perp} \subseteq (V^{*})^{\otimes N} \simeq (V^{\otimes N})^{*}$ is the annihilator of $R$.
In this case, the \emph{left Koszul $N$-complex} of $A$ is
\begin{equation}
\label{eq:NcomplejoKoszul}
    \dots \overset{b_{n+1}}{\rightarrow} A \otimes (A_{n}^{!})^{*} \overset{b_{n}}{\rightarrow}
    \dots \overset{b_{3}}{\rightarrow} A \otimes (A_{2}^{!})^{*} \overset{b_{2}}{\rightarrow} A \otimes
    (A_{1}^{!})^{*} \overset{b_{1}}{\rightarrow} A \rightarrow 0,
\end{equation}
where $(A_{i}^{!})^{*} \subseteq V^{\otimes i}$ and the differential $b_{i}$ is the restriction of multiplication
$a \otimes (e_{1} \otimes \dots \otimes e_{i}) \mapsto a e_{1} \otimes \dots \otimes e_{i}$.
Notice that the differentials of the previous $N$-complex are homogeneous of degree $0$.

From the $N$-complex \eqref{eq:NcomplejoKoszul} one can obtain complexes $C_{p,r}(A)$, for $0 \leq r \leq N-2$ and $r+1 \leq p \leq N-1$,
given by
\begin{equation}
\label{eq:Cpr}
    \dots \overset{b^{N-p}}{\rightarrow} A \otimes (A_{N+r}^{!})^{*}
    \overset{b^{p}}{\rightarrow} A \otimes (A_{N-p+r}^{!})^{*} \overset{b^{N-p}}{\rightarrow}
    A \otimes (A_{r}^{!})^{*} \overset{b^{p}}{\rightarrow} 0.
\end{equation}

Following \cite{Ber1} and \cite{BDVW1}, the complex $C_{N-1,0}(A)$ is called the \emph{left Koszul complex} of $A$ and
the algebra $A$ is called \emph{left Koszul} if this complex is acyclic in positive degrees. 
There are analogous definitions of \emph{right Koszul complex} and \emph{bimodule Koszul complex} of $A$, and hence of \emph{right Koszul} and $\emph{Koszul}$ algebra. 
Since the three definitions are equivalent (\textit{cf.} \cite{BDVW1}, Prop. 3, and \cite{BM1}, Thm. 4.4), we shall call them Koszul complex or Koszul algebra, respectively. 

From its very definition the Yang-Mills algebra is a cubic homogeneous algebra.

The following proposition describes its dual algebra.
\begin{proposition}
Let $\YM(n) = T(V(n))/\cl{R(n)}$ be the Yang-Mills algebra with set of generators given by $\{ x_{1}, \dots, x_{n} \}$.
If we denote by $\B^{*} = \{ x_{1}^{*}, \dots, x_{n}^{*} \}$ the dual basis of $V(n)^{*}$,
then the homogeneous components of $\YM(n)^{!}$  are
\[
\xymatrix@R-20pt
{
   \YM(n)_{0}^{!} = \CC 1,
   &
   \YM(n)_{2}^{!} = \bigoplus_{i,j = 1}^{n} \CC x_{i}^{*} x_{j}^{*},
   &
   \YM(n)_{4}^{!} = \CC z^{2},
   \\
   \YM(n)_{1}^{!} = V^{*},
   &
   \YM(n)_{3}^{!} = \bigoplus_{i=1}^{n} \CC x_{l}^{*} z,
   &
   \YM(n)_{i}^{!} = 0,
}
\]
for all $i > 4$ and $z = \sum_{i=1}^{n} (x_{i}^{*})^{2}$. The element $z$ is central in $\YM(n)^{!}$.
\end{proposition}
\begin{proof}
See \cite{CD1}, Prop. 1.
\end{proof}

From the proposition we easily obtain the following isomorphisms, which are necessary for the explicit description of 
the differentials of the Koszul complex of the Yang-Mills algebra:
\[
\xymatrix@R-20pt
{
   (\YM(n)_{1}^{!})^{*} \simeq V(n),
   &
   (\YM(n)_{2}^{!})^{*} \simeq V(n)^{\otimes 2},
   \\
   (\YM(n)_{3}^{!})^{*} \simeq R(n),
   &
   (\YM(n)_{4}^{!})^{*} \simeq (V(n) \otimes R) \cap (R \otimes V(n)).
}
\]


Furthermore, the following is true. 
\begin{proposition}
\label{prop:koszulyangmills}
The Yang-Mills algebra is Koszul of global dimension $3$.
\end{proposition}
\begin{proof}
See \cite{CD1}, Thm. 1.
\end{proof}

This proposition tells us that the Koszul complex of $\YM(n)$ is a projective resolution of $k$. 
We shall now present its proper form in the category of graded left $\YM(n)$-modules,
\textit{i.e.} we shall give the explicit description of the minimal projective resolution of graded left-$\YM(n)$-modules. 
The modules of the resolution are also provided with a compatible action of $\so(n)$ and the morphisms are equivariant. 
We consider the following complex
\begin{equation}
\label{eq:complejoKoszulcompleto}
    0 \rightarrow \YM(n)[-4]  \overset{b'_{3}}{\rightarrow} \YM(n) \otimes V(n)[-2]
    \overset{b'_{2}}{\rightarrow} \YM(n) \otimes V(n)
    \overset{b'_{1}}{\rightarrow} \YM(n) \overset{b'_{0}}{\rightarrow} k \rightarrow 0,
\end{equation}
with differential
\[
\xymatrix@R-20pt
{
   b'_{3}(z) = \sum_{i=1}^{n} z x_{i} \otimes x_{i},
   &
   b'_{2}(z \otimes x_{i}) = \sum_{j=1}^{n} (z x_{j}^{2} \otimes x_{i} - 2 z x_{j} x_{i} \otimes x_{j} + z x_{i} x_{j} \otimes x_{j}),
   \\
   b'_{1}(z \otimes x_{i}) = z x_{i},
   &
   b'_{0}(z) = \epsilon_{\ym(n)}(z),
}
\]
where $\epsilon_{\ym(n)}$ is the augmentation of the algebra $\YM(n)$.


Let $Y$ be an equivariant left $\YM(n)$-module.
When we apply the functor $\Hom_{\YM(n)}(\place , Y)$ to the resolution \eqref{eq:complejoKoszulcompleto},
we obtain the complex, which we will denote by $(C^{\bullet} (\YM(n),Y),d)$,
\begin{equation}
\label{eq:complejocohomologiayangmills}
    0 \longrightarrow Y \overset{d^{1}}{\longrightarrow} Y \otimes V(n)[2] \overset{d^{2}}{\longrightarrow} Y \otimes V(n)[4]
    \overset{d^{3}}{\longrightarrow} Y[4] \longrightarrow 0,
\end{equation}
after having used the equivariant isomorphisms
$\Hom_{\YM(n)} (\YM(n)[j] , Y) \simeq Y[2-j]$ and
\begin{align*}
   \Hom_{\YM(n)} (\YM(n) \otimes V(n)[j] , Y) &\overset{\simeq}{\longrightarrow} Y \otimes V(n)[-j]
   \\
   f &\mapsto \sum_{i=1}^{n} f(1 \otimes x_{i}) \otimes x_{i},
\end{align*}
where $j \in \ZZ$.
The differentials are given by
\begin{align*}
   d^{3}(y \otimes x_{i}) &= x_{i} y,
   \hskip 2.55cm
   d^{1}(y) = \sum_{i=1}^{n} x_{i} y \otimes x_{i},
   \\
   d^{2}(y \otimes x_{i}) &= \sum_{j=1}^{n} (x_{j}^{2} y \otimes x_{i} + x_{j} x_{i} y \otimes x_{j} - 2 x_{i} x_{j} y \otimes x_{j}).
\end{align*}

Analogously, let $Y$ be an equivariant right $\YM(n)$-module.
If we apply the functor $Y \otimes_{\YM(n)} (\place)$ to the resolution \eqref{eq:complejoKoszulcompleto}
and we use the equivariant right $\YM(n)$-linear isomorphisms
$Y \otimes_{\YM(n)} \YM(n)[d] \simeq Y[d]$ and
\begin{align*}
   Y \otimes_{\YM(n)} \YM(n) \otimes V(n)[d] &\overset{\simeq}{\longrightarrow} Y \otimes V(n)[d]
   \\
   y \otimes_{\YM(n)} 1 \otimes x_{i} &\mapsto y \otimes x_{i},
\end{align*}
where $d \in \ZZ$, we obtain the complex, which we shall denote by $(C_{\bullet} (\YM(n),Y),d)$,
\begin{equation}
\label{eq:complejohomologiayangmills}
    0 \longrightarrow Y[-4] \overset{d_{3}}{\longrightarrow} Y \otimes V(n)[-2] \overset{d_{2}}{\longrightarrow} Y \otimes V(n)
    \overset{d_{1}}{\longrightarrow} Y \longrightarrow 0,
\end{equation}
with differentials
\begin{align*}
   d_{1}(y \otimes x_{i}) &= y x_{i},
   \hskip 2.55cm
   d_{3}(y) = \sum_{i=1}^{n} y x_{i} \otimes x_{i},
   \\
   d_{2}(y \otimes x_{i}) &= \sum_{j=1}^{n} (y x_{j}^{2} \otimes x_{i} + y x_{i} x_{j} \otimes x_{j} - 2 y x_{j} x_{i} \otimes x_{j}).
\end{align*}

Taking into account that $V(n)$ is concentrated in degree $1$,
$Y \otimes V(n)[j] \simeq (Y \otimes V(n))[j]$, for all $j \in \ZZ$.
Comparing \eqref{eq:complejocohomologiayangmills} and \eqref{eq:complejohomologiayangmills},
we see that $(C^{\bullet} (\YM(n),Y),d)$ and $(C_{\bullet} (\YM(n),Y),d')[4]$ coincide, where
$(d')_{\bullet} = (-1)^{\bullet} d_{\bullet}$.
These complexes compute $\Ext_{\text{\begin{tiny}$\YM(n)$\end{tiny}}}^{\bullet}(k,Y)$ and $\Tor^{\text{\begin{tiny}$\YM(n)$\end{tiny}}}_{\bullet}(Y,k)$, respectively.
The natural isomorphisms $\Ext_{\text{\begin{tiny}$\YM(n)$\end{tiny}}}^{\bullet}(k,Y) \simeq H^{\bullet} (\ym(n),Y)$
and $\Tor^{\text{\begin{tiny}$\YM(n)$\end{tiny}}}_{\bullet}(Y,k) \simeq H_{\bullet} (\ym(n),Y)$ (see \cite{Wei1}, Coro. 7.3.6) 
tell us that 
\[     H^{i} (\ym(n) , Y) \simeq H_{3 - i} (\ym(n) , Y)[4],     \]
for $0 \leq i \leq 3$.

Just to state notation, if $Z$ is a graded $k$-vector space, 
we denote by $Z(t) = \sum_{n \in \ZZ} \dim(Z_{n}) t^{n} \in \ZZ[[t^{-1},t]]$ its Hilbert series. 

Of course, since the global dimension of the Yang-Mills algebra is $3$, $H^{i} (\ym(n) , Y)$ and $H_{i} (\ym(n) , Y)$ vanish for $i > 3$.
Both Hilbert series coincide up to a shift
\[     H^{i} (\ym(n) , Y)(t) = t^{-4} H_{3 - i} (\ym(n) , Y)(t).     \]
This relation between homology and cohomology is usually referred to as \emph{Poincar\'e duality}, because of its resemblance 
to the case of closed oriented manifolds.

We can state the previous results as follows.
\begin{proposition}(see \cite{CD1}, Eq. (1.15))
\label{prop:(co)homologiaYangMills}
Let $Y$ be a left $\YM(n)$-module, which will be considered also as a right $\YM(n)$-module in the usual manner.
The cohomology of $\ym(n)$ with coefficients in $Y$ equals  the cohomology of
the complex \eqref{eq:complejocohomologiayangmills}, and the homology of the $\ym(n)$ with coefficients on $Y$ equals the
homology of the complex \eqref{eq:complejohomologiayangmills}.
We have that $(C_{\bullet} (\YM(n),Y),d')[4] = (C^{\bullet} (\YM(n),Y),d)$, where $(d')_{\bullet} = (-1)^{\bullet} d_{\bullet}$,
so in particular $H^{i} (\ym(n) , Y) \simeq H_{3 - i} (\ym(n) , Y)[4]$, for $0 \leq i \leq 3$.
\end{proposition}

We want to stress that the Chevalley-Eilenberg resolution $(C_{\bullet}(\ym(n)),\delta_{\bullet})$ of the equivariant left $\YM(n)$-module 
$k$ is also a projective resolution of graded left-$\YM(n)$-modules. 
The modules of the resolution are also provided with a compatible action of $\so(n)$ and the morphisms are equivariant. 
If $Y$ is an equivariant left (resp. right) $\YM(n)$-module,
we see that the morphisms of the Chevalley-Eilenberg complex $(C^{\bullet}(\ym(n),Y),d^{\bullet}_{\mathrm{CE}})$
(resp. $(C_{\bullet}(\ym(n),Y),d_{\bullet}^{\mathrm{CE}})$) 
for the cohomology (resp. homology) of $\ym(n)$ with coefficients in $Y$ are $\so(n)$-linear homogeneous of degree $0$.

We can easily check that the following diagram gives a morphism from the Koszul resolution to the Chevalley-Eilenberg resolution of $k$:
\begin{scriptsize}
\[
\xymatrix@C-10pt
{
\dots
\ar[r]
&
\YM \otimes \wedge^{4} \ym
\ar[r]^{\delta_{4}}
&
\YM \otimes \wedge^{3} \ym
\ar[r]^{\delta_{3}}
&
\YM \otimes \wedge^{2} \ym
\ar[r]^{\delta_{2}}
&
\YM \otimes \ym
\ar[r]^(.6){\delta_{1}}
&
\YM
\ar[r]^(0.6){\delta_{0}}
&
k
\ar[r]
\ar@{=}[d]
&
0
\\
\dots
\ar[r]
&
0
\ar[r]
\ar[u]
&
\YM[-4]
\ar[r]^{b'_{3}}
\ar[u]^{\theta}
&
\YM \otimes V[-2]
\ar[r]^{b'_{2}}
\ar[u]^{\eta}
&
\YM \otimes V
\ar[r]^(.6){b'_{1}}
\ar[u]^{\id_{\YM} \otimes \mathrm{inc}}
&
\YM
\ar[r]^(0.6){b'_{0}}
\ar@{=}[u]
&
k
\ar[r]
\ar@{=}[u]
&
0
}
\]
\end{scriptsize}
with vertical maps given by 
\begin{align}
   \label{eq:eta}
   \eta(z \otimes x_{i}) &= \sum_{j = 1}^{n} (z x_{j} \otimes x_{j} \wedge x_{i} + z \otimes x_{j} \wedge [x_{j} , x_{i}]),
   \\
   \label{eq:theta}
   \theta(z) &= \frac{1}{2} \sum_{i,j = 1}^{n} z \otimes x_{i} \wedge x_{j} \wedge [x_{j} , x_{i}].
\end{align}
It is clear that these morphisms are equivariant.

Given a left $\ym(n)$-module $Y$, it can be considered as a $\YM(n)$-bimodule, which we denote by $Y_{\epsilon_{\ym(n)}}$,
where the action on the right is given by the augmentation $\epsilon_{\ym(n)}$ of $\YM(n)$.
It is known that there are natural isomorphisms of the form
$H^{\bullet}(\ym(n),Y) \simeq H^{\bullet}(\YM(n),Y_{\epsilon_{\ym(n)}})$ and
$H_{\bullet}(\ym(n),Y) \simeq H_{\bullet}(\YM(n),Y_{\epsilon_{\ym(n)}})$ (see \cite{CE1}, Thm. X.2.1).

Conversely, if $Y$ is a $\YM(n)$-bimodule, it can be considered as a (left or right) $\ym(n)$-module via the adjoint action,
denoted by $Y^{\mathrm{ad}}$.
There are natural isomorphisms $H^{\bullet}(\YM(n),Y) \simeq H^{\bullet}(\ym(n),Y^{\mathrm{ad}})$
and $H_{\bullet}(\YM(n),Y) \simeq H_{\bullet}(\ym(n),Y^{\mathrm{ad}})$ (see \cite{CE1}, Thm. XIII.7.1).

By the Poincar\'e-Birkhoff-Witt Theorem, there is a left $\YM(n)$-linear isomorphism
given by symmetrization $S(\ym(n)) \simeq \YM(n)^{\mathrm{ad}}$ (see \cite{Dix1}, 2.4.5, Prop. 2.4.10) and one can check that it is equivariant.
This implies that $HH^{\bullet}(\YM(n)) \simeq H^{\bullet} (\ym(n),S(\ym(n)))$ and $HH_{\bullet}(\YM(n)) \simeq H_{\bullet} (\ym(n),S(\ym(n)))$.
We point out that these isomorphisms are $\so(n)$-linear.

We recall that if $A$ is an $\NN_{0}$-graded associative algebra, $X$ is a $\ZZ$-graded right $A$-module 
and $Y$ is a $\ZZ$-graded left (resp. right) $A$-module, 
then the homology groups $\Tor^{\text{\begin{tiny}$A$\end{tiny}}}_{p} (X,Y)$ (resp. $\Ext_{\text{\begin{tiny}$A$\end{tiny}}}^{p} (X,Y)$, for $X$ with a projective resolution of finitely generated modules) 
are in fact graded vector spaces with respect to the internal grading 
and we denote by $\Tor^{\text{\begin{tiny}$A$\end{tiny}}}_{p,q} (X,Y)$ (resp. $\Ext_{\text{\begin{tiny}$A$\end{tiny}}}^{p,q} (X,Y)$) 
its homogeneous component of internal degree $q \in \ZZ$. 
We apply the same notation for other homology groups, e.g. Lie algebra (co)homology and Hochschild (co)homology groups. 

\begin{remark}
We remark that if $A$ is a $\NN_{0}$-graded associative algebra, $X$ and $Y$ are $\ZZ$-graded right $A$-modules such that $X$ is finitely generated, 
then $\mathrm{Hom}_{A}(X,Y) = \mathcal{H}om_{A}(X,Y)$, where, as always in this article, the first member denotes the space of morphisms of $A$-modules 
and the second one is the graded vector space expanded by homogeneous morphisms, \textit{i.e.} 
$\mathcal{H}om_{A}(X,Y) = \oplus_{n \in \ZZ}  \mathrm{hom}_{A} (X,Y[n])$, for $\mathrm{hom}_{A} (X,Y)$ the space 
of $A$-linear homogeneous morphisms of degree zero (see \cite{NO1}, Cor. 2.4.4). 
This explains the internal grading of the $\Ext$ groups considered before. 

On the other hand, if the algebra $A$ is $N$-homogeneous Koszul, the minimality of the bimodule Koszul complex $K_{\bullet}(A)$ of $A$ 
tells us that that there exists an isomorphism of complexes of graded $A$-bimodules $C_{\bullet}(A) \simeq K_{\bullet}(A) \oplus H_{\bullet}(A)$, 
where $H_{\bullet}(A)$ is an homotopically trivial complex. 
Therefore, 
\[     \mathrm{Hom}_{A^{e}}(C_{\bullet}(A),M) \simeq \mathrm{Hom}_{A^{e}}(K_{\bullet}(A),M) \oplus \mathrm{Hom}_{A^{e}}(H_{\bullet}(A),M)     \]
(resp.
\[     \mathcal{H}om_{A^{e}}(C_{\bullet}(A),M) \simeq \mathcal{H}om_{A^{e}}(K_{\bullet}(A),M) \oplus \mathcal{H}om_{A^{e}}(H_{\bullet}(A),M)).     \]
Since $H_{\bullet}(A)$ is an acyclic complex of projective graded $A$-bimodules, both $\mathcal{H}om_{A^{e}}(H_{\bullet}(A),M)$ and 
$\mathrm{Hom}_{A^{e}}(H_{\bullet}(A),M)$ have zero cohomology. 
Moreover, taking into account that the bimodule Koszul complex $K_{\bullet}(A)$ is made of finitely generated $A$-bimodules 
(for the $k$-vector spaces $(A_{\bullet}^{!})^{*}$ are finite dimensional $k$-vector spaces), 
we obtain that $\mathrm{Hom}_{A^{e}}(K_{\bullet}(A),M) = \mathcal{H}om_{A^{e}}(K_{\bullet}(A),M)$. 
In consequence, we see that the plain Hochschild cohomology coincides with the graded Hochschild cohomology for a Koszul algebra 
(in fact it is sufficient to have a graded projective resolution of $A$ given by finitely generated $A$-bimodules). 
Since we shall be dealing with this kind of algebras, we are not going to make any distinction between both cohomology theories. 
\qed
\end{remark}

By Proposition \ref{prop:(co)homologiaYangMills}, $HH^{\bullet} (\YM(n)) = HH_{3 - \bullet} (\YM(n))[4]$, for $0 \leq \bullet \leq 3$, 
and $HH^{\bullet} (\YM(n)) = HH_{\bullet} (\YM(n)) = 0$, for $\bullet > 3$, so one needs to compute either cohomology or homology groups.

Let us now focus on the case that $Y$ is a right $S(V(n))$-module, and by Remark \ref{rem:svn} also a right $\YM(n)$-module.
The Chevalley-Eilenberg complex $(C_{\bullet}(V(n),Y),d_{\bullet}^{\mathrm{CE}})$ is provided with a homogeneous
$k$-linear morphism of degree $2$ of the form
\begin{align*}
   h_{p} : C_{p}(V(n),Y) &\rightarrow C_{p +1}(V(n),Y)
   \\
   y \otimes x_{i_{1}} \wedge \dots \wedge x_{i_{p}} &\mapsto \sum_{j=1}^{n} y x_{j} \otimes x_{j} \wedge x_{i_{1}} \wedge \dots \wedge x_{i_{p}},
\end{align*}
such that $d_{p+1}^{\mathrm{CE}} \circ h_{p} + h_{p-1} \circ d_{p}^{\mathrm{CE}} = q . \id_{C_{p}(V(n),Y)}$, for all $p$,
where $q = \sum_{i=1}^{n} x_{i}^{2} \in S(V(n))$.
Hence, $h$ is a homotopy between the zero morphism and the one given by multiplication by $q$ and in particular
$d_{p}^{\mathrm{CE}} \circ h_{p-1} \circ d_{p}^{\mathrm{CE}} = q . d_{p}$.

Moreover, if we define for each $p$ such that $0 \leq p \leq n$ the homogeneous $k$-linear isomorphism of degree $n-2p$ given by
\begin{align*}
   i_{p} : C_{p}(V(n),Y) &\rightarrow C_{n-p}(V(n),Y)
   \\
   y \otimes x_{i_{1}} \wedge \dots \wedge x_{i_{p}} &\mapsto (-1)^{i_{1} + \dots + i_{p} + p} y \otimes x_{j_1} \wedge \dots \wedge x_{j_{n-p}},
\end{align*}
where $i_{1} < \dots < i_{p}$, $j_{1} < \dots < j_{n-p}$ and $\{i_1, \dots ,i_p \} \cup \{j_1, \dots ,j_{n-p} \}= \{1, \dots, n \}$, 
we obtain that $i_{p-1} \circ d_{p}^{\mathrm{CE}} = h_{n-p} \circ i_{p}$.
So $h_{\bullet}$ essentially coincides with the differential $d^{\mathrm{CE}}_{n-\bullet}$,
when viewing the $i_{\bullet}$ as an identification.
Notice that $i_{p}^{-1} = (-1)^{n(n-1)/2} i_{n-p}$, for all $0 \leq p \leq n$.

We may thus provide an alternative description of the complex $C_{\bullet}(\YM(n),Y)$ as follows.
First, it is direct to check that $C_{0}(\YM(n),Y) = C_{0}(V(n),Y)$ and $C_{1}(\YM(n),Y) = C_{1}(V(n),Y)$.
On the other hand, the maps $i_{0}$ and $i_{1}$ give the isomorphisms $C_{3}(\YM(n),Y) \overset{\simeq}{\rightarrow} C_{n}(V(n),Y)$ and
$C_{2}(\YM(n),Y) \overset{\simeq}{\rightarrow} C_{n-1}(V(n),Y)$, respectively.
Furthermore, it is easily verified that $d_{1} = d_{1}^{\mathrm{CE}}$, $d_{3} = i_{1}^{-1} \circ d_{n}^{\mathrm{CE}} \circ i_{0}$
and $d_{2} = d_{2}^{\mathrm{CE}} \circ h_{1} = i_{1}^{-1} \circ h_{n-2} \circ d_{n-1}^{\mathrm{CE}} \circ i_{1}$.
As a consequence, $H_{3}(\ym(n),Y) \simeq H_{n}(V(n),Y)$.

The following proposition is a generalization of Prop. 14 and 15 in \cite{Mov1}.
\begin{proposition}
\label{prop:annhyangmills}
Let $Y$ be a right $S(V(n))$-module (and by Remark \ref{rem:svn} also a right $\YM(n)$-module).
There is an isomorphism $H_{3}(\ym(n),Y) \simeq H_{n}(V(n),Y)$.
Moreover, if the element $q = \sum_{i=1}^{n} x_{i}^{2} \in S(V(n))$ is a nonzerodivisor on $Y$
there is also an isomorphism $H_{n-1}(V(n),Y) \simeq H_{2}(\ym(n),Y)$.
\end{proposition}
\begin{proof}
The first part of the proposition has been already proved.

Suppose that $q$ is a nonzerodivisor on $Y$.
Since $d_{3} = i_{1}^{-1} \circ d_{n}^{\mathrm{CE}} \circ i_{0}$, it is direct to check that
$\mathrm{Im}(d_{3}) = i_{1}^{-1}(\mathrm{Im}(d_{n}^{\mathrm{CE}}))$.

We shall prove that $\Ker(d_{2}) = i_{1}^{-1}(\Ker(d_{n-1}^{\mathrm{CE}}))$.
The equality $d_{2} = i_{1}^{-1} \circ h_{n-2} \circ d_{n-1}^{\mathrm{CE}} \circ i_{1}$ yields that
$\Ker(d_{2}) \supseteq i_{1}^{-1}(\Ker(d_{n-1}^{\mathrm{CE}}))$.
Let us prove the other inclusion.
The previous identity implies that
$z \in \Ker(d_{2})$ if and only if $i_{1}(z) \in \Ker(h_{n-2} \circ d_{n-1}^{\mathrm{CE}})$.
Hence, for an element $z \in \Ker(d_{2})$ we have that $h_{n-2} \circ d_{n-1}^{\mathrm{CE}} \circ i_{1} (z) = 0$, so
\[     0 = d_{n-1}^{\mathrm{CE}} \circ h_{n-2} \circ d_{n-1}^{\mathrm{CE}} \circ i_{1} (z) = q . d_{n-1}^{\mathrm{CE}} \circ i_{1} (z),     \]
and this implies that $d_{n-1}^{\mathrm{CE}} \circ i_{1} (z) = 0$, for $q$ is a nonzerodivisor on $Y$.
This proves that $z \in i_{1}^{-1}(\Ker(d_{n-1}^{\mathrm{CE}}))$ and thus the other inclusion.

Since $\Ker(d_{2}) = i_{1}^{-1}(\Ker(d_{n-1}^{\mathrm{CE}}))$,
$\mathrm{Im}(d_{3}) = i_{1}^{-1}(\mathrm{Im}(d_{n}^{\mathrm{CE}}))$ and $i_{1}$ is an isomorphism,
the proposition follows.
\end{proof}

\section{\texorpdfstring {The module $W(n)$}
        {The module W(n)}}
\label{sec:Wn}

In this section we shall study the graded vector space $W(n)$ of generators of the free Lie algebra $\tym(n)$.

\subsection{Generalities}

In \cite{HS1}, we proved that the Lie ideal $\tym(n)$ given in \eqref{eq:tym}, when considered with the special grading inherited by that of $\ym(n)$, is concentrated in even degrees strictly greater than $2$
and that it is itself a graded free Lie algebra:
it is isomorphic as graded Lie algebra to the graded free Lie
algebra on a graded vector space $W(n)$ (see \cite{HS1}, Theorem 3.12). 
The previous grading for $W(n)$ is called \emph{special}, but we will not make use of it in this article. 

Of course, when considering $\tym(n)$ with the usual grading it is also a free Lie algebra and its space of generators $W(n)$ 
is provided with the induced grading, called \emph{usual}.

Since the Lie algebra $\tym(n)$ is free on $W(n)$, the morphism $W(n) \rightarrow \tym(n)/[\tym(n),\tym(n)]$ given by composing the inclusion and the canonical projection is an isomorphism. 
Furthermore, since $\tym(n)$ is a Lie ideal of $\ym(n)$, $\tym(n)/[\tym(n),\tym(n)]$ has an action of $\ym(n)$
induced by the adjoint action of $\ym(n)$, such that $\tym(n)$ acts trivially.
Hence the quotient $\tym(n)/[\tym(n),\tym(n)]$ becomes a $\ym(n)/\tym(n)$-module, \textit{i.e.} a $V(n)$-module,
if we identify $V(n)$ with the abelian Lie algebra of dimension $n$.
The $S(V(n))$-module $\tym(n)/[\tym(n),\tym(n)]$ is graded.

On the other hand, since the action of $\so(n)$ on $\ym(n)$ is homogeneous of degree $0$,
it preserves the Lie ideals $\tym(n)$ and $[\tym(n),\tym(n)]$,
so it induces a compatible action on $\tym(n)/[\tym(n),\tym(n)]$.
Using the isomorphism $W(n) \overset{\sim}{\rightarrow} \tym(n)/[\tym(n),\tym(n)]$, $W(n)$ becomes a graded $S(V(n))$-module with a
compatible action of $\so(n)$ and hence an equivariant left $S(V(n))$-module and hence a $\YM(n)$-module, such that $\tym(n)$ acts trivially.


From the previous discussion we obtain an action of $V(n)$ on $W(n)$, which we shall denote by $x_{i}.w$, such that
\begin{equation}
\label{eq:accion}
     x_{i}\cdot w = - [x_{i},w] + \sum_{l \in L} [v_{i,l},v_{i,l}'],
\end{equation}
for $L$ a set of indices and some $v_{i,l}, v_{i,l}' \in \tym(n)$.

The Hilbert series of the Yang-Mills algebra $\YM(n)$ was computed in \cite{CD1}, Corollary 3, to be
\[     \YM(n)(t) = \frac{1}{(1-t^{2})(1-n t + t^{2})}.     \]
In \cite{HS1}, Proposition 3.14, we found the Hilbert series of $W(n)$ for the usual grading 
\[     W(n)(t) = \frac{(1-t)^{n}-1+n t -n t^{3} +t^{4}}{(1-t)^{n}}.     \]

If $n = 2$, it is easily checked from the previous formula that $W(2)$ is one dimensional and concentrated in degree $2$. 
Moreover, it may be considered as the graded $k$-vector space spanned by $z = [x_{1},x_{2}]$, 
and it is provided with the trivial action of $S(V(2))$ and $\so(2)$. 
Also, we see that $\tym(2) \simeq k[z]$. 

The previous considerations may be summarized as follows.
\begin{proposition}
\label{prop:W(2)}
If $k$ denotes the trivial equivariant $S(V(2))$-module, then $W(2) \simeq k[-2]$ as equivariant $S(V(2))$-modules (for the usual grading), 
and it is spanned by $[x_{1},x_{2}]$.
\end{proposition}

On the contrary, if $n \geq 3$, the Hilbert series of $W(n)$ tells us that it is infinite dimensional, 
which implies that $\tym(n)$ is a free Lie algebra with an infinite number of generators. 
We shall present a set of generators of $W(n)$ as $S(V(n))$-module in Corollary \ref{coro:W(n)}. 

Taking into account that $\tym(n) = \f(W(n))$, $\tym(n)$ is the Lie subalgebra generated by $W(n)$ inside $\mathrm{Lie}(T(W(n)))$, 
so we may consider another grading on $\tym(n)$, which we call the \emph{internal weight}, given by forgetting the grading of $W(n)$ but regarding the grading given by the tensor algebra. 
In other words, using that $T(W(n)) = \bigoplus_{p \in \NN_{0}} W(n)^{\otimes p}$, we may write
\[     \tym(n) = \bigoplus_{p \in \NN} \tym(n)^{p},     \]
for $\tym(n)^{p} = \tym(n) \cap W(n)^{\otimes p}$.
When $z \in \tym(n)^{p}$ we shall say that $z$ has internal weight $p$ (not to be confused with the weight of the $\so(n)$-modules).  
If we denote by $\rho^{p}_{i}(w)$ the projection of $[x_{i},w] \in \tym(n)$ in the $p$-th component $\tym(n)^{p}$,
we can write
\begin{equation}
\label{eq:acciontotal}
     [x_{i},w] = x_{i}.w + \sum_{p \geq 2} \rho_{i}^{p}(w).
\end{equation}
Notice that the sum is finite.

\begin{proposition}
\label{prop:W(n)}
The graded vector space $W(n)$ is a graded $S(V(n))$-module, when both are considered with the usual grading.
Moreover, since the homogeneous element $q = \sum_{i=1}^{n} x_{i}^{2} \in S(V(n))$ acts by $0$, $W(n)$ is a graded $S(V(n))/\cl{q}$-module.
\end{proposition}
\begin{proof}
The first part has been already proved.
We proceed with the second one.
From \eqref{eq:accion}, it suffices to prove that
\begin{equation}
\label{eq:qnulo}
     \sum_{i=1}^{n} [x_{i},[x_{i},w]] \in [\tym(n),\tym(n)],
\end{equation}
for any homogeneous element $w \in W(n)$. 
In order to do this we shall do induction on the usual degree $d$ of $w$.
If $d = 2$, we can suppose that $w = [x_{j},x_{l}]$, with $1 \leq j,l \leq n$.
In this case,
\begin{align*}
   \sum_{i=1}^{n} [x_{i},[x_{i},w]] &= \sum_{i=1}^{n} [x_{i},[x_{i},[x_{j},x_{l}]]]
   \\
   &= \sum_{i=1}^{n} [[x_{i},[x_{i},x_{j}]],x_{l}] + 2 \sum_{i=1}^{n} [[x_{i},x_{j}],[x_{i},x_{l}]] + \sum_{i=1}^{n} [x_{j},[x_{i},[x_{i},x_{l}]]]
   \\
   &= 2 \sum_{i=1}^{n} [[x_{i},x_{j}],[x_{i},x_{l}]] \in [\tym(n),\tym(n)],
\end{align*}
where we have used the Jacobi identity and the Yang-Mills relations in the last step. 

Let us suppose that property \eqref{eq:qnulo} holds for any $w$ of degree $d \leq d_{0}$ and let $w$ be of degree $d_{0} + 1$.
We may write $w = \sum_{j=1}^{n} [x_{j},w_{j}]$, with $w_{j}$ of degree less than or equal to $d_{0}$, and by the inductive hypothesis
\[     \sum_{i=1}^{n} [x_{i},[x_{i},w_{j}]] = \sum_{a \in A_{j}} [c_{a}^{j},d_{a}^{j}], \forall \hskip 0.5mm 1 \leq j \leq n,     \]
where $A_{j}$ is a set of indices and $c_{a}^{j}, d_{a}^{j} \in \tym(n)$.
As a consequence,
\begin{align*}
     \sum_{i=1}^{n} [x_{i},[x_{i},w]] &= \sum_{i=1}^{n} \sum_{j=1}^{n} [[x_{i},[x_{i},x_{j}]],w_{j}] 
       + 2 \sum_{i=1}^{n} \sum_{j=1}^{n} [[x_{i},x_{j}],[x_{i},w_{j}]]
       + \sum_{i=1}^{n} \sum_{j=1}^{n} [x_{j},[x_{i},[x_{i},w_{j}]]]
       \\
       &= 2 \sum_{i=1}^{n} \sum_{j=1}^{n} [[x_{i},x_{j}],[x_{i},w_{j}]]
       + \sum_{j=1}^{n} \sum_{a \in A_{j}} [x_{j},[c_{a}^{j},d_{a}^{j}]]
       \\
       &= 2 \sum_{i=1}^{n} \sum_{j=1}^{n} [[x_{i},x_{j}],[x_{i},w_{j}]]
       + \sum_{j=1}^{n} \sum_{a \in A_{j}} ([[x_{j},c_{a}^{j}],d_{a}^{j}] + [c_{a}^{j},[x_{j},d_{a}^{j}]])
\end{align*}
belongs to $[\tym(n),\tym(n)]$.
\end{proof}

As a consequence of Theorem 3.12 and Proposition 3.14 in \cite{HS1}, we can describe the center of the Yang-Mills algebra $\YM(n)$ for $n \geq 3$. 
\begin{proposition}
\label{prop:centroym}
If $n \geq 3$, the center of $\YM(n)$ is $k$.
\end{proposition}
\begin{proof}
On the one hand, it is clear that $k \subseteq \Z(\YM(n))$.

On the other hand, $HH^{0}(\YM(n)) \simeq H^{0}(\ym(n),\YM(n)^{\ad})$.
As stated before, since symmetrization gives a graded isomorphism of $\ym(n)$-modules from $S(\ym(n))$ to $\YM(n)^{\mathrm{ad}}$, 
$H^{0}(\ym(n),\YM(n)^{\ad}) \simeq H^{0}(\ym(n),S(\ym(n))) = S(\ym(n))^{\ym(n)}$.

Let us consider $z \in S(\ym(n))$ of the form
\begin{equation}
\label{eq:zeta}
  z = \sum_{(i_{1},\dots,i_{n}) \in \NN_{0}^{n}, l \in L} c_{(i_{1},\dots,i_{n}),l} x_{1}^{i_{1}} \dots x_{n}^{i_{n}} t_{l},   
\end{equation}
for $c_{(i_{1},\dots,i_{n}),l} \in k$ and $\{ t_{l} \}_{l \in L}$ a PBW basis of $\TYM(n)$.
Then, $z \in S(\ym(n))^{\ym(n)}$ if and only if
\begin{equation}
\label{eq:centro}
   0 = [w,z] = \sum_{(i_{1},\dots,i_{n}) \in \NN_{0}^{n}, l \in L} c_{(i_{1},\dots,i_{n}),l} \Big(x_{1}^{i_{1}} \dots x_{n}^{i_{n}} [w,t_{l}]
             + \sum_{j=1}^{n} x_{1}^{i_{1}} \dots i_{j} x_{j}^{i_{j}-1} [w,x_{j}] \dots x_{n}^{i_{n}} t_{l}\Big)
\end{equation}
for all $w \in \YM(n)$.
We claim that this implies that $z \in \TYM(n)$.
Indeed, let us suppose that this is not the case.
Then there would exist $(i_{1}^{0},\dots,i_{n}^{0}) \in \NN_{0}^{n}$ different from zero and $l_{0} \in L$ such that
$c_{(i_{1},\dots,i_{n}),l} \neq 0$.
Let 
\[     \J = \{ (i_{1},\dots,i_{n}) \in \NN_{0}^{n} : \text{exists} \hskip 0.5mm l \in L \hskip 0.5mm \text{such that} \hskip 0.5mm
c_{(i_{1},\dots,i_{n}),l} \neq 0 \},     \]
and let $(i'_{1},\dots,i'_{n}) \in \J$ be an element of maximal degree $i'_{1} + \dots + i'_{n}$.
Then $[w,z]$ possesses a term of the form
\[     c_{(i'_{1},\dots,i'_{n}),l'} x_{1}^{i'_{1}} \dots x_{n}^{i'_{n}} [w,t_{l'}],     \]
with $c_{(i'_{1},\dots,i'_{n}),l'} \neq 0$, which cannot be cancelled with any other term in the sum \eqref{eq:centro}
for degree 
reasons.
Therefore, $[w,t_{l'}] = 0$, for all $w \in \tym(n)$, or equivalently, $t_{l'} \in \Z(\TYM(n))$.

Since $n \geq 3$, $\TYM(n)$ is a free algebra with an infinite set of generators, so 
its center is the base field $k$.
In other words, $t_{l'} = 1$.

We thus see that any term of the form $c_{(i'_{1},\dots,i'_{n}),l'} x_{1}^{i'_{1}} \dots x_{n}^{i'_{n}} t_{l'}$ in \eqref{eq:zeta},
with $c_{(i'_{1},\dots,i'_{n}),l'} \neq 0$ and maximal $i'_{1} + \dots + i'_{n} = i_{\mathrm{max}}$ has $t_{l'} = 1$.

Since $[x_{h},z] = 0$ for all $h = 1, \dots, n$, it turns out that
\begin{equation}
\label{eq:centro2}
\begin{split}
   0 = [x_{h},z] &= \sum_{\underset{i_{1} + \dots + i_{n} < i_{\mathrm{max}}}{(i_{1},\dots,i_{n}) \in \NN_{0}^{n}, l \in L}}
                       c_{(i_{1},\dots,i_{n}),l} \Big(\overset{\star_{1}}{\overbrace{x_{1}^{i_{1}} \dots x_{n}^{i_{n}} [x_{h},t_{l}]}}
             + \sum_{j=1}^{n}
             \overset{\star_{2}}{\overbrace{x_{1}^{i_{1}} \dots i_{j} x_{j}^{i_{j}-1} [x_{h},x_{j}] \dots x_{n}^{i_{n}} t_{l}}}\Big)
   \\
             &+ \sum_{\underset{i_{1} + \dots + i_{n} = i_{\mathrm{max}}}{(i_{1},\dots,i_{n}) \in \NN_{0}^{n}, l \in L}}
             \sum_{j=1}^{n}
             \overset{\star_{3}}{\overbrace{c_{(i_{1},\dots,i_{n}),l} x_{1}^{i_{1}} \dots i_{j} x_{j}^{i_{j}-1} [x_{h},x_{j}] \dots x_{n}^{i_{n}}}}.
\end{split}
\end{equation}
Notice that in $\star_{1}$ we need only consider the summands with $t_{l} \neq 1$, since $[x_{i},t_{l}] = 0$ if $t_{l} = 1$.

For degree 
reasons, we see that no term of the form $\star_{3}$ can be cancelled with any other term appearing in $\star_{2}$.
On the other hand, the former can neither be cancelled with terms from $\star_{1}$, since $[x_{h},x_{j}]$ is in a homogeneous component
of usual degree $2$ of $\tym(n)$, while $[x_{h},t_{l}]$ is in the homogeneous component
of usual degree strictly greater than $2$ of $\tym(n)$. 
This tells us that the coefficients $c_{(i_{1},\dots,i_{n}),l}$ with maximal $i_{1} + \dots + i_{n}$ must vanish, which is absurd.
As a consequence, $z \in \TYM(n)$.

Again, since 
$z \in \Z(\YM(n)) \cap \TYM(n)$ we see that $z \in \Z(\TYM(n))$.
Therefore $z \in k$ and the proposition is proved.
\end{proof}

\subsection{\texorpdfstring {Another characterization of $W(n)$}{Another characterization of W(n)}}

In \cite{HS1}, Section 3, it was proved that $W(n) \simeq H_{1}(\ym(n),S(V(n)))$ as graded vector spaces.
By Proposition \ref{prop:W(n)}, $W(n)$ is an equivariant $S(V(n))$-module and, by definition, 
the complex $C_{\bullet}(\YM(n),S(V(n)))$ is composed of equivariant $S(V(n))$-modules and equivariant 
$S(V(n))$-linear differentials, so its homology is an equivariant $S(V(n))$-module. 
We recall that $S(V(n)) \otimes V(n)$ is provided with the regular left $S(V(n))$-module structure.
In Proposition \ref{prop:isow(n)} of this section, we shall exhibit an equivariant $S(V(n))$-linear isomorphism from $W(n)$ to
the first homology group of the complex $C_{\bullet}(\YM(n),S(V(n)))$.
From this result, we shall derive three important consequences: a set of generators of the $S(V(n))$-module $W(n)$ given in Corollary \ref{coro:W(n)} and a description of the isotypic decomposition of $W(n)$ and $W(n) \otimes_{S(V(n))} W(n)$ in Corollaries \ref{coro:W(n)son} and \ref{coro:compisownwn}, respectively.

Let $T^{+}V(n)$ be the graded vector subspace of $T(V(n))$ spanned by all homogeneous 
elements of degree greater than or equal to $1$ and $\pi : T(V(n)) \rightarrow S(V(n))$ be the canonical projection.
We start considering the following homogeneous linear map of degree $0$
\begin{align*}
   \phi : T^{+}V(n) &\rightarrow S(V(n)) \otimes V(n)
   \\
   \sum_{i=1}^{n} q_{i} x_{i} &\mapsto \sum_{i=1}^{n} \pi(q_{i}) \otimes x_{i}.
\end{align*}
The previous mapping is well-defined since every element $x \in T^{+}V(n)$ may be written in a unique way as 
$x = \sum_{i=1}^{n} q_{i} x_{i}$ with $q_{i} \in T(V(n))$.
The linearity and homogeneity are direct.
Furthermore, since $\pi$ is surjective, $\phi$ is also surjective.

Given homogeneous elements $z, z' \in T^{+}V(n)$, we have that $\phi(z' z x_{i}) = \pi(z' z) \otimes x_{i} = \pi(z') \pi(z) \otimes x_{i}$,
since $\pi$ is a $k$-algebra morphism. 
In other words, $\phi(z' z) = \pi(z')\cdot \phi(z)$ for $z, z' \in T^{+}V(n)$ homogeneous.
In particular, taking $z' = x_{j}$, we see that $\phi(x_{j} z) = x_{j}.\phi(z)$.
Notice that this does not imply that $\phi$ is $V(n)$-linear, since $T(V(n))$ is not a $S(V(n))$-module for the left multiplication.

We shall denote $\phi'$ the restriction of $\phi$ to $\f(V(n)) \subseteq T^{+}V(n)$.
Then
\begin{equation}
   \label{eq:phi'1}
   \phi'(x_{i}) = 1 \otimes x_{i},
\end{equation}
and we shall prove by induction on $l$ that 
\begin{equation}
   \label{eq:phi'2}
   \phi'([x_{i_{1}},[\dots,[x_{i_{l-1}},x_{i_{l}}]\dots]])
   = x_{i_{1}} \dots x_{i_{l-2}} x_{i_{l-1}} \otimes x_{i_{l}} - x_{i_{1}} \dots x_{i_{l-2}} x_{i_{l}} \otimes x_{i_{l-1}},
\end{equation}
where $l \geq 2$.
The case $l= 2$ is direct.

Let us suppose that $l > 2$ and that the previous identity holds for $l-1$.
In this case
\begin{align*}
   \MoveEqLeft
   \phi'([x_{i_{1}},[\dots,[x_{i_{l-1}},x_{i_{l}}]\dots]])
   \\
   &= \phi(x_{i_{1}} [x_{i_{2}},[\dots,[x_{i_{l-1}},x_{i_{l}}]\dots]]) - \phi([x_{i_{2}},[\dots,[x_{i_{l-1}},x_{i_{l}}]\dots]] x_{i_{1}})
   \\
   &= x_{i_{1}} \phi([x_{i_{2}},[\dots,[x_{i_{l-1}},x_{i_{l}}]\dots]]) - \pi([x_{i_{2}},[\dots,[x_{i_{l-1}},x_{i_{l}}]\dots]]) \otimes x_{i_{1}}
   \\
   &= x_{i_{i}} \phi'([x_{i_{2}},[\dots,[x_{i_{l-1}},x_{i_{l}}]\dots]])
   \\
   &= x_{i_{1}} x_{i_{2}} \dots x_{i_{l-2}} x_{i_{l-1}} \otimes x_{i_{l}} - x_{i_{1}} x_{i_{2}} \dots x_{i_{l-2}} x_{i_{l}} \otimes x_{i_{l-1}},
\end{align*}
where we have used that $\phi(x_{j} z') = x_{j}.\phi(z')$, the inductive hypothesis 
and the fact that, since $\pi$ is a $k$-algebra morphism, $\pi([x,z]) = 0$, for all $x,z \in T(V(n))$.

\begin{lemma}
\label{lema:mu}
If $d_{1} : S(V(n)) \otimes V(n) \rightarrow S(V(n))$ denotes the differential of the complex \eqref{eq:complejohomologiayangmills} 
for $Y = S(V(n))$, then $\Ker(d_{1}) = \phi'([\f(V(n)),\f(V(n))])$.
\end{lemma}
\begin{proof}
The inclusion $\phi'([\f(V(n)),\f(V(n))]) \subseteq \Ker(d_{1})$ is immediate from identity \eqref{eq:phi'2} and the fact that
every element of $[\f(V(n)),\f(V(n))]$ may be written as a linear combination of elements of the form
$[x_{i_{1}},[\dots,[x_{i_{l-1}},x_{i_{l}}]]]$ for $l \geq 2$.

Let us prove the other inclusion.
Consider
\[     y = \sum_{j=1}^{n} \sum_{\bar{i} \in \NN_{0}^{n}} a_{\bar{i}}^{j} x_{1}^{i_{1}} \dots x_{n}^{i_{n}} \otimes x_{j} 
         = \sum_{j=1}^{n} \sum_{\bar{i} \in \NN_{0}^{n}} a_{\bar{i}}^{j} \bar{x}^{\bar{i}} \otimes x_{j}
         \in \Ker(d_{1}),     \]
where $\bar{i} = (i_{1},\dots,i_{n})$ and the previous sum is finite.
We will denote by $e_{i} \in \NN_{0}^{n}$, for $1 \leq i \leq n$, the vector such that $(e_{i})_{j} = \delta_{i,j}$, $1 \leq j \leq n$
and we write $|\bar{i}| = i_{1} + \dots + i_{n}$.

We shall prove that there exists $z \in [\f(V(n)), \f(V(n))]$ such that $y = \phi'(z)$.

On one hand, $y \in \Ker(d_{1})$ if and only if
\[     d_{1}(y) = \sum_{j=1}^{n} \sum_{\bar{i} \in \NN_{0}^{n}} a_{\bar{i}}^{j} \bar{x}^{\bar{i}+e_{j}} = 0.     \]
This condition is equivalent to the following:
for every $(i_{1},\dots,i_{n}) \in \NN_{0}^{n}$
\begin{equation}
\label{eq:cond}
     \sum_{j=1}^{n} a^{j}_{\bar{i}-e_{j}} = 0,
\end{equation}
where we agree to write $a_{\bar{i}}^{j} = 0$, in case there exists $l$ with $1\leq l \leq n$ such that $i_{l} < 0$.

As a consequence, if we define
\[     y_{\bar{i}} = \sum_{j=1}^{n}  a_{\bar{i}-e_{j}}^{j} x_{1}^{i_{1}} \dots x_{j}^{i_{j}-1} \dots x_{n}^{i_{n}} \otimes x_{j},     \]
for every  $\bar{i} \in \NN_{0}^{n}$,
\[     y = \sum_{j=1}^{n} \sum_{\bar{i} \in \NN_{0}^{n}} a_{\bar{i}}^{j} \bar{x}^{\bar{i}} \otimes x_{j}
         = \sum_{\bar{i} \in \NN_{0}^{n}} \left( \sum_{j=1}^{n}  a_{\bar{i}-e_{j}}^{j} \bar{x}^{\bar{i}-e_{j}} \otimes x_{j} \right)
         = \sum_{\bar{i} \in \NN_{0}^{n}} y_{\bar{i}}.
         \]
On the other hand, from \eqref{eq:cond},
we see that $d_{1}(y) = 0$ if and only if $d_{1}(y_{\bar{i}}) = 0$, for all $\bar{i} \in \NN_{0}^{n}$.
Therefore it suffices to prove that, given $\bar{i} \in \NN_{0}^{n}$ and $y \in \Ker(d_{1})$ of the form
$\sum_{j=1}^{n} a^{j}_{\bar{i}-e_{j}} \bar{x}^{\bar{i}-e_{j}} \otimes x_{j}$
there exists $z \in [\f(V(n)), \f(V(n))]$ such that $y = \phi'(z)$.

Suppose given $\bar{i} \in \NN_{0}^{n}$ and $y = \sum_{j=1}^{n} a^{j}_{\bar{i}-e_{j}} \bar{x}^{\bar{i}-e_{j}} \otimes x_{j}$
satisfying $\sum_{j=1}^{n} a^{j}_{\bar{i}-e_{j}} = 0$.
Let $i_{j_{1}}, \dots, i_{j_{l}}$, with $0 \leq l \leq n$, be the nonzero elements of the $n$-tuple $\bar{i}$, \textit{i.e.} $i_{j} \neq 0$ 
if and only if $j \in \{j_{1}, \dots, j_{l}\}$.

If $l = 0$
then necessarily $y = 0 = \phi'(0) \in \phi'([\f(V(n)),\f(V(n))])$, since in this case $a^{j}_{-e_{j}} = 0$, for all $j$ such that $1 \leq j \leq n$.

If $l = 1$, then there is $j_{0}$, with $1 \leq j_{0} \leq n$, such that $\bar{i} = m.e_{j_{0}}$, $m \in \NN$.
Hence condition \eqref{eq:cond} implies that $a^{j_{0}}_{(m-1).e_{j}} = 0$, and therefore $y = 0 = \phi'(0) \in \phi'([\f(V(n)),\f(V(n))])$.

Let $l \geq 2$.
We shall proceed by induction on $l$, assuming that it is true for $l-1$.
We may write
\begin{align*}
      y &= \sum_{j=1}^{n} a^{j}_{\bar{i}-e_{j}} \bar{x}^{\bar{i}-e_{j}} \otimes x_{j}
         = \sum_{p=1}^{l} a^{j_{p}}_{\bar{i}-e_{j_{p}}} \bar{x}^{\bar{i}-e_{j_{p}}} \otimes x_{j_{p}}
         \\
         &= a^{j_{1}}_{\bar{i}-e_{j_{1}}} (\bar{x}^{\bar{i}-e_{j_{1}}} \otimes x_{j_{1}} - \bar{x}^{\bar{i}-e_{j_{2}}} \otimes x_{j_{2}})
         + a^{j_{1}}_{\bar{i}-e_{j_{1}}} \bar{x}^{\bar{i}-e_{j_{2}}} \otimes x_{j_{2}}
         + \sum_{p=2}^{l} a^{j_{p}}_{\bar{i}-e_{j_{p}}} \bar{x}^{\bar{i}-e_{j_{p}}} \otimes x_{j_{p}}
         \\
         &= a^{j_{1}}_{\bar{i}-e_{j_{1}}} (\bar{x}^{\bar{i}-e_{j_{1}}} \otimes x_{j_{1}} - \bar{x}^{\bar{i}-e_{j_{2}}} \otimes x_{j_{2}})
         + \sum_{p=2}^{l} b^{j_{p}}_{\bar{i}-e_{j_{p}}} \bar{x}^{\bar{i}-e_{j_{p}}} \otimes x_{j_{p}}
         \\
         &= a^{j_{1}}_{\bar{i}-e_{j_{1}}} \phi'(\mathrm{ad}^{i_{j_{1}}-1}(x_{j_{1}}) \circ \mathrm{ad}^{i_{j_{2}}-1}(x_{j_{2}}) \circ
         \dots \circ \mathrm{ad}^{i_{j_{l}}}(x_{j_{l}})([x_{j_{2}},x_{j_{1}}]))
         + \sum_{p=2}^{l} b^{j_{p}}_{\bar{i}-e_{j_{p}}} \bar{x}^{\bar{i}-e_{i_{p}}} \otimes x_{i_{p}},
\end{align*}
for $b^{j_{2}}_{\bar{i}-e_{j_{2}}} = a^{j_{1}}_{\bar{i}-e_{j_{1}}} + a^{j_{2}}_{\bar{i}-e_{j_{2}}}$ and
$b^{j_{p}}_{\bar{i}-e_{j_{p}}} = a^{j_{p}}_{\bar{i}-e_{j_{p}}}$, if $3 \leq p \leq l$.

Since  $\sum_{p=2}^{l} b^{j_{p}}_{\bar{i}-e_{j_{p}}} = 0$,
the element  $y' = \sum_{p=2}^{l} b^{j_{p}}_{\bar{i}-e_{j_{p}}} \bar{x}^{\bar{i}-e_{i_{p}}} \otimes x_{i_{p}}$
belongs to the kernel of $d_{1}$.
By the inductive hypothesis, there is $z' \in [\f(V(n)),\f(V(n))]$ such that $y' = \phi'(z')$.
Then
\begin{align*}
   y &= a^{j_{1}}_{\bar{i}-e_{j_{1}}} \phi'(\mathrm{ad}^{i_{j_{1}}-1}(x_{j_{1}}) \circ \mathrm{ad}^{i_{j_{2}}-1}(x_{j_{2}}) \circ
      \mathrm{ad}^{i_{j_{3}}}(x_{j_{3}}) \circ \dots \circ \mathrm{ad}^{i_{j_{l}}}(x_{j_{l}})([x_{j_{2}},x_{j_{1}}]))
     + \phi'(z')
     \\
     &= \phi'(a^{j_{1}}_{\bar{i}-e_{j_{1}}} \mathrm{ad}^{i_{j_{1}}-1}(x_{j_{1}}) \circ \mathrm{ad}^{i_{j_{2}}-1}(x_{j_{2}}) \circ
       \mathrm{ad}^{i_{j_{3}}}(x_{j_{3}}) \circ \dots \circ \mathrm{ad}^{i_{j_{l}}}(x_{j_{l}})([x_{j_{2}},x_{j_{1}}]) + z').
\end{align*}
This proves the lemma.
\end{proof}

Let $d_{2} : S(V(n)) \otimes V(n) \rightarrow S(V(n)) \otimes V(n)$ be the differential of the complex \eqref{eq:complejohomologiayangmills}
in degree $2$ with $Y = S(V(n))$.
In this case,
\[     d_{2}(\sum_{i=1}^{n} z_{i} \otimes x_{i}) = \sum_{i,j = 1}^{n} (z_{i}x_{j}^{2} \otimes x_{i} - z_{i} x_{i} x_{j} \otimes x_{j}).     \]
We may consider the homogeneous linear map of degree $0$, denoted by $\tilde{\phi}$,
\[     [\f(V(n)),\f(V(n))] \rightarrow \Ker(d_{1})/\mathrm{Im}(d_{2}),     \]
given by composition of $\phi'$ and the canonical projection.
Being the composition of surjective morphisms, $\tilde{\phi}$ is surjective. 

\begin{lemma}
Let $d_{2}$ be the differential of the complex \eqref{eq:complejohomologiayangmills} in degree $2$ with $Y = S(V(n))$ and
let $\tilde{\phi}$ be as above.
If $\cl{R(n)}$ denotes the Lie ideal in $\f(V(n))$ generated by the vector space of Yang-Mills relations \eqref{eq:relym}, then
$\phi'(\cl{R(n)}) \subseteq \mathrm{Im}(d_{2})$, and therefore $\tilde{\phi}$ induces a surjective homogeneous linear morphism of degree $0$
from $[\f(V(n)),\f(V(n))]/\cl{R(n)} = \tym(n)$ to $\Ker(d_{1})/\mathrm{Im}(d_{2}) = H_{1} (\ym(n), S(V(n))) \simeq W(n)$. 
\end{lemma}
\begin{proof}
First, note that $\cl{R(n)} \subseteq [\f(V(n)),\f(V(n))]$.

Given $j$, with $1 \leq j \leq n$, we shall denote $r_{j} = \sum_{i=1}^{n} [x_{i},[x_{i},x_{j}]]$.
Using the Jacobi relation it is easy to see that every element of $\cl{R(n)}$ may be written as a linear combination of elements of the form
$[x_{i_{1}}, [x_{i_{2}},[\dots, [x_{i_{p-1}},r_{i_{p}}]\dots]]]$, 
for $p \in \NN$, $i_{1}, \dots, i_{p} \in \{ 1, \dots, n \}$.

Using the identity \eqref{eq:phi'2}, we get
\begin{align*}     
       \phi'([x_{i_{1}}, [x_{i_{2}},[\dots, [x_{i_{p-1}},r_{i_{p}}]\dots]]])
       &= \sum_{j=1}^{n} (x_{i_{1}} x_{i_{2}} \dots x_{i_{p-1}} x_{j}^{2} \otimes x_{i_{p}}
       - x_{i_{1}} x_{i_{2}} \dots x_{i_{p-1}} x_{j} x_{i_{p}} \otimes x_{j})
       \\
       &= d_{2}(x_{i_{1}} x_{i_{2}} \dots x_{i_{p-1}} \otimes x_{i_{p}}),     
\end{align*}
and so $\phi'(\cl{R(n)}) \subseteq \mathrm{Im}(d_{2})$.
\end{proof}

We have therefore defined a surjective homogeneous $k$-linear map of degree $0$
\[     \tilde{\phi}: \tym(n) \rightarrow H_{1}(\ym(n),S(V(n))).     \]
We will see that $\tilde{\phi}([\tym(n),\tym(n)]) = 0$ as follows.
Let us consider $a, b \in [\f(V(n)),\f(V(n))]$ such that $\bar{a}, \bar{b} \in \tym(n)$.
Taking into account that $\tilde{\phi}([\bar{a},\bar{b}])$ is the class of $\phi'([a,b])$ in $\Ker(d_{1})/\mathrm{Im}(d_{2})$,
it suffices to show that $\phi'([a,b]) \in \mathrm{Im}(d_{2})$.
We shall see that in fact $\phi'([a,b]) = 0$.

Since $a, b \in [\f(V(n)),\f(V(n))]$, we can write $a = \sum_{j=1}^{n} [x_{j},a'_{j}]$ and $b = \sum_{j=1}^{n} [x_{j},b'_{j}]$, for some
$a'_{j}, b'_{j} \in \f(V(n))$.
Hence
\[     \phi'([a,b]) = \phi(a b - b a) = \phi(a b) - \phi(b a) = \pi(a) \phi(b) - \pi(b) \phi(a) = 0,     \]
where 
we have used that $\pi(a) = \pi(b) = 0$, for $\pi$ is a $k$-algebra morphism.

Finally, the fact that $\tilde{\phi}([\tym(n),\tym(n)]) = 0$ implies that $\tilde{\phi}$ induces a surjective morphism,
which will be denoted by $\Phi$,
\[     \tym(n)/[\tym(n),\tym(n)] \overset{\Phi}{\rightarrow} H_{1}(\ym(n),S(V(n))).     \]
Also, taking into account that $\tym(n)/[\tym(n),\tym(n)]$ is isomorphic to $W(n)$ and the latter is locally finite dimensional and isomorphic to 
the first homology group $\Ker(d_{1})/\mathrm{Im}(d_{2}) \simeq H_{1}(\ym(n),S(V(n)))$, 
$\Phi$ turns out to be an isomorphism.

We have then proved the following proposition.
\begin{proposition}
\label{prop:isow(n)}
The map
\[     \Phi : \tym(n)/[\tym(n),\tym(n)] \rightarrow \Ker(d_{1})/\mathrm{Im}(d_{2}) \simeq H_{1}(\ym(n),S(V(n)))     \]
is equivariant.  
\end{proposition}
\begin{proof} 
We have already proved that $\Phi$ is a homogeneous linear isomorphism of degree $0$. 
Also, the equation \eqref{eq:phi'2} tells us that $\Phi$ is $V(n)$-linear and $\so(n)$-equivariant. 
\end{proof}

The previous proposition has the following important consequences. 
\begin{corollary}
\label{coro:W(n)}
The graded vector space $W(n)$ is generated by the finite set $ \{ [x_{i},x_{j}] \}_{1 \leq i < j \leq n}$
both as a graded $S(V(n))$-module and as a graded $S(V(n))/\cl{q}$-module. 
Furthermore, a collection of generators of $W(n)$ for both module structures is given by
$ \{ [x_{i},x_{j}] \}_{1 \leq i < j \leq n}$
\end{corollary}
\begin{proof}
As stated at the beginning of this section, we consider $S(V(n)) \otimes V(n)$ provided with the regular left action of $S(V(n))$.
It is finitely generated, and $S(V(n))$ being noetherian, $S(V(n)) \otimes V(n)$ is also noetherian.
Since the differential $d_{1}$ of the Koszul complex with coefficients in $S(V(n))$ is a $S(V(n))$-linear map, its kernel
is also a finitely generated $S(V(n))$-submodule.
By Lemma \ref{lema:mu}, the set $\{ x_{i} \otimes x_{j} - x_{j} \otimes x_{i} \}_{1 \leq i < j \leq n}$
is a set of generators of $\Ker(d_{1})$ as $S(V(n))$-module.

On the other hand, since $d_{2}$ is also a $V(n)$-linear map, $\mathrm{Im}(d_{2})$ is a $V(n)$-submodule of $\Ker(d_{1})$.
The $S(V(n))$-module $W(n) \simeq H_{1}(\ym(n),S(V(n)))$ is then a quotient of
the finitely generated $S(V(n))$-module $\Ker(d_{1})$ by the submodule $\mathrm{Im}(d_{2})$,
and hence it is finitely generated with set of generators $\{ [x_{i},x_{j}] \}_{1 \leq i < j \leq n}$.
All these considerations hold as well over the algebra $S(V(n))/\cl{q}$.
\end{proof}

In the following corollaries and the rest of this article we shall use the standard notation for the irreducible finite dimensional representations 
of the Lie algebras $\so(n)$ (see \cite{FH1}). 
\begin{corollary}
\label{coro:W(n)son}
Let $n \geq 3$. 
The homogeneous component of degree $p$ of $W(n)$ vanishes for $p \leq 1$. 

For $p \geq 2$, the homogeneous component of degree $p$ of the $\so(n)$-module $W(n)$ is also an $\so(n)$-module. 
If $n = 3$, it is isomorphic to $\Gamma_{(p-1)L_{1}}$; in case $n = 4$, it is isomorphic to 
$\Gamma_{(p-1)L_{1}+L_{2}} \oplus \Gamma_{(p-1)L_{1}-L_{2}}$; and finally, 
if $n \geq 5$, it is isomorphic to $\Gamma_{(p-1)L_{1}+L_{2}}$. 
\end{corollary}
\begin{proof}
The complex of graded $\so(n)$-modules $C_{\bullet}(\YM(n),S(V(n)))$ is the direct sum of the complexes of finite dimensional $\so(n)$-modules
\begin{equation}
\label{eq:compsvn}
    0 \rightarrow S^{p-4}(V(n))[-4] \overset{d_{3}^{p-4}}{\rightarrow} (S^{p-3}(V(n)) \otimes V(n))[-2] \overset{d_{2}^{p-3}}{\rightarrow}
    S^{p-1}(V(n)) \otimes V(n) \overset{d_{1}^{p-1}}{\rightarrow} S^{p}(V(n)) \rightarrow 0,
\end{equation}
where $p \in \ZZ$ and we consider $S^{p}(V(n)) = 0$ if $p < 0$. 
By Proposition \ref{prop:isow(n)}, its homology is isomorphic to $W(n)$ in degree one, to $k$ in degree zero 
and all other homology modules vanish. 

Let $\mathcal{S}(n)$ denote the set of isomorphism classes of irreducible finite dimensional $\so(n)$-modules. 
If $M$ is a finite dimensional $\so(n)$-module and $s \in \mathcal{S}(n)$, 
we shall denote by $n_{s}(M)$ the number of copies of the isotypic component of type $s$ appearing in $M$. 
Hence, the isotypic decomposition of $M$ may be encoded in the formal sum of finite support $\sum_{s \in \mathcal{S}(n)} n_{s}(M) s$. 
It is directly checked that $M \mapsto \sum_{s \in \mathcal{S}(n)} n_{s}(M) s$ is an Euler-Poincar\'e map 
(see \cite{La1}, Chap. III, \S 8). 

In consequence, the Euler-Poincar\'e characteristic of a complex of finite dimensional $\so(n)$-modules coincides 
with the Euler-Poincar\'e characteristic of its homology (see \cite{La1}, Chap. XX, \S 3, Thm. 3.1). 
This result applied to the complex \eqref{eq:compsvn} allows us to compute the isotypic decomposition of the $p$-th homogeneous component of 
$W(n)$ once we have obtained the Euler-Poincar\'e characteristic of \eqref{eq:compsvn}. 
In order to do so, we shall proceed as follows. 

First, we recall that $V(n) \simeq \Gamma_{L_{1}}$ and, by \cite{FH1}, Exercise 19.21, 
\begin{equation}
\label{eq:isosp}
     S^{p}(V(n)) \simeq \bigoplus_{d = 0}^{[p/2]} \Gamma_{(p-2d) L_{1}},     
\end{equation}
where $[p/2]$ denotes the integral part of $p/2$. 
The isomorphism $\Gamma_{0 L_{1}} \simeq k$ tells us that $\Gamma_{0 L_{1}} \otimes \Gamma_{L_{1}} \simeq \Gamma_{L_{1}}$. 
Also, we have the following fusion rule for the tensor product 
\begin{equation}
\label{eq:compiso}
     \Gamma_{q L_{1}} \otimes \Gamma_{L_{1}} \simeq \begin{cases}
                                                         \Gamma_{(q+1) L_{1}} \oplus \Gamma_{q L_{1}} \oplus \Gamma_{(q-1) L_{1}}, &\text{if $n=3$,}
                                                         \\
                                                         \Gamma_{(q+1) L_{1}} \oplus \Gamma_{q L_{1} + L_{2}} \oplus \Gamma_{q L_{1} - L_{2}} 
                                                         \oplus \Gamma_{(q-1) L_{1}}, &\text{if $n=4$,}
                                                         \\
                                                         \Gamma_{(q+1) L_{1}} \oplus \Gamma_{q L_{1} + L_{2}} \oplus \Gamma_{(q-1) L_{1}}, 
                                                         &\text{if $n\geq 5$,}
                                                      \end{cases}
\end{equation}
for $q \geq 1$. 
The previous computation is straightforward from the \v{Z}elobenko fusion rules (see \cite{Zel1}, \S 131, Thm. 5). 

Using the isomorphisms \eqref{eq:isosp} and \eqref{eq:compiso} we find that the Euler-Poincar\'e characteristic of the complex \eqref{eq:compsvn} 
is $k$ if $p=0$, it vanishes if $p=1$, and, for $p \geq 2$, it is $\Gamma_{(p-1)L_{1}}$ if $n=3$, 
$\Gamma_{(p-1)L_{1}+L_{2}} + \Gamma_{(p-1)L_{1}-L_{2}}$ if $n =4$ and 
$\Gamma_{(p-1)L_{1}+L_{2}}$ if $n \geq 5$. 
The corollary thus follows. 
\end{proof}

As a direct consequence of the previous corollary we obtain the following result which we shall use in Subsection \ref{subsec:kerneld}. 
\begin{corollary}
\label{coro:compisownwn}
Let $n \geq 3$. 
The equivariant $S(V(n))$-module $W(n) \otimes_{S(V(n))} W(n)$ has no isotypic component of type $k$ in degree greater than $4$. 
Also, the homogeneous component of degree $4$ of $W(n) \otimes_{S(V(n))} W(n)$ is isomorphic to $\Lambda^{2} V(n) \otimes \Lambda^{2} V(n)$ as $\so(n)$-modules; it may be decomposed as
\begin{scriptsize}
\[     \Lambda^{2} V(n) \otimes \Lambda^{2} V(n) = \begin{cases}
                                                 \Gamma_{2 L_{1}} \oplus \Lambda^{4} V(n) \oplus \Lambda^{2} V(n) \oplus k, &\text{if $n=3$,}
                                                 \\
                                                 \Gamma_{2 L_{1}} \oplus \Gamma_{2 L_{1} + 2 L_{2}} \oplus \Gamma_{2 L_{1} - 2 L_{2}} 
                                                 \oplus \Gamma_{2 L_{1}}
                                                 \oplus \Lambda^{4} V(n) \oplus \Lambda^{2} V(n) \oplus k, &\text{if $n=4$,}
                                                 \\
                                                 \Gamma_{2 L_{1} + L_{2}} \oplus \Gamma_{2 L_{1} + 2 L_{2}} \oplus \Gamma_{2 L_{1}} 
                                                 \oplus \Lambda^{4} V(n) \oplus \Lambda^{2} V(n) \oplus k, &\text{if $n=5$,}
                                                 \\
                                                 \Gamma_{2 L_{1} + L_{2} + L_{3}} \oplus \Gamma_{2 L_{1} + L_{2} - L_{3}}
                                                 \oplus \Gamma_{2 L_{1} + 2 L_{2}} \oplus \Gamma_{2 L_{1}} 
                                                 \oplus \Lambda^{4} V(n) \oplus \Lambda^{2} V(n) \oplus k, &\text{if $n=6$,}
                                                 \\
                                                 \Gamma_{2 L_{1} + L_{2} + L_{3}} \oplus \Gamma_{2 L_{1} + 2 L_{2}} \oplus \Gamma_{2 L_{1}} 
                                                 \oplus \Lambda^{4} V(n) \oplus \Lambda^{2} V(n) \oplus k, &\text{if $n \geq 7$.}
                                               \end{cases}
\]
\end{scriptsize}
\end{corollary}
\begin{proof}
The existence of an equivariant epimorphism 
$S(V(n)) \otimes \Lambda^{2} V(n) \twoheadrightarrow W(n)$ says that the $S(V(n))$-module $W(n) \otimes_{S(V(n))} W(n)$ is an 
epimorphic image of $W(n) \otimes \Lambda^{2} V(n)$, which has homogeneous components of degree greater than or equal to $4$. 
By the previous corollary, the component of degree $p+2$ of $W(n) \otimes \Lambda^{2} V(n)$, for $p > 2$, is given by 
\[     W(n)_{p} \otimes \Lambda^{2} V(n) \simeq \begin{cases}
                                                   \Gamma_{(p-1) L_{1}} \otimes \Gamma_{L_{1}}, &\text{if $n=3$,}
                                                   \\
                                                   (\Gamma_{(p-1) L_{1} + L_{2}} \oplus \Gamma_{(p-1) L_{1} - L_{2}}) \otimes \Gamma_{L_{1}}, 
                                                   &\text{if $n=4$,}
                                                   \\
                                                   \Gamma_{(p-1) L_{1} + L_{2}} \otimes \Gamma_{L_{1}}, &\text{if $n \geq 5$.}
                                                \end{cases}
\]
Using the \v{Z}elobenko fusion rules we find that $k \simeq \Gamma_{0 L_{1}}$ is not an isotypic component of $W(n)_{p} \otimes \Lambda^{2} V(n)$ 
for $p > 2$, which proves the first statement. 

For the second statement we proceed as follows. 
Using again the \v{Z}elobenko fusion rules for the tensor product $\Lambda^{2} V(n) \otimes \Lambda^{2} V(n) \simeq \so(n) \otimes \so(n)$, we find that 
\begin{scriptsize}
\[     \Lambda^{2} V(n) \otimes \Lambda^{2} V(n) \simeq \begin{cases}
                                                 \Gamma_{2 L_{1}} \oplus V(n) \oplus k, &\text{if $n=3$,}
                                                 \\
                                                 \Gamma_{2 L_{1} + 2 L_{2}} \oplus \Gamma_{2 L_{1} - 2 L_{2}} \oplus \Gamma_{2 L_{1}}^{\oplus 2} 
                                                 \oplus \Gamma_{L_{1} + L_{2}} \oplus \Gamma_{L_{1} - L_{2}} \oplus k^{\oplus 2}, &\text{if $n=4$,}
                                                 \\
                                                 \Gamma_{2 L_{1} + 2 L_{2}} \oplus \Gamma_{2 L_{1} + L_{2}} \oplus \Gamma_{2 L_{1}}
                                                 \oplus \Gamma_{L_{1} + L_{2}} \oplus \Gamma_{L_{1}} \oplus k, &\text{if $n=5$,}
                                                 \\
                                                 \Gamma_{2 L_{1} + L_{2} + L_{3}} \oplus \Gamma_{2 L_{1} + L_{2} - L_{3}}
                                                 \oplus \Gamma_{2 L_{1} + 2 L_{2}} \oplus \Gamma_{2 L_{1}} \oplus 
                                                 \Gamma_{L_{1} + L_{2}} \oplus \Gamma_{L_{1} + L_{2}} \oplus k, &\text{if $n=6$,}
                                                 \\
                                                 \Gamma_{2 L_{1} + L_{2} + L_{3}} \oplus \Gamma_{2 L_{1} + 2 L_{2}} \oplus \Gamma_{2 L_{1}}
                                                 \oplus \Gamma_{L_{1} + L_{2} + L_{3}} \oplus \Gamma_{L_{1} + L_{2}} \oplus k, &\text{if $n=7$,}
                                                 \\
                                                 \Gamma_{2 L_{1} + L_{2} + L_{3}} \oplus \Gamma_{2 L_{1} + 2 L_{2}} \oplus \Gamma_{2 L_{1}}
                                                 \oplus \Gamma_{L_{1} + L_{2} + L_{3} + L_{4}} \oplus 
                                                 \Gamma_{L_{1} + L_{2} + L_{3} - L_{4}} \oplus \Gamma_{L_{1} + L_{2}} \oplus k, &\text{if $n=8$,}
                                                 \\
                                                 \Gamma_{2 L_{1} + L_{2} + L_{3}} \oplus \Gamma_{2 L_{1} + 2 L_{2}} \oplus \Gamma_{2 L_{1}}
                                                 \oplus \Gamma_{L_{1} + L_{2} + L_{3} + L_{4}} \oplus \Gamma_{L_{1} + L_{2}} \oplus k, &\text{if $n>8$.}
                                                 \end{cases}
\]
\end{scriptsize}
Taking into account that 
\begin{align*}
   \Lambda^{2} V(n) &\simeq V(n),   &\text{if $n = 3$,}
   \\
   \Lambda^{2} V(n) &\simeq \Gamma_{L_{1}+L_{2}} \oplus \Gamma_{L_{1}-L_{2}},   &\text{if $n = 4$,}
   \\
   \Lambda^{2} V(n) &\simeq \Gamma_{L_{1} + L_{2}},    &\text{if $n \geq 5$,}
\end{align*}
and
\begin{align*}
   \Lambda^{4} V(n) &\simeq 0, &\text{if $n=3$,}
   \\
   \Lambda^{4} V(n) &\simeq k,   &\text{if $n = 4$,}
   \\
   \Lambda^{4} V(n) &\simeq V(n) \simeq \Gamma_{L_{1}},   &\text{if $n = 5$,}
   \\
   \Lambda^{4} V(n) &\simeq \Lambda^{2} V(n) \simeq \Gamma_{L_{1} + L_{2}},    &\text{if $n = 6$,}
   \\
   \Lambda^{4} V(n) &\simeq \Lambda^{3} V(n) \simeq \Gamma_{L_{1} + L_{2} + L_{3}},    &\text{if $n = 7$,}
   \\
   \Lambda^{4} V(n) &\simeq \Gamma_{L_{1} + L_{2}+L_{3}+L_{4}} \oplus \Gamma_{L_{1} + L_{2}+L_{3}-L_{4}},
   &\text{if $n = 8$,}
   \\
   \Lambda^{4} V(n) &\simeq \Gamma_{L_{1} + L_{2}+L_{3}+L_{4}},    &\text{if $n > 8$,}
\end{align*}
we obtain the desired decomposition for $\Lambda^{2} V(n) \otimes \Lambda^{2} V(n)$. 
\end{proof}

\subsection{\texorpdfstring{Some algebraic properties of $W(n)$}{Some algebraic properties of W(n)}}
\label{subsec:geogen}

In this subsection we shall prove some algebraic properties of $W(n)$ which will be very useful in the sequel.
At the end of this subsection we briefly discuss  a geometric interpretation of these properties.

The following Lemma is analogous to the K\"unneth formula
.
\begin{lemma}
\label{lema:kunneth}
Let $C_{\bullet} = C_{\bullet}(\YM(n),S(V(n)))$ be the complex \eqref{eq:complejohomologiayangmills} for the equivariant left
$\YM(n)$-module $S(V(n))$ and let $z \in S(V(n))$ be a nonzero homogeneous element of degree $d$.
After applying the functor $S(V(n))/\cl{z} \otimes_{S(V(n))} (\place)$ to the complex $C_{\bullet}$, we obtain the following
short exact sequence composed of graded $S(V(n))$-modules and homogeneous morphisms of degree $0$
\begin{small}
\[     0 \rightarrow S(V(n))/\cl{z} \otimes_{S(V(n))} H_{q}(C) 
         \rightarrow H_{q}(S(V(n))/\cl{z} \otimes_{S(V(n))} C_{\bullet})
         \rightarrow \Tor^{\text{\begin{tiny}$S(V(n))$\end{tiny}}}_{q}(S(V(n))/\cl{z},H_{q-1}(C_{\bullet})) 
         \rightarrow 0.     \]
\end{small}
\end{lemma}
\begin{proof}
First, we see that $C_{\bullet}$ is a complex of free graded left $S(V(n))$-modules.
Its homology was computed in \cite{HS1}, Prop. 3.5.

We consider a free graded resolution of the $S(V(n))$-module $S(V(n))/\cl{z}$, which will be denoted by $P_{\bullet}$,
provided with morphisms of degree $0$. 
Since the $S(V(n))$-module $S(V(n))/\cl{z}$ has projective dimension $1$, we may choose $P_{\bullet}$ such that 
$P_{i} = 0$ for $i \geq 2$. 

We can apply the K\"unneth spectral sequence (see \cite{Wei1}, Thm. 5.6.4), which yields
\[     E^{2}_{p,q} = \Tor^{\text{\begin{tiny}$S(V(n))$\end{tiny}}}_{p} (S(V(n))/\cl{z} , H_{q}(C_{\bullet})) \Rightarrow H_{p+q} (S(V(n))/\cl{z} \otimes_{S(V(n))} C_{\bullet}).     \]
If we consider the double complex $D_{p,q} = P_{p} \otimes_{S(V(n))} C_{q}$, the previous spectral sequence is just the
spectral sequence of the filtration by columns of $D_{\bullet,\bullet}$.
Hence, the first term of this spectral sequence is of the form $E^{1}_{p,q} = H_{q}(D_{p,\bullet}) = P_{p} \otimes_{S(V(n))} H_{q}(C_{\bullet})$,
since $P_{p}$ is a free graded $S(V(n))$-modules, for all $p$.
As a consequence, $E^{1}_{p,q}$ consists of only two columns $p = 0,1$, so \textit{a fortiori}, $E^{2}_{p,q}$
vanishes outside the columns $p = 0, 1$.
Hence we obtain a short exact sequence of graded $S(V(n))$-modules provided with homogeneous morphisms of degree $0$
(see \cite{Rot1}, Cor. 10.29) 
\begin{small}
\[     0 \rightarrow \Tor^{\text{\begin{tiny}$S(V(n))$\end{tiny}}}_{0} (S(V(n))/\cl{z} , H_{q}(C_{\bullet})) \rightarrow H_{q}(S(V(n))/\cl{z} \otimes_{S(V(n))} C_{\bullet})
         \rightarrow \Tor^{\text{\begin{tiny}$S(V(n))$\end{tiny}}}_{1}(S(V(n))/\cl{z},H_{q-1}(C)) \rightarrow 0,     \]
\end{small}
which proves the lemma.
\end{proof}

Since $H_{2}(C_{\bullet}) = 0$, the previous lemma implies that
\begin{align*}
     H_{2} (\ym(n),S(V(n))/\cl{z}) &= H_{2}(S(V(n))/\cl{z} \otimes_{S(V(n))} C_{\bullet})
                            \simeq \Tor^{\text{\begin{tiny}$S(V(n))$\end{tiny}}}_{1}(S(V(n))/\cl{z},H_{1}(C_{\bullet})) 
     \\                            
     &\simeq \Tor^{\text{\begin{tiny}$S(V(n))$\end{tiny}}}_{1}(S(V(n))/\cl{z},W(n)).     
\end{align*}

On the other hand, there is a free graded resolution of the $S(V(n))$-module $S(V(n))/\cl{z}$ of the form
\begin{equation}
\label{eq:resA}
     0 \rightarrow S(V(n))[-d] \overset{\cdot z}{\rightarrow} S(V(n)) \rightarrow S(V(n))/\cl{z} \rightarrow 0,
\end{equation}
so the $S(V(n))$-module $S(V(n))/\cl{z}$ has projective dimension less than or equal to $1$.
From the previous resolution, we conclude that $\Tor^{\text{\begin{tiny}$S(V(n))$\end{tiny}}}_{1}(S(V(n))/\cl{z},W(n)) \simeq \mathrm{ann}_{W(n)[-d]}(z)$,
where $\mathrm{ann}_{W}(z) = \{ w \in W : z\cdot w = 0 \}$.

We recall that $q = \sum_{i=1}^{n} x_{i}^{2} \in S(V(n))$. 
If we set $z = q$, then $d = 2$ and we obtain a homogeneous left $S(V(n))/\cl{q}$-linear isomorphism of degree $0$ of the form
$H_{2}(S(V(n))/\cl{q} \otimes_{S(V(n))} C) \simeq W(n)[-2]$.

We also have the following result.
\begin{proposition}
\label{prop:Minyectivo}
Let $n \geq 3$.
The generators $x_{1}, \dots, x_{n} \in S(V(n))$ are nonzerodivisors on $W(n)$.
\end{proposition}
\begin{proof}
Let us now assume that $z = x_{i}$.
By the previous isomorphisms, we have that
\[     H_{2} (\ym(n),S(V(n))/\cl{x_{i}}) \simeq \Tor^{\text{\begin{tiny}$S(V(n))$\end{tiny}}}_{1}(S(V(n))/\cl{x_{i}},W(n)) \simeq \mathrm{ann}_{W(n)[-1]}(x_{i}).     \]
Since 
we have chosen $n \geq 3$, then $H_{n-1}(V(n),S(V(n))/\cl{x_{i}}) = 0$.

Taking into account that $q$ and $x_{i}$ are coprime in $S(V(n))$, the map on $S(V(n))/\cl{x_{i}}$
given by multiplication by $q$ is injective.
Proposition \ref{prop:annhyangmills} tells us that $H_{2}(\ym(n),S(V(n))/\cl{x_{i}}) = 0$.
This in turn implies that $\mathrm{ann}_{W(n)}(x_{i}) = 0$, for all $i = 1, \dots, n$, so every $x_{i}$ is a nonzerodivisor on $W(n)$
and hence the natural morphism of localization $W(n) \rightarrow W(n)_{(x_{i})}$ is injective for all $i = 1, \dots, n$.
\end{proof}

We recall that, if $A$ is an $\NN_{0}$-graded algebra, the homogeneous morphism of degree $0$ given by
\begin{align*}
   d_{\mathrm{eu}} : A &\rightarrow A
   \\
   a &\mapsto |a| a,
\end{align*}
where $a \in A$ is homogeneous of  degree $|a|$,
is a derivation of $A$, called the \emph{Eulerian derivation}.

The following fact is implicit in \cite{Mov1}.
\begin{proposition}
\label{prop:succorta}
Consider $q = \sum_{i=1}^{n} x_{i}^{2} \in S(V(n))$ and $A = S(V(n))/\cl{q}$. 
There is a short exact sequence of graded $A$-modules 
\begin{equation}
\label{eq:sucW(n)}
   0 \rightarrow W(n) \rightarrow \Omega_{A/k} \overset{d'_{\mathrm{eu}}}{\rightarrow} A_{+} \rightarrow 0,
\end{equation}
where $A_{+} = \oplus_{m \geq 1} A_{m}$ is the irrelevant ideal of the $\NN_{0}$-graded algebra $A$, $\Omega_{A/k}$ is the module of 
K\"ahler differentials of $A$ over $k$ and $d'_{\mathrm{eu}}$ is the map induced by the Eulerian derivation $d_{\mathrm{eu}} : A \rightarrow A$. 
\end{proposition}
\begin{proof}
We know that there is a homogeneous isomorphism $W(n)[-2] \simeq H_{2}(A \otimes_{S(V(n))} C_{\bullet})$ 
of graded $A$-modules of degree $0$.
Inspecting the complex $A \otimes_{S(V(n))} C_{\bullet}$, we conclude that 
\[     \Ker(\id_{A} \otimes d_{2}) = \Big\{ \sum_{i=1}^{n} a_{i} \otimes x_{i} : \sum_{i}^{n} a_{i} x_{i} = 0 \Big\},     \]
because
\begin{align*}
   (\id_{A} \otimes d_{2})(\sum_{i=1}^{n} a_{i} \otimes x_{i})
   &= \sum_{i,j = 1}^{n} (a_{i} x_{j}^{2} \otimes x_{i} - a_{i} x_{j} x_{i} \otimes x_{j})
   \\
   &= \sum_{i=1}^{n} a_{i} q \otimes x_{i} - \sum_{j=1}^{n}(\sum_{i=1}^{n}a_{i} x_{i})x_{j} \otimes x_{j}
   \\
   &= - \sum_{j=1}^{n}(\sum_{i=1}^{n}a_{i} x_{i})x_{j} \otimes x_{j},
\end{align*}
and thus $(\id_{A} \otimes d_{2})(\sum_{i=1}^{n} a_{i} \otimes x_{i}) = 0$ if and only if $\sum_{i=1}^{n}a_{i} x_{i} = 0$, 
for $A$ is a domain.
Also,
\[     \mathrm{Im}(\id_{A} \otimes d_{3}) = \Big\{ \sum_{i=1}^{n} a x_{i} \otimes x_{i} : a \in A \Big\}.     \]

The second fundamental sequence for the quotient $A = S(V(n))/\cl{q}$ is 
\[     \cl{q}/\cl{q}^{2} \overset{\delta}{\rightarrow} A \otimes_{S(V(n))} \Omega_{S(V(n))/k}
       \overset{\alpha}{\rightarrow} \Omega_{A/k} \rightarrow 0,     \]
where $\delta(\bar{q}) =  dq$ and $\alpha(\bar{p} \otimes_{S(V(n))} dz) = \bar{p} d\bar{z}$, for $z, p \in S(V(n))$.

Since $\Omega_{S(V(n))/k} \simeq S(V(n)) \otimes V(n)$, by the isomorphism $z \otimes x_{i} \mapsto z dx_{i}$
(see \cite{Hart1}, Example 8.2.1), we derive that, using this identification, 
$\mathrm{Im}(\delta) = \mathrm{Im}(\id_{A} \otimes d_{3})$ and, moreover, the map 
\begin{align*}
   (A \otimes V(n))/\mathrm{Im}(\id_{A} \otimes d_{3}) &\rightarrow \Omega_{A/k}
   \\
   \overline {\bar{p} \otimes x_{i}} &\mapsto \bar{p} d\bar{x}_{i},
\end{align*}
is a homogeneous $A$-linear isomorphism of degree $0$.
In other words, there is a short exact sequence of graded $A$-modules
\begin{equation}
\label{eq:euler}
     0 \rightarrow A[-2] \rightarrow A \otimes V(n) \rightarrow \Omega_{A/k} \rightarrow 0,
\end{equation}
where the first mapping is $a \mapsto \sum_{i=1}^{n}a x_{i} \otimes x_{i}$.
Notice that $A \otimes V(n) \simeq (A[-1])^{n}$.

Finally, the morphism given by the inclusion $\Ker(d_{2}) \hookrightarrow A \otimes V(n)$ induces a map of graded 
$A$-modules $W(n) \simeq H_{2}(A \otimes_{S(V(n))} C_{\bullet})[2] \hookrightarrow (A \otimes V(n))/\mathrm{Im}(\id_{A} \otimes d_{3}) \simeq \Omega_{A/k}$. 
It is readily verified that it provides the first morphism of the short exact sequence of the proposition. 
Also, it is clear that the map $d'_{\mathrm{eu}}$ is an epimorphism and its kernel coincides with the image of the previous morphism.
\end{proof}

\begin{remark}
\label{rem:geown}
Let $n \geq 3$. 
The projective spectrum of the graded $k$-algebra $A = S(V(n))/\cl{q}$ gives an irreducible projective variety $X$ with structure sheaf $\O_{X}$ 
and the finitely generated graded $S(V(n))/\cl{q}$-module $W(n)$ provides a coherent sheaf $W(n)^{\sim}$.
Proposition \ref{prop:Minyectivo} may be interpreted as stating that the natural morphism $\alpha : W(n) \rightarrow \Gamma_{\bullet}(W(n)^{\sim})$
is in fact injective (see \cite{Mi1}, p. 115), a result implicitly used in \cite{Mov1}.
We define $M(n) = W(n)[2]$.
For reasons that will be clear later, we will prefer to work with $M(n)$.

Also, from the previous considerations we may derive the following fact mentioned in \cite{MS1}, Example 4, and in \cite{Mov1}:
the sheaf of $\O_{X}$-modules $M(n)^{\sim}$ is isomorphic to the tangent sheaf of $X$.
This is proved as follows.
We first note that the functor $(\place)^{\sim}$ is exact.
Let $i : X \rightarrow \PP(V(n))$ be the inclusion of $X$ in $\PP(V(n))$. 
Since the functor $i^{*}$ is right exact, applying $i^{*}$ to the Euler exact sequence for the projective space
(see \cite{Hart1}, Example 8.20.1, \cite{Huy1}, Prop. 2.4.4), we derive the exact sequence of sheaves of $\O_{X}$-modules
\[     \O_{X} \rightarrow (\O_{X}[1])^{n} \rightarrow i^{*}(\mathcal{T}_{\PP(V(n))}) \rightarrow 0,     \]
where the first map is induced by
\begin{align*}
   A &\rightarrow (A[1])^{n}
   \\
   z &\mapsto (z x_{1}, \dots, z x_{n}).
\end{align*}
We may compare this exact sequence with the one obtained by applying the functor $(\place)^{\sim}$ to the short exact sequence \eqref{eq:euler}.
This implies that $\Omega_{A/k}^{\sim} \simeq i^{*}(\mathcal{T}_{\PP(V(n))})[-2]$.

On the other hand, we can consider the short exact sequence of the normal fiber bundle of a subvariety (see \cite{Mi1}, p. 150)
\[     0 \rightarrow \T_{X} \rightarrow i^{*}(\T_{\PP(V(n))}) \rightarrow \mathcal{N}_{X|\PP(V(n))} \rightarrow 0,     \]
where $\mathcal{N}_{X|\PP(V(n))} = \mathcal{H}om_{\O_{X}}(\I / {\I}^{2} , \O_{X})$ denotes the normal fiber bundle associated to the inclusion
$i : X \rightarrow \PP(V(n))$ and $\I = \cl{q}^{\sim}$ is the sheaf of ideals of $\O_{\PP(V(n))}$ which defines $X$.
In this case, we have the following chain of isomorphisms
\[     \mathcal{N}_{X|\PP(V(n))} = \mathcal{H}om_{\O_{X}}(\I / {\I}^{2} , \O_{X}) \simeq (\Hom_{A} (\cl{q}/\cl{q}^{2},A))^{\sim}
       \simeq (A[2])^{\sim} = \O_{X}[2],     \]
where the penultimate isomorphism is induced by the $A$-linear isomorphism
$\Hom_{A} (\cl{q}/\cl{q}^{2},A) \overset{\simeq}{\rightarrow} A[2]$ 
given by $f \mapsto f(\bar{q})$.
Also, the last map in the previous short exact sequence is induced by $d'_{\mathrm{eu}}$.
Hence, $\T_{X} \simeq W(n)[2]^{\sim} \simeq M(n)^{\sim}$.

The previous results allow us to give a geometrical interpretation of $M(n)$ as the tangent bundle over $X$.
Moreover, since $X$ has a transitive action of $\SO(n)$, it becomes a homogeneous space $\SO(n)/P$, for some parabolic subgroup $P$ 
with Lie algebra $\p \simeq (\so(n-2) \times k) \rtimes V(n-2)$, where $V(n-2)$ is an abelian Lie algebra
within $\so(n-2)$ acts by the standard representation and $k$ acts diagonally.
Both the tangent bundle $M(n)^{\sim}$ and the tautological line bundle $\O_{X}[-1]$ are homogeneous vector bundles associated to
some irreducible representations of lowest weight $-\lambda$ over $\p$, so we may apply the Borel-Weil-Bott theorem
in order to compute $H^{\bullet}(X,E)$, for $E$ equal to $M(n)[i]^{\sim}$ or $(M(n)^{\sim} \otimes_{\O_{X}} M(n)^{\sim})[i]$.
Amazingly, this gives plenty of information about the module $M(n)$: it allows us to prove that the natural morphism
$\alpha : M(n) \rightarrow \Gamma_{\bullet}(M(n)^{\sim})$ is an isomorphism for $n \geq 4$ and also gives a complete description in case $n = 3$,
to compute the groups $\Tor^{\text{\begin{tiny}$S(V(n))$\end{tiny}}}_{\bullet}(M(n),M(n))$, etc.
We shall not pursue these ideas in this article, since 
we shall replace them by shorter algebraic considerations. 
We refer to \cite{Hers1} and references therein for a complete description of the previous geometrical insight.
\qed
\end{remark}

\subsection{\texorpdfstring{Homological properties of $W(n)$}{Homological properties of W(n)}}

\subsubsection{Generalities}


Let $R$ and $S$ be two $k$-algebras, $X$ a right $R$-module, $Y$ an $R$-$S$-bimodule and $Z$ a left $S$-module.
If $Q_{\bullet} \twoheadrightarrow X$ and $P_{\bullet} \twoheadrightarrow Z$ are corresponding projective resolutions,
we can consider the second term of the base-change spectral sequence $E^{2}_{p,q} = \Tor^{\text{\begin{tiny}$R$\end{tiny}}}_{p} 
(X,\Tor^{\text{\begin{tiny}$S$\end{tiny}}}_{q}(Y,Z))$,
given by the filtration by rows of the double complex $C_{p,q} = Q_{q} \otimes_{R} Y \otimes_{S} P_{p}$.
If $Y$ is a flat $R$-module, it converges to $\Tor^{\text{\begin{tiny}$S$\end{tiny}}}_{\bullet}(X \otimes_{R} Y , Z)$ 
(see \cite{Rot1}, Thm. 10.59). 

We shall consider the previous spectral sequence for the case
$R = S(V(n))$, $S = \YM(n)$, $Y = S(V(n))$, $Z = k$ and any graded $S(V(n))$-module $X$, given by 
\[     E^{2}_{p,q} = \Tor^{\text{\begin{tiny}$S(V(n))$\end{tiny}}}_{p}(X, \Tor^{\text{\begin{tiny}$\YM(n)$\end{tiny}}}_{q}(S(V(n)),k)) 
       \Rightarrow \Tor^{\text{\begin{tiny}$\YM(n)$\end{tiny}}}_{p+q}(X,k).     \]
Notice that $\Tor^{\text{\begin{tiny}$\YM(n)$\end{tiny}}}_{p+q}(X,k) \simeq H_{p+q}(\ym(n),X)$. 

\begin{remark}
\label{obs:diagram}
We choose $Q_{\bullet} = \Lambda^{\bullet} V(n) \otimes X \otimes S(V(n)) = C_{\bullet}(V(n), X \otimes S(V(n)))$.
It is easily verified that it is indeed a projective resolution of the right $S(V(n))$-module $X$.
On the other side, we choose $P_{\bullet}$ as the Koszul resolution \eqref{eq:complejoKoszulcompleto} of the left $\YM(n)$-module $k$.

In this case, the base-change spectral sequence is given by the filtration by rows of the double complex
\begin{equation}
\label{eq:complejodoblesucesp}
       C_{p,q} = Q_{q} \otimes_{S(V(n))} S(V(n)) \otimes_{\YM(n)} P_{p}
       \simeq \Lambda^{q} V(n) \otimes X \otimes C_{p}(\YM(n),S(V(n))),
\end{equation}
with vertical differential $d_{\bullet}^{\mathrm{CE}} \otimes \id_{\YM(n)^{!}_{\bullet}}$, where we consider the action of $V(n)$ on
$X \otimes S(V(n))$, and with horizontal differential
$\id_{\Lambda^{\bullet} V(n) \otimes X} \otimes d_{\bullet}$, with $d_{\bullet}$ the differential of
$C_{\bullet}(\YM(n),S(V(n)))$.
\qed
\end{remark}

On the other hand, since $\Tor^{\text{\begin{tiny}$\YM(n)$\end{tiny}}}_{q}(S(V(n)),k) \simeq H_{q}(\ym(n),S(V(n)))$ and using Proposition 3.5 
of~\cite{HS1} and Proposition \ref{prop:isow(n)}, we have the equivariant $S(V(n))$-linear isomorphisms
\[     \Tor^{\text{\begin{tiny}$\YM(n)$\end{tiny}}}_{\bullet}(S(V(n)),k) \simeq \begin{cases}
                                                       k, &\text{if $\bullet = 0$,}
                                                       \\
                                                       W(n), &\text{if $\bullet = 1$,}
                                                       \\
                                                       0, &\text{if not.}
                                                    \end{cases}
\]
Thus, our spectral sequence has only two nonzero rows
\begin{equation}
\label{eq:isosucesp1}
     E^{2}_{p,0} \simeq \Tor^{\text{\begin{tiny}$S(V(n))$\end{tiny}}}_{p}(X , k) \simeq H_{p}(V(n) , X)     
\end{equation}
and
\begin{equation}
\label{eq:isosucesp2}
     E^{2}_{p,1} \simeq \Tor^{\text{\begin{tiny}$S(V(n))$\end{tiny}}}_{p}(X , W(n)) \simeq H_{p}(V(n), X \otimes W(n) ).     
\end{equation}
The last isomorphism follows from the usual fact that $\Tor^{\text{\begin{tiny}$\U(\g)$\end{tiny}}}_{\bullet}(M,N) \simeq H_{\bullet}(\g, M \otimes N)$ 
(see \cite{CE1}, Chap XI, Prop. 9.2). 

Furthermore, since the spectral sequence has only two rows, it gives a long exact sequence of the form (see \cite{Rot1}, Prop. 10.28) 
\begin{equation}
\label{eq:longexacseq}
\begin{split}
   &\rightarrow H_{p}(\ym(n),X) \rightarrow 
   H_{p}(V(n),X) \rightarrow 
   H_{p-2}(V(n),X \otimes W(n))
   \\
   &\rightarrow H_{p - 1}(\ym(n),X) 
   \rightarrow H_{p-1}(V(n),X)
   \rightarrow H_{p-3}(V(n),X \otimes W(n))
   \rightarrow \dots
   \\
   \dots
   &\rightarrow H_{2}(\ym(n),X) 
   \rightarrow H_{2}(V(n),X)
   \rightarrow H_{0}(V(n),X \otimes W(n))
   \rightarrow 
   \\ 
   &\rightarrow H_{1}(\ym(n),X) 
   \rightarrow H_{1}(V(n),X) 
   \rightarrow 0,
\end{split}
\end{equation}
and the isomorphism 
$H_{0}(\ym(n),X) \simeq H_{0}(V(n),X)$.

However, since $\YM(n)$ has global dimension equal to $3$, $H_{p}(\ym(n),X) = 0$, if $p \geq 4$.
This implies that
\begin{equation}
\label{eq:imp}
   H_{p+1}(V(n),X) \simeq H_{p-1}(V(n),X \otimes W(n)),
\end{equation}
for $p \geq 4$.

\begin{remark}
\label{rem:accionsvntor}
As for the long exact sequence, we may write the previous spectral sequence using the identification 
$\mathrm{Tor}^{\text{\begin{tiny}$\U(\g)$\end{tiny}}}_{\bullet}(M,N) \simeq H_{\bullet}(\g,M \otimes N)$. 
In this case, it is easy to see that the spectral sequence coincides with the Hochschild-Serre spectral sequence 
$H_{p}(V(n),H_{q}(\tym(n),X)) \Rightarrow H_{p+q}(\ym(n),X)$, for $X$ a $V(n)$-module. 
However, we point out some differences. 
First, the Hochschild-Serre spectral sequence may also be used when $X$ is a $\ym(n)$-module. 
Second, the base-change spectral sequence has further structure since it lives in the category of (graded) $S(V(n))$-modules, 
whereas $\mathrm{Tor}^{\text{\begin{tiny}$S(V(n))$\end{tiny}}}_{\bullet}(X,W(n)) \simeq H_{\bullet}(V(n),X \otimes W(n))$ only holds as $k$-modules 
(and also as homogeneous $\so(n)$-modules when $X$ is an equivariant $S(V(n))$-module). 
That the base-change spectral sequence can be considered in the category of (graded) $S(V(n))$-modules comes from the fact that 
the considered modules are provided with an action of $S(V(n))$ which commutes with all other actions. 
\qed
\end{remark}

We shall be mostly interested in the case that $X$ is given by the equivariant $S(V(n))$-module $W(n)^{\otimes i}$, for $i \in \NN_{0}$, 
which is an equivariant $S(V(n))$-module provided with the diagonal action. 
It is readily verified that, if $X = W(n)^{\otimes i}$, all previously considered morphisms are in fact 
homogeneous of degree $0$ and $\so(n)$-linear. 
For the rest of this subsection, all morphisms will also be $\so(n)$-linear, unless we say the opposite. 

The long exact sequence \eqref{eq:longexacseq} for $X=W(n)^{\otimes i}$ becomes 
\begin{equation}
\label{eq:longexacseqforW}
\begin{split}
   &\rightarrow H_{p}(\ym(n),W(n)^{\otimes i}) \rightarrow 
   H_{p}(V(n),W(n)^{\otimes i}) \rightarrow 
   H_{p-2}(V(n),W(n)^{\otimes (i+1)})
   \\
   &\rightarrow H_{p - 1}(\ym(n),W(n)^{\otimes i}) 
   \rightarrow H_{p-1}(V(n),W(n)^{\otimes i})
   \rightarrow H_{p-3}(V(n),W(n)^{\otimes (i+1)})
   \rightarrow \dots
   \\
   \dots
   &\rightarrow H_{2}(\ym(n),W(n)^{\otimes i}) 
   \rightarrow H_{2}(V(n),W(n)^{\otimes i})
   \rightarrow H_{0}(V(n),W(n)^{\otimes (i+1)})
   \rightarrow 
   \\ 
   &\rightarrow H_{1}(\ym(n),W(n)^{\otimes i}) 
   \rightarrow H_{1}(V(n),W(n)^{\otimes i}) 
   \rightarrow 0,
\end{split}
\end{equation}
and we obtain the isomorphism 
$H_{0}(\ym(n),W(n)^{\otimes i}) \simeq H_{0}(V(n),W(n)^{\otimes i})$.

Also, the isomorphisms \eqref{eq:imp} tell us that 
\begin{equation}
\label{eq:impforW}
   H_{p+1}(V(n),W(n)^{\otimes i}) \simeq H_{p-1}(V(n),W(n)^{\otimes (i+1)}),
\end{equation}
for $p \geq 4$ and $i \in \NN_{0}$.

In fact, a stronger statement relating these homology groups holds. 
\begin{theorem}
\label{teo:exacta}
Let $n \geq 3$ and $i \geq 1$. 
There is a long exact sequence of $\so(n)$-modules and homogeneous $\so(n)$-equivariant morphisms 
\begin{align*}
   0 &\rightarrow H_{3}(V(n),W(n)^{\otimes i})
   \overset{S'_{i}}{\rightarrow} H_{1}(V(n),W(n)^{\otimes (i+1)})
   \overset{B'_{i}}{\rightarrow} H_{2}(\ym(n),W(n)^{\otimes i})
   \overset{I'_{i}}{\rightarrow} 
   \\
   &\rightarrow H_{2}(V(n),W(n)^{\otimes i})
   \overset{S_{i}}{\rightarrow} H_{0}(V(n),W(n)^{\otimes (i+1)})
   \overset{B_{i}}{\rightarrow} H_{1}(\ym(n),W(n)^{\otimes i})
   \overset{I_{i}}{\rightarrow} 
   \\
   &\rightarrow H_{1}(V(n),W(n)^{\otimes i}) \rightarrow 0,
\end{align*}
and a collection of homogeneous $\so(n)$-equivariant isomorphisms 
$H_{0}(\ym(n),W(n)^{\otimes i}) \simeq H_{0}(V(n),W(n)^{\otimes i})$ and
\[
\xymatrix@R-20pt@C-15pt
{
   H_{0}(V(n),W(n)) \simeq \Lambda^{2} V(n),
   &
   H_{1}(V(n),W(n)) \simeq \Lambda^{3} V(n) \oplus V(n)[-2],
   \\
   H_{2}(V(n),W(n)) \simeq \Lambda^{4} V(n) \oplus k[-4],
   &
   H_{p}(V(n),W(n)^{\otimes i}) \simeq \Lambda^{p+2 i} V(n), \hskip 1mm \text{for $p \geq 2$,}
}
\]
where in the last isomorphism we exclude the case $p = 2$ and $i = 1$.
\end{theorem}
\begin{proof}
Let us first suppose that $i = 0$, so $W(n)^{\otimes i} \simeq k$. 
From the fact that $H_{4}(\ym(n),W(n)^{\otimes i})$ vanishes, we have the following exact sequence
\begin{align*}
   0 &\rightarrow H_{4}(V(n),k) \rightarrow H_{2}(V(n),W(n))
   \rightarrow H_{3}(\ym(n),k) \rightarrow H_{3}(V(n),k)
   \rightarrow H_{1}(V(n),W(n)) \rightarrow 
   \\
   &\rightarrow H_{2}(\ym(n),k)
   \rightarrow H_{2}(V(n),k)
   \rightarrow H_{0}(V(n),W(n)) \rightarrow H_{1}(\ym(n),k)
   \rightarrow H_{1}(V(n),k) \rightarrow 0.
\end{align*}
On the other hand, the isomorphisms $H_{p}(V(n),k) \simeq \Lambda^{p} V(n)$, obtained from the Chevalley-Eilenberg complex,
and 
\[     H_{0}(\ym(n),k) \simeq k, \hskip 0.2cm  H_{1}(\ym(n),k) \simeq V(n), \hskip 0.2cm H_{2}(\ym(n),k) \simeq V(n)[-2], \hskip 0.2cm
       H_{3}(\ym(n),k) \simeq k[-4],     \]
which follow from the Koszul complex \eqref{eq:complejohomologiayangmills}, imply that
$H_{0}(V(n),W(n)) \simeq \Lambda^{2} V(n)$ and the following short exact sequences
\[     0 \rightarrow \Lambda^{4} V(n) \rightarrow H_{2}(V(n),W(n)) \rightarrow k[-4] \rightarrow 0     \]
and
\[     0 \rightarrow \Lambda^{3} V(n) \rightarrow H_{1}(V(n),W(n)) \rightarrow V(n)[-2] \rightarrow 0.     \]
Note that we have used that the maps $H_{2}(\ym(n),k) \rightarrow H_{2}(V(n),k)$ and $H_{3}(\ym(n),k) \rightarrow H_{3}(V(n),k)$ vanish,
since they are $\so(n)$-linear maps between different irreducible representations.

Since $H_{p}(\ym(n),k) = 0$, for all $p \geq 4$ and $H_{p}(V(n),k) \simeq \Lambda^{p} V(n)$, there are isomorphisms 
\begin{equation}
\label{eq:comphom}
     H_{p}(V(n),W(n)) \simeq \Lambda^{p+2} V(n), \hskip 1mm \text{for all $p \geq 3$.}
\end{equation}

Let us now assume that $i \geq 1$.

We shall first prove the following proposition.
\begin{proposition}
\label{prop:annh3yangmills}
If $n \geq 3$ and $j \in \NN$, 
then $H_{p}(V(n),W(n)^{\otimes j}) = 0$, for $p \geq n$, and, in consequence, $H_{3}(\ym(n),W(n)^{\otimes j}) = 0$, for $j \in \NN$. 
\end{proposition}
\begin{proof}
The second statement follows directly from the first one, since by Proposition \ref{prop:annhyangmills},
there is an isomorphism $H_{n}(V(n),W(n)^{\otimes j}) \simeq H_{3}(\ym(n),W(n)^{\otimes j})$.

Let us prove the first one, proceeding by induction on $j$. 

Assume that $j=1$.
In this case, using \eqref{eq:imp} for $i=0$ we obtain that there is an isomorphism   
$H_{p+1} (V(n),k) \simeq H_{p-1}(V(n),W(n))$, for $p \geq 4$, so
$H_{p}(V(n),W(n)) \simeq \Lambda^{p+2} V(n)$, for $p \geq 3$.
Then $H_{p}(V(n),W(n)) = 0$, for $p \geq n$. 

Supposing that the proposition holds for $j-1$,
we will prove it for $j$.
In this case, the isomorphism \eqref{eq:imp} for $i = j-1$ implies that
$H_{p+1} (V(n),W(n)^{\otimes (j-1)}) \simeq H_{p-1}(V(n),W(n)^{\otimes j})$, for $p \geq 4$.
Hence $H_{p}(V(n),W(n)^{\otimes j}) \simeq H_{p+2} (V(n),W(n)^{\otimes (j-1)}) = 0$, for $p \geq n$.
The proposition is then proved.
\end{proof}

As a consequence of the previous proposition, $H_{3}(\ym(n),W(n)^{\otimes i}) = 0$ if $i \geq 1$, and we obtain
an exact sequence 
\begin{align*}
   0 &\rightarrow H_{3}(V(n),W(n)^{\otimes i})
   \rightarrow H_{1}(V(n),W(n)^{\otimes (i+1)}) 
   \rightarrow H_{2}(\ym(n),W(n)^{\otimes i})
   \rightarrow H_{2}(V(n),W(n)^{\otimes i}) 
   \\
   &\rightarrow H_{0}(V(n),W(n)^{\otimes (i+1)}) 
   \rightarrow H_{1}(\ym(n),W(n)^{\otimes i})
   \rightarrow H_{1}(V(n),W(n)^{\otimes i}) 
   \rightarrow 0.
\end{align*}
The vanishing of $H_{\bullet}(\ym(n),W(n)^{\otimes i})$, for $\bullet \geq 3$, yields the isomorphisms 
$H_{p}(V(n),W(n)^{\otimes (i+1)}) \simeq H_{p+2}(V(n),W(n)^{\otimes i})$, for $p \geq 2$. 
By induction, it turns out that $H_{p}(V(n),W(n)^{\otimes j})$ is isomorphic to $H_{p+2(j-1)}(V(n),W(n))$, for all $p \geq 2$, $j \geq 2$. 
Using \eqref{eq:comphom}, we see that 
$H_{p}(V(n),W(n)^{\otimes j}) \simeq \Lambda^{p+2 j} V(n)$,
in case $p \geq 2$ and $j \geq 2$.
We may summarize the previous information as follows: 
\[     H_{p}(V(n),W(n)^{\otimes i}) \simeq \Lambda^{p+2 i} V(n)     \]
if $p \geq 2$ and $i \geq 1$, except in case $p = 2$ and $i = 1$.
This completes the proof of the theorem. 
\end{proof}

The following remark describes some of the morphisms appearing in the exact sequence of Theorem \ref{teo:exacta}.
\begin{remark}
\label{obs:morfismossucesp}
If $i \geq 1$, the maps $S'_{i}$ and $S_{i}$ in the theorem coincide with the differentials $d^{2}_{3,0}$ and $d^{2}_{2,0}$
of the second term of the base-change spectral sequence, resp.
and the isomorphisms $H_{p}(V(n),W(n)^{\otimes i}) \rightarrow H_{p-2}(V(n),W(n)^{\otimes (i+1)})$
for $p \geq 4$ coincide with the differentials $d^{2}_{p,0}$.

If $i = 0$, the injections $\Lambda^{4} V(n) \hookrightarrow H_{2}(V(n),W(n))$
and $\Lambda^{3} V(n) \hookrightarrow H_{1}(V(n),W(n))$
coincide with $d^{2}_{4,0}$ and $d^{2}_{3,0}$, respectively.
Also, the family of isomorphisms $H_{p}(V(n),W(n)^{\otimes i}) \rightarrow H_{p-2}(V(n),W(n)^{\otimes (i+1)})$
for $p \geq 5$ coincide with differentials $d^{2}_{p,0}$.

On the other hand, the total complex of the double complex \eqref{eq:complejodoblesucesp} for $X = W(n)^{\otimes i}$, which may be rewritten as 
\[     {}^{i}C_{p,q} = \Lambda^{q} V(n) \otimes W(n)^{\otimes i} \otimes C_{p}(\YM(n),S(V(n))),     \]
is quasi-isomorphic to $C_{\bullet}(\YM(n),W(n)^{\otimes i})$, and the quasi-isomoprhism is given by the map 
\begin{equation}
\label{eq:quasiisosucesp}
   \mathrm{Tot}({}^{i}C_{\bullet,\bullet}) \rightarrow C_{\bullet}(\YM(n),W(n)^{\otimes i})
\end{equation}
induced by the projection
\begin{equation}
\label{eq:quasiisosucesp2}
   {}^{i}C_{\bullet,0} = W(n)^{\otimes i} \otimes C_{\bullet}(\YM(n),S(V(n))) \rightarrow C_{\bullet}(\YM(n),W(n)^{\otimes i})
\end{equation}
given by the action of $S(V(n))$ on $W(n)^{\otimes i}$, \textit{i.e.} $w \otimes z \otimes v \mapsto w z \otimes v$,
if $z \otimes v$ belongs to $C_{1}(\YM(n),S(V(n)))$ or $C_{2}(\YM(n),S(V(n)))$ ($w \in W(n)^{\otimes i}$,
$z \in S(V(n))$ and $v \in V(n)$); and $w \otimes z \mapsto w z$,
if $z$ belongs to $C_{0}(\YM(n),S(V(n)))$ or $C_{3}(\YM(n),S(V(n)))$ ($w \in W(n)^{\otimes i}$ and $z \in S(V(n))$). 
From this quasi-isomorphism, the maps $B_{i}$ and $B'_{i}$ appearing in Theorem \ref{teo:exacta} can be described as follows. 
As it is usual, identifying $E^{2}_{p,q}$ with subquotients of $\mathrm{Tot}({}^{i}C_{\bullet,\bullet})$, 
the mappings $B_{i}$ and $B'_{i}$ are induced by the composition of the inclusion and the quasi-isomorphism \eqref{eq:quasiisosucesp}
(see \cite{Rot1}, Thm. 10.31). 

Finally, let us describe the morphism $I_{i} : H_{1}(\ym(n),W(n)^{\otimes i}) \rightarrow H_{1}(V(n),W(n)^{\otimes i})$.
In order to do so, we recall that, if $E$ is a $V(n)$-module, the Lie algebra morphism
$\pi : \ym(n) \rightarrow V(n)$ given by the canonical projection induces a morphism of complexes
$C_{\bullet}(\ym(n),E) \rightarrow C_{\bullet}(V(n),E)$
.
This in turn induces a morphism in the homology groups $H_{\bullet}(\ym(n),E) \rightarrow H_{\bullet}(V(n),E)$.
In particular, there is a map $H_{1}(\ym(n),E) \rightarrow H_{1}(V(n),E)$, induced by $\id_{E} \otimes \pi$.

On the other hand, from the comparison of resolutions achieved at the end of Subsection \ref{subsec:homandcohom},
we see that the map $\id_{E} \otimes \mathrm{inc} : E \otimes V(n) \rightarrow E \otimes \ym(n)$ induces an isomorphism
$H_{1}(\ym(n),E) \rightarrow H_{1}(\ym(n),E)$.
Therefore, if we choose as representatives of the homology $H_{1}(\ym(n),E)$ the cycles of $C_{1}(\YM(n),E)$,
and as representatives of the homology $H_{1}(V(n),E)$ the cycles of $C_{1}(V(n),E)$, the mapping
$H_{1}(\ym(n),E) \rightarrow H_{1}(V(n),E)$ induced by the identity $\id_{E \otimes V(n)}$ coincides with the one induced by $\id_{E} \otimes \pi$.
By \cite{Rot1}, Thm. 10.31, 
if we choose as representatives of the homology $H_{1}(\ym(n),W(n)^{\otimes i})$
the cycles of $C_{1}(\YM(n),W(n)^{\otimes i})$, and as representatives of the homology $H_{1}(V(n),W(n)^{\otimes i})$
the cycles of $C_{1}(V(n),W(n)^{\otimes i})$, the map $I_{i}$ is induced by the identity $\id_{W(n)^{\otimes i} \otimes V(n)}$.
This shows that $I_{i}$ also coincides with the morphism induced by $\id_{W(n)^{\otimes i}} \otimes \pi$.
\qed
\end{remark}

\begin{proposition}
\label{prop:s1inyectivo}
Let $n \geq 3$. 
The morphism $S_{1} : H_{2}(V(n),W(n)) \rightarrow H_{0}(V(n),W(n)^{\otimes 2})$ is an injection. 
By exactness of the sequence of Theorem \ref{teo:exacta}, $B'_{1}$ is surjective and $I'_{1} = 0$.
\end{proposition}
\begin{proof}
We proceed by inspection on the morphisms at the level of the double complex which defines the base-change spectral sequence. 

In the first place, from the double complex \eqref{eq:complejodoblesucesp}, Remark \ref{obs:morfismossucesp} and
standard computations on a second term spectral sequence, 
the morphism $\Lambda^{4} V(n) = H_{4}(V(n) , k) \hookrightarrow H_{2}(V(n),W(n))$ is induced by
\[     x_{i_{1}} \wedge x_{i_{2}} \wedge x_{i_{3}} \wedge x_{i_{4}} \mapsto
       \sum_{\sigma \in \SSS_{4}} \epsilon(\sigma) x_{i_{\sigma(1)}} \wedge x_{i_{\sigma(2)}} \otimes [x_{i_{\sigma(3)}}, x_{i_{\sigma(4)}}],      \]
where the image element is a cycle in $\Lambda^{2} V(n) \otimes W(n)$.

Also, the element
\[     \sum_{1 \leq i < j \leq n} x_{i} \wedge x_{j} \otimes [x_{i},x_{j}] \in \Lambda^{2} V(n) \otimes W(n)     \]
is a non trivial cycle, since it is the image of the cycle
\[     (0, \sum_{1 \leq i < j \leq n} x_{i} \wedge x_{j} \otimes (x_{i} \otimes x_{j} - x_{j} \otimes x_{i}),
       \sum_{i = 1}^{n} x_{i} \otimes 1 \otimes x_{i}, - 1 \otimes 1 \otimes 1)     \]
in the degree $3$ component of the total complex of the double complex \eqref{eq:complejodoblesucesp} for $i=0$.
This element does not vanish in the homology of the total complex, since $- 1 \otimes 1 \otimes 1$ cannot be
in the image of the vertical differential by degree reasons.
Moreover, the image of this element in $H_{3}(\ym(n),k) \simeq k$ is $- 1$.

We conclude that a basis for $H_{2}(V(n),W(n))$ is given by the set of classes of the following collection of cycles
\begin{equation}
\label{eq:basish2}
     \Big\{ \sum_{\sigma \in \SSS_{4}} \epsilon(\sigma) x_{i_{\sigma(1)}} \wedge x_{i_{\sigma(2)}} \otimes [x_{i_{\sigma(3)}}, x_{i_{\sigma(4)}}]
       : 1 \leq i_{1} < i_{2} < i_{3} < i_{4} \leq n \Big\}
       \cup \Big\{ \sum_{1 \leq i < j \leq n} x_{i} \wedge x_{j} \otimes [x_{i},x_{j}] \Big\}.     
\end{equation}

By standard computations on a second term spectral sequence, 
the morphism from $H_{2}(V(n),W(n))$ to $H_{0}(V(n),W(n)^{\otimes 2})$ 
satisfies that
\begin{align*}
   \sum_{\sigma \in \SSS_{4}} \epsilon(\sigma) x_{i_{\sigma(1)}} \wedge x_{i_{\sigma(2)}} \otimes [x_{i_{\sigma(3)}}, x_{i_{\sigma(4)}}]
   &\mapsto
   \sum_{\sigma \in \SSS_{4}} \epsilon(\sigma) [x_{i_{\sigma(1)}}, x_{i_{\sigma(2)}}] \otimes [x_{i_{\sigma(3)}}, x_{i_{\sigma(4)}}],
   \\
   \sum_{1 \leq i < j \leq n} x_{i} \wedge x_{j} \otimes [x_{i},x_{j}] &\mapsto \sum_{1 \leq i < j \leq n} [x_{i},x_{j}] \otimes [x_{i},x_{j}],
\end{align*}
for all $1 \leq i_{1} < i_{2} < i_{3} < i_{4} \leq n$ and 
it is not hard to check that these image elements are linearly independent in $\Lambda^{2} V(n) \otimes \Lambda^{2} V(n)$.

Since $W(n)$ is a graded $S(V(n))$-module and $W(n)_{2} = \Lambda^{2} V(n)$, it turns out that
$(W(n) \otimes W(n))_{4} = \Lambda^{2} V(n) \otimes \Lambda^{2} V(n)$ is the non trivial homogeneous component of lowest degree.
Notice that the image of the basis \eqref{eq:basish2} of $H_{2}(V(n),W(n))$ is in fact included in $(W(n) \otimes W(n))_{4}$ 
and this implies that $S_{1}$ is an injection. 
\end{proof}

Let us suppose that $n \geq 3$ and $i \geq 2$. 
The action of $q = \sum_{i=1}^{n} x_{i} \otimes x_{i}$ on $W(n)^{\otimes i}$ is given by the coproduct of $\YM(n)$, 
so in order to make it explicit we will compute $\Delta^{(i)}(q)$, which is of the form
\[     \Delta^{(i)}(q) = \sum_{p=0}^{i-1} 1_{\YM(n)}^{\otimes p} \otimes q \otimes 1_{\YM(n)}^{\otimes (i-p-1)}
                       + 2 \underset{\text{\begin{tiny}$\begin{matrix} p,q \in \NN_{0}
                                                                                 \\  p + q \leq i-2
                                                                                 \end{matrix}$\end{tiny}}}{\sum}
                        1_{\YM(n)}^{\otimes p} \otimes x_{l} \otimes 1_{\YM(n)}^{\otimes q}
                        \otimes x_{l} \otimes 1_{\YM(n)}^{\otimes (i-p-q-2)}.
\]

\begin{proposition}
\label{prop:wni}
If $n \geq 3$ and $i \geq 2$, then $q$ is nonzerodivisor on the $S(V(n))$-module $W(n)^{\otimes i}$.
\end{proposition}
\begin{proof}
We recall that $\Tor_{p}^{\text{\begin{tiny}$S(V(n))$\end{tiny}}} (Y,Y') \simeq \Tor_{p}^{\text{\begin{tiny}$S(V(n))$\end{tiny}}} (k, Y \otimes Y')$, where $Y \otimes Y'$ has the diagonal action.

We first notice that $(S(V(n))/\cl{q})^{\otimes i}$ is an algebra with the usual structure 
and that the hypothesis on $n$ implies that $q \in S(V(n))$ is irreducible 
and hence $S(V(n))/\cl{q}$ is a domain. 
Since $k$ is algebraically closed and $S(V(n))/\cl{q}$ is a finitely generated integral domain, 
$(S(V(n))/\cl{q})^{\otimes i}$ is in fact a finitely generated integral domain (see \cite{Ch-L1}, Thm. 14.1.5). 

The action of $q \in S(V(n))$ on $(S(V(n))/\cl{q})^{\otimes i}$ is given by multiplication by the nonzero element
\[     2 \underset{\text{\begin{tiny}$\begin{matrix} p,q \in \NN_{0}
                                                                                 \\  p + q \leq i-2
                                                                                 \end{matrix}$\end{tiny}}}{\sum}
                        1_{S(V(n))/\cl{q}}^{\otimes p} \otimes x_{l} \otimes 1_{S(V(n))/\cl{q}}^{\otimes q}
                        \otimes x_{l} \otimes 1_{S(V(n))/\cl{q}}^{\otimes (i-p-q-2)} \in (S(V(n))/\cl{q})^{\otimes i},     \]
and this shows that $q$ is a nonzerodivisor on $(S(V(n))/\cl{q})^{\otimes i}$.

We claim that $\Tor_{p}^{\text{\begin{tiny}$S(V(n))$\end{tiny}}} ((S(V(n))/\cl{q})^{\otimes a},k)$ vanishes for 
$p \geq 2$ and $a \in \NN$. 
In order to prove this result, we proceed as follows. 
The case $a = 1$ has already been analyzed, since the projective resolution \eqref{eq:resA} implies that $S(V(n))/\cl{q}$ has projective dimension $1$. 
If $a > 1$, then 
\[     \Tor_{p}^{\text{\begin{tiny}$S(V(n))$\end{tiny}}} ((S(V(n))/\cl{q})^{\otimes a},k) \simeq \Tor_{p}^{\text{\begin{tiny}$S(V(n))$\end{tiny}}} ((S(V(n))/\cl{q})^{\otimes (a-1)},S(V(n))/\cl{q}),     \]
and the latter homology group vanishes, since $S(V(n))/\cl{q}$ has projective dimension $1$. 
Our claim implies that the bounded below graded $S(V(n))$-module $(S(V(n))/\cl{q})^{\otimes a}$ has projective dimension $1$ for all $a \in \NN$ (see \cite{Ber2}, Prop. 2.3 and Cor. 2.4). 

We will next prove that
\begin{equation}
\label{eq:iso1}
H_{n-1}(V(n),(S(V(n))/\cl{q})^{\otimes a} \otimes W(n)^{\otimes b}) = 0
\end{equation}
and
\begin{equation}
\label{eq:iso0}
H_{n}(V(n),(S(V(n))/\cl{q})^{\otimes a} \otimes W(n)^{\otimes b}) = 0,
\end{equation}
for all $a,b \in \NN_{0}$ such that $a + b = i$.
The case $a=0$ follows directly from Theorem \ref{teo:exacta}. 
Let us now assume that $a \in \NN$. 
Since
\begin{align*}
       H_{p}(V(n),(S(V(n))/\cl{q})^{\otimes a} \otimes W(n)^{\otimes b})
       &\simeq \Tor^{\text{\begin{tiny}$S(V(n))$\end{tiny}}}_{p} (k , (S(V(n))/\cl{q})^{\otimes a} \otimes W(n)^{\otimes b})
       \\
       &\simeq \Tor^{\text{\begin{tiny}$S(V(n))$\end{tiny}}}_{p} ((S(V(n))/\cl{q})^{\otimes a} , W(n)^{\otimes b}),
\end{align*}
it suffices to show that this last homology group vanishes for $p = n - 1$ and $p = n$.
This is indeed the case, since, as explained before, the projective dimension of the graded $S(V(n))$-module $(S(V(n))/\cl{q})^{\otimes a}$ is~$1$. 

In view of the form of the projective resolution \eqref{eq:resA},
$\mathrm{ann}_{W(n)^{\otimes i}[-2]} (q)$ is isomorphic to the homology group $\Tor_{1}^{\text{\begin{tiny}$S(V(n))$\end{tiny}}} (S(V(n))/\cl{q},W(n)^{\otimes i})$, 
so it suffices to show that this last group vanishes in order to prove the proposition.
In fact we will prove a stronger statement asserting that, for every $a,b \in \NN_{0}$ such that $i = a + b \geq 2$,
the homology group
\[     \Tor_{1}^{\text{\begin{tiny}$S(V(n))$\end{tiny}}} (S(V(n))/\cl{q},(S(V(n))/\cl{q})^{\otimes a} \otimes W(n)^{\otimes b}) 
\simeq \mathrm{ann}_{(S(V(n))/\cl{q}^{\otimes a} \otimes W(n)^{\otimes b})[-2]} (q)     \]
vanishes, holds.
The case $a = i$ and $b = 0$ follows directly since $q$ is a nonzerodivisor on $(S(V(n))/\cl{q})^{\otimes i}$, as previously stated.

Let us now assume that $i \geq 2$.
We shall proceed by induction on $b$.
The case $b = 0$ (so $a = i$) has already been proved. 

Let us assume that $\Tor_{1}^{\text{\begin{tiny}$S(V(n))$\end{tiny}}} (S(V(n))/\cl{q},(S(V(n))/\cl{q})^{\otimes (i-j)} \otimes W(n)^{\otimes j})$ vanishes for $j = 0, \dots, b-1 < i$, where $b \geq 1$. 
We shall prove that it also vanishes for $j = b \leq i$.

Since
\[     \Tor_{1}^{\text{\begin{tiny}$S(V)$\end{tiny}}} \Big(\tfrac{S(V)}{\cl{q}},
                 \big(\tfrac{S(V)}{\cl{q}}\big)^{\otimes (i-b)} \otimes W^{\otimes b}\Big)
       \simeq \Tor_{1}^{\text{\begin{tiny}$S(V)$\end{tiny}}} \Big(\big(\tfrac{S(V)}{\cl{q}}\big)^{\otimes (i+1-b)} \otimes W^{\otimes (b-1)} , W\Big),     \]
where we have omitted the index $n$, it suffices to prove that the last homology group vanishes. 

By the inductive hypothesis, $q$ is a nonzerodivisor on $(S(V(n))/\cl{q})^{\otimes (i+1-b)} \otimes W(n)^{\otimes (b-1)}$, 
so Proposition \ref{prop:annhyangmills} implies that 
\[     H_{2}\Big(\ym(n),\big(\tfrac{S(V)}{\cl{q}}\big)^{\otimes (i+1-b)} \otimes W(n)^{\otimes (b-1)}\Big) 
       \simeq H_{n-1}\Big(V(n),\big(\tfrac{S(V)}{\cl{q}}\big)^{\otimes (i+1-b)} \otimes W(n)^{\otimes (b-1)}\Big).     \] 
Also, the same proposition tells us that 
\[     H_{3}\Big(\ym(n),\big(\tfrac{S(V)}{\cl{q}}\big)^{\otimes (i+1-b)} \otimes W(n)^{\otimes (b-1)}\Big) 
       \simeq H_{n}\Big(V(n),\big(\tfrac{S(V)}{\cl{q}}\big)^{\otimes (i+1-b)} \otimes W(n)^{\otimes (b-1)}\Big).     \]
Isomorphisms \eqref{eq:iso1} and \eqref{eq:iso0} imply that the previous homology groups vanish. 
Using the previous computations and the exact sequence \eqref{eq:longexacseq} for 
$X = (S(V(n))/\cl{q})^{\otimes (i+1-b)} \otimes W(n)^{\otimes (b-1)}$, we conclude that 
\[     H_{1}(V(n),X \otimes W(n)) 
       \simeq H_{3}(V(n),X).     \]
Isomorphisms \eqref{eq:imp} tell us that $H_{3}(V(n),X)$ is isomorphic to 
$H_{3+2(b-1)}(V(n),(S(V(n))/\cl{q})^{\otimes i})$, 
which vanishes, since $(S(V(n))/\cl{q})^{\otimes i}$ has projective dimension $1$. 
The proposition is thus proved.
\end{proof}

From the proposition we obtain the following corollary.
\begin{corollary}
\label{coro:anulacion}
If $n \geq 3$ and $i \geq 2$, then $H_{2}(\ym(n),W(n)^{\otimes i})=0$.
\end{corollary}
\begin{proof}
Since $q$ is a nonzerodivisor on $W(n)^{\otimes i}$ by Proposition \ref{prop:wni} and
$H_{n-1}(V(n),W(n)^{\otimes i}) = 0$ by Theorem \ref{teo:exacta}, Proposition \ref{prop:annhyangmills}
yields that $H_{2}(\ym(n),W(n)^{\otimes i}) = 0$.
\end{proof}

\subsubsection{\texorpdfstring{The minimal projective resolution of $W(n)$ and other results}
{The minimal projective resolution of W(n) and other results}}

We define $M(n) = W(n)[2]$. 
We shall focus ourselves on $M(n)$ instead of $W(n)$ because of the following homological properties (see also Remark \ref{rem:geown}).
\begin{proposition}
The $\NN_{0}$-graded $S(V(n))$-module $M(n)$ is Koszul, \textit{i.e.} $\Tor^{\text{\begin{tiny}$S(V(n))$\end{tiny}}}_{p}(M(n),k)$ 
is concentrated in degree $p$, for all $p \in \NN_{0}$.
\end{proposition}
\begin{proof}
Let us first consider $n = 2$. 
Since $W(2) \simeq k[-2]$ (see Proposition \ref{prop:W(2)}) and $S(V(2))$ is a Koszul algebra 
(where the Chevalley-Eilenberg resolution of the Lie $V(n)$-module $k$ coincides with the Koszul resolution), 
$M(2)$ is a Koszul $S(V(2))$-module. 

Let now $n \geq 3$. 
Taking into account that $\Tor^{\text{\begin{tiny}$S(V(n))$\end{tiny}}}_{p}(M(n),k) \simeq H_{p}(V(n),M(n))$, 
the proposition follows from Theorem \ref{teo:exacta} since it implies that
\begin{align*}
   H_{0}(V(n),M(n)) &\simeq H_{0}(V(n),W(n))[2] \simeq (\Lambda^{2} V(n))[2],
   \\
   H_{1}(V(n),M(n)) &\simeq H_{1}(V(n),W(n))[2] \simeq (\Lambda^{3} V(n))[2] \oplus V(n),
   \\
   H_{2}(V(n),M(n)) &\simeq H_{2}(V(n),W(n))[2] \simeq (\Lambda^{4} V(n))[2] \oplus k[-2],
   \\
   H_{p}(V(n),M(n)) &\simeq H_{p}(V(n),W(n))[2] \simeq (\Lambda^{p+2} V(n))[2],
\end{align*}
for $p \geq 3$.
\end{proof}

As explained in \cite{Ber2}, Prop. 2.3, the minimal projective resolution $P(M(n))_{\bullet}$ of the graded $S(V(n))$-module $M(n)$ 
for $n \geq 3$ has the form:
\begin{multline}
\label{eq:reskoszulm}
  0 \rightarrow S(V(n)) \otimes \Tor^{\text{\begin{tiny}$S(V(n))$\end{tiny}}}_{n}(M(n),k) \rightarrow S(V(n)) \otimes \Tor^{\text{\begin{tiny}$S(V(n))$\end{tiny}}}_{n-1}(M(n),k) \rightarrow \dots \rightarrow
  \\
  \rightarrow S(V(n)) \otimes \Tor^{\text{\begin{tiny}$S(V(n))$\end{tiny}}}_{1}(M(n),k)
  \rightarrow S(V(n)) \otimes \Tor^{\text{\begin{tiny}$S(V(n))$\end{tiny}}}_{0}(M(n),k) \rightarrow M(n) \rightarrow 0.
\end{multline}
This in turn implies that, if $N$ is another $\NN_{0}$-graded $S(V(n))$-module, then $\Tor_{i}^{\text{\begin{tiny}$S(V(n))$\end{tiny}}}(M(n),N)$ 
has homogeneous components of internal degree greater than or equal to $i$.

We shall find a differential $d_{\bullet}^{P}$ for the previous resolution \eqref{eq:reskoszulm}.
In order to do so, let us consider the Chevalley-Eilenberg complex $R(M(n))_{\bullet} = (C_{\bullet}(V(n),S(V(n))),d_{\bullet}^{\mathrm{CE}})$ for the regular module $S(V(n))$.
It is acyclic in positive degrees since its homology is $H_{\bullet}(V(n),S(V(n))) = 0$ for $\bullet \geq 1$.
Notice that $R(M(n))_{\bullet+2}$ is a graded vector subspace of 
$S(V(n)) \otimes \Tor^{\text{\begin{tiny}$S(V(n))$\end{tiny}}}_{\bullet}(M(n),k)$ for $\bullet \in \NN_{0}$.

We define $d_{\bullet}^{P}$ of the minimal projective resolution of $M(n)$ as follows.
If $v \in R(M(n))_{p+2}$, for $p \geq 1$, we take $d_{p}^{P}(v) = d_{p+2}^{\mathrm{CE}}(v)$. 
Let us denote $\{ e_{1}, \dots, e_{n} \}$ a basis of $V(n) \subseteq H_{1}(V(n),M(n))$ and $\{ c \}$ a basis of $k[-2] \subseteq H_{2}(V(n),M(n))$; given $z\in S(V(n))$. 
We set
\begin{align}
   \label{eq:difkoszul2}
   d_{2}^{P}(z \otimes c) &= \sum_{j=1}^{n} z x_{j} \otimes e_{j},
   \\
   \label{eq:difkoszul3}
   d_{1}^{P}(z \otimes e_{i}) &= \sum_{j=1}^{n} z x_{j} \otimes x_{j} \wedge x_{i}.
\end{align}
The differential $d^{P}_{\bullet}$, for $\bullet \in \NN$, is given extending $k$-linearly. 

Finally, the augmentation morphism $d_{0}^{P} : P(M(n))_{0} \rightarrow M(n)$ is given by
\begin{equation}
\label{eq:aumcomkoszulm}
     d_{0}^{P} (z \otimes x_{i} \wedge x_{j}) = z.[x_{i},x_{j}].
\end{equation}
By Corollary \ref{coro:W(n)}, $M(n)$ is a finitely generated $S(V(n))$-module with set of generators $\{ [x_{i},x_{j}] : 1 \leq i < j \leq n\}$,
so $d_{0}^{P}$ is surjective.

It is readily verified that $d^{P}_{\bullet}$ is a homogeneous $S(V(n))$-linear and $\so(n)$-equivariant map of degree $0$
and $d_{p}^{P} \circ d_{p+1}^{P} = 0$, for all $p \in \NN_{0}$.

Furthermore, we will now prove  that the complex $P(M(n))_{\bullet}$ is acyclic in positive degrees and hence a resolution of $M(n)$.
On the one hand, since $H_{\bullet}(P(M(n))) = H_{\bullet+2}(R(M(n)))$ for $\bullet \geq 3$, the exactness of
$P(M(n)))_{\bullet}$, for $\bullet \geq 3$ is direct.
The case $\bullet = 2$ is also direct, for the differential given in \eqref{eq:difkoszul2} is injective.
The other cases, \textit{i.e.} $\bullet = 0, 1$, can be checked as follows.

For $\bullet = 1$, let us consider $z = z'+ z'' \in \Ker(d_{1}^{P})$ with
\begin{align*}
  z' &= \sum_{1 \leq i_{1} < i_{2} < i_{3} \leq n} z_{(i_{1},i_{2},i_{3})} \otimes x_{i_{1}} \wedge x_{i_{2}}
       \wedge x_{i_{3}} \in S(V(n)) \otimes (\Lambda^{3} V(n))[2],
  \\
  z''  &= \sum_{i=1}^{n} z_{i} \otimes e_{i} \in S(V(n)) \otimes V(n),
\end{align*}
and thus $d_{1}^{P}(z'') \in \mathrm{Im}(d_{3}^{\mathrm{CE}})$.
Since $\mathrm{Im}(d_{3}^{\mathrm{CE}}) = \Ker(d_{2}^{\mathrm{CE}})$, we see that 
\[     0 = d_{2}^{\mathrm{CE}}(d_{1}^{P}(z''))
      = \sum_{1 \leq i,j \leq n} (z_{i} x_{j}^{2} \otimes x_{i} - z_{i} x_{i} x_{j} \otimes x_{j}).     \]
Therefore, it turns out that $d_{2}^{\mathrm{CE}} \circ d_{1}^{P}|_{S(V(n)) \otimes V(n)} = d_{2}$, where $d_{2}$ is 
the differential of the complex $C_{\bullet}(\YM(n),S(V(n)))$ given by \eqref{eq:complejohomologiayangmills} 
and $z'' \in S(V(n)) \otimes V(n)$ belongs to the kernel of $d_{2}$.
By Proposition 3.5 of \cite{HS1}, $H_{2}(\ym(n), S(V(n))) = 0$ and, as a consequence,
there is a $w \in S(V(n))$ such that $z'' = d_{1} (w) = d_{2}^{P} (w \otimes c)$,
so the complex $P(M(n))_{\bullet}$ is exact in degree $1$.

The only condition left to prove for $P(M(n))_{\bullet}$ to be a resolution of $M(n)$ is that $\Ker(d_{0}^{P}) = \mathrm{Im}(d_{1}^{P})$.
We start observing that the following diagram 
\[
\xymatrix
{
S(V(n)) \otimes \Lambda^{2} V(n)
\ar[r]^{d_{0}^{P}}
\ar[d]^{d_{2}^{\mathrm{CE}}}
&
M(n)
\ar[d]^{\Phi}
\\
\Ker(d_{1})
\ar[r]^{\pi}
&
\Ker(d_{1})/\mathrm{Im}(d_{2})
}
\]
is commutative, where $\pi$ is the canonical projection and $\Phi$ is the isomorphism of Proposition \ref{prop:isow(n)}.
Then, given $w \in S(V(n)) \otimes \Lambda^{2} V(n)$, $w \in \Ker(d_{0}^{P})$ if and only if there exists 
$w' \in S(V(n)) \otimes V(n)$ such that $d_{2}^{\mathrm{CE}}(w) = d_{2}(w')$.
Taking into account that $d_{2} = d_{2}^{\mathrm{CE}} \circ d_{1}^{P}|_{S(V(n)) \otimes V(n)}$, it turns out that
$d_{2}^{\mathrm{CE}}(w) = d_{2}^{\mathrm{CE}}(d_{1}^{P}(w'))$, or, equivalently, $w - d_{1}^{P}(w') \in \Ker(d_{2}^{\mathrm{CE}}) = \mathrm{Im}(d_{3}^{\mathrm{CE}})$.
The fact that $d_{3}^{\mathrm{CE}} = d_{1}^{P}|_{R(M(n))_{3}}$ yields that $\Ker(d_{0}^{P}) = \mathrm{Im}(d_{1}^{P})$.

We have thus proved the following result.
\begin{proposition}
Let $n \geq 3$. 
The complex \eqref{eq:reskoszulm} provided with the $\so(n)$-equivariant differential $d_{\bullet}^{P}$ satisfying $d_{\bullet}^{P}|_{R(M(n))} = d_{\bullet+2}^{\mathrm{CE}}$, \eqref{eq:difkoszul2} and \eqref{eq:difkoszul3} 
is a minimal projective resolution of the graded $S(V(n))$-module $M(n)$.
\end{proposition}

Let $R$ be a graded commutative $k$-algebra, and let $N$, $N'$, $M$ and $M'$ be graded $R$-modules.
By \cite{CE1}, p. 204, 
the external product
\[     \Tor_{i}^{\text{\begin{tiny}$R$\end{tiny}}} (M,M') \otimes_{R} \Tor_{j}^{\text{\begin{tiny}$R$\end{tiny}}} (N,N') 
       \rightarrow \Tor_{i+j}^{\text{\begin{tiny}$R$\end{tiny}}} (M \otimes_{R} N , M' \otimes_{R} N')     \]
is an $R$-linear map homogeneous of degree $0$.
It can be constructed as follows.
If $P_{\bullet} \twoheadrightarrow M$, $P'_{\bullet} \twoheadrightarrow N$ and $P''_{\bullet} \twoheadrightarrow M \otimes_{R} N$
are respectively graded projective resolutions of the graded $R$-modules $M$, $N$ and $M \otimes_{R} N$, 
there exists a morphism $\mathrm{Tot}(P_{\bullet} \otimes_{R} P'_{\bullet}) \rightarrow P''_{\bullet}$,
unique up to chain homotopy equivalence.
The external product is then induced by the morphism
$\mathrm{Tot}(P_{\bullet} \otimes_{R} M' \otimes_{R} P'_{\bullet} \otimes_{R} N') \rightarrow P''_{\bullet} \otimes_{R} (M' \otimes_{R} N')$.

We are interested in the case $R = S(V(n))$, $N = N' = S(V(n))/\cl{q}$ (which will still be denoted by $A$, as in Proposition \ref{prop:succorta}), $M = M(n)$, $i = 0$ and $j = 1$.
We shall also assume that $M'$ is a graded $A$-module.
By taking into account the graded $S(V(n))$-linear isomorphisms $M(n) \otimes_{S(V(n))} A \simeq M(n)$
and $M' \otimes_{S(V(n))} A \simeq M'$, we shall consider the $S(V(n))$-linear map homogeneous of degree $0$
\[     \times_{p,1}^{\text{\begin{tiny}$M'$\end{tiny}}} : \Tor_{p}^{\text{\begin{tiny}$S(V(n))$\end{tiny}}} (M(n),M') 
       \otimes_{S(V(n))} \Tor_{1}^{\text{\begin{tiny}$S(V(n))$\end{tiny}}} (A,A) 
       \rightarrow \Tor_{p+1}^{\text{\begin{tiny}$S(V(n))$\end{tiny}}} (M(n) , M').    \]
But $\Tor_{1}^{\text{\begin{tiny}$S(V(n))$\end{tiny}}} (A,A) \simeq A$, so the previous map is in fact
\[     \times_{p,1}^{\text{\begin{tiny}$M'$\end{tiny}}} : \Tor_{p}^{\text{\begin{tiny}$S(V(n))$\end{tiny}}} (M(n),M') 
       \rightarrow \Tor_{p+1}^{\text{\begin{tiny}$S(V(n))$\end{tiny}}} (M(n) , M').     \]
We take $P_{\bullet} = P''_{\bullet} = P(M(n))_{\bullet}$ and $P'_{\bullet}$ given by \eqref{eq:resA} with $z = q$ of degree $d = 2$.
In this case we must find a chain map $t_{\bullet} : \mathrm{Tot}(P(M(n))_{\bullet} \otimes_{S(V(n))} P'_{\bullet}) \rightarrow P(M(n))_{\bullet}$.

Since $P'_{\bullet}$ consists of only two non-zero terms, the double complex 
$C_{p,q} = P(M(n))_{p} \otimes_{S(V(n))} P'_{q}$ has only two non trivial rows $q = 0, 1$.
We define the morphism $t_{\bullet}$ as follows.
In the first place, if $v \in C_{p,0} = P(M(n))_{p} \otimes_{S(V(n))} S(V(n)) \simeq P(M(n))_{p}$, let $t_{p}(v \otimes 1_{A}) = v$.
Suppose now that $p \geq 1$ and $v' \in C_{p-1,1} = P(M(n))_{p-1} \otimes_{S(V(n))} S(V(n))[-2] \simeq P(M(n))_{p-1}[-2]$.

There is an isomorphism $P(M(n))_{p} \simeq R(M(n))_{p+2} \oplus C_{p}$ as graded $S(V(n))$-modules,
where $C_{1} = S(V(n)) \otimes V(n)$, $C_{2} = S(V(n)) \otimes k[-2]$ and the other ones are trivial.
Consider an $S(V(n))$-linear map $s_{\bullet} : P(M(n))_{\bullet-1}[-2] \rightarrow P(M(n))_{\bullet}$ 
homogeneous of degree $0$ for $\bullet \in \NN$
and represented in matrix form as follows 
\[     s_{\bullet} = \left(\begin{matrix}
                     s^{1,1}_{\bullet} & s^{1,2}_{\bullet}
                     \\
                     s^{2,1}_{\bullet} & s^{2,2}_{\bullet}
                     \end{matrix}\right).     \] 
where $s_{\bullet}^{i,j}$ are maps of graded $S(V(n))$-modules defined as follows:
\begin{align*}
  s_{p}^{1,1} : R(M(n))_{p+1}[-2] &\rightarrow R(M(n))_{p+2}
  \\
  z \otimes x_{i_{1}} \wedge \dots \wedge x_{i_{p+1}} &\mapsto
  \sum_{l=1}^{n} z x_{l} \otimes x_{l} \wedge x_{i_{1}} \wedge \dots \wedge x_{i_{p+1}},\\
  s^{2,1}_{1} : S(V(n)) \otimes (\Lambda^{2} V(n))
     &\rightarrow
  S(V(n)) \otimes V(n)
  \\
  z \otimes x_{g} \wedge x_{h} &\mapsto
  (z x_{g} \otimes x_{h} - z x_{h} \otimes x_{g}),
\end{align*}
and $s^{2,1}_{p} = 0$, for $p \geq 1$. 
Also, 
\begin{align*}
  s^{2,2}_{2} : S(V(n)) \otimes V(n)[-2]
     &\rightarrow
  S(V(n)) \otimes (k[-2])
  \\
  z \otimes e_{i} &\mapsto z x_{i} \otimes c,
\end{align*}
and $s^{2,2}_{p} = 0$, if $p \neq 2$. 
Finally, $s^{1,2}_{p} = 0$, for all $p \in \NN$. 

It is trivially verified that $s_{\bullet}$ satisfies 
$s_{\bullet + 1} \circ s_{\bullet} = 0$, $d \circ s + s \circ d = q.(\place)$
and the external product $\times_{p,1}^{\text{\begin{tiny}$M'$\end{tiny}}}$ coincides with the morphism induced in homology
by $\id_{M'} \otimes_{S(V(n))} s_{p+1}$, for $p \in \NN_{0}$.

We have the following simple result. 
\begin{proposition}
\label{prop:gentor1}
Let 
\[     x = \sum_{i=1}^{n} (x_{i} \otimes 1 - 1 \otimes x_{i}) \otimes x_{i} \in A \otimes A \otimes V(n)     \]
be a cycle of the Chevalley-Eilenberg complex for the homology of $V(n)$ with coefficients in $A \otimes A$. 
It is a generator of $\Tor^{\text{\begin{tiny}$S(V(n))$\end{tiny}}}_{1}(A,A)$ as an $S(V(n))$-module.
\end{proposition}
\begin{proof}
It is clear that $x$ is a cycle.
Also, since there are no boundaries for $\Tor^{\text{\begin{tiny}$S(V(n))$\end{tiny}}}_{1,2}(A,A)$, by degree reasons, it cannot be a boundary. 
Since $x$ has internal degree $2$ and $\Tor^{\text{\begin{tiny}$S(V(n))$\end{tiny}}}_{1} (A,A) \simeq A[-2]$, 
$x$ must be a generator of the $S(V(n))$-module $\Tor^{\text{\begin{tiny}$S(V(n))$\end{tiny}}}_{1}(A,A)$.
\end{proof}

The previous proposition yields a useful description of the image of $\times_{0,1}^{\text{\begin{tiny}$M(n)$\end{tiny}}}$.
\begin{corollary}
\label{coro:ciclostor2}
The image of $\times_{0,1}^{\text{\begin{tiny}$M(n)$\end{tiny}}}$ is the set of elements of degree $j \geq 2$ in $H_{1}(V(n),M(n) \otimes M(n))$ of the form
\[     \sum_{i=1}^{n} \sum_{l \in L} (m_{l} x_{i} \otimes m'_{l} - m_{l} \otimes m'_{l} x_{i}) \otimes x_{i},     \]
for $L$ a finite set of indices and $m_{l}, m'_{l} \in M(n)$ homogeneous of degrees $d_{l}$ and $d'_{l}$
such that $d_{l} + d'_{l} = j-2$, for all $l \in L$.
These elements will be called \emph{generic}.
\end{corollary}
\begin{proof}
By the previous proposition,
\[     \sum_{i=1}^{n} (x_{i} \otimes 1 - 1 \otimes x_{i}) \otimes x_{i}     \]
is a homogeneous generator of the graded $S(V(n))$-module $\Tor^{\text{\begin{tiny}$S(V(n))$\end{tiny}}}_{1}(A,A)$.
Since the morphism $\times_{0,1}^{\text{\begin{tiny}$M(n)$\end{tiny}}}$ is $S(V(n))$-linear homogeneous of degree $2$, the elements of its image are of the prescribed form.
\end{proof}

We want to describe the kernel and cokernel of the map $\times_{0,1}^{\text{\begin{tiny}$M(n)$\end{tiny}}}$.
As $\Tor_{1}^{\text{\begin{tiny}$S(V(n))$\end{tiny}}}(A,A) \simeq A[-2]$, 
it has homogeneous components of degree greater than or equal to $2$.
By degree reasons, the cokernel of the external product for $M' = M(n)$ contains the homogeneous component of internal degree $1$ of
$\Tor_{1}^{\text{\begin{tiny}$S(V(n))$\end{tiny}}}(M(n),M(n))$, which is an $\NN$-graded $S(V(n))$-module by the Koszul property of $M(n)$.

In this case, we can consider the following mixed complex of (non graded) $A$-modules $B(M')_{\bullet} = M' \otimes_{S(V(n))} P(M(n))_{\bullet}$
with vertical differential $\id_{M'} \otimes d_{\bullet}^{P}$ and horizontal differential $\id_{M'} \otimes_{S(V(n))} s_{\bullet+1}$.
Since $q$ acts by zero on any $A$-module, the previous conditions imply that  
$(B(M')_{\bullet},d_{\bullet},s_{\bullet-1})$ satisfies the definition of a mixed complex (see \cite{Ka1}).
We recall that the first quadrant Connes' double complex associated to the previous mixed complex is defined as $B(M')_{p,q}=B(M')_{q-p}$
if $0 \leq p \leq q$ and zero otherwise, $H_{\bullet}(B(M'))$ denotes the usual vertical homology of $B(M')_{\bullet}$, 
and the homology of its total complex is called the \emph{cyclic homology} $HC_{\bullet}(B(M'))$ of the mixed complex $B(M')_{\bullet}$.
We note that $HC_{0}(B(M')) \simeq H_{0}(B(M'))$. 

The filtration by columns of Connes' double complex yields the so-called Connes' spectral sequence,
which also gives that the first cyclic homology group $HC_{1} (B(M'))$ fits in the low degree exact sequence of $A$-modules (see \cite{Wei1}, 9.8.6)
\begin{equation}
\label{eq:connes}
     H_{1}(B(M')) \shortrightarrow H_{2}(B(M')) \shortrightarrow HC_{2}(B(M'))    
     \shortrightarrow H_{0}(B(M')) \shortrightarrow H_{1}(B(M')) \shortrightarrow HC_{1}(B(M')) \shortrightarrow 0,
\end{equation}
where the first and the fourth map are induced by $\id_{M'} \otimes s_{2}$ and $\id_{M'} \otimes s_{1}$, respectively. 
So, $HC_{1} (B(M'))$ is isomorphic to the cokernel of the map $\times_{0,1}^{\text{\begin{tiny}$M'$\end{tiny}}}$ (without considering the grading).
When dealing with these mixed complexes we will not take into account any grading.

If $0 \rightarrow M'_{1} \rightarrow M'_{2} \rightarrow M'_{3} \rightarrow 0$ is a short exact sequence of $A$-modules,
it induces in turn a short exact sequence of mixed complexes $0 \rightarrow B(M'_{1}) \rightarrow B(M'_{2}) \rightarrow B(M'_{3}) \rightarrow 0$,
so a short exact sequence of the corresponding total complexes 
and hence a long exact sequence
of cyclic homologies
\[     \dots \rightarrow HC_{p}(B(M'_{1})) \rightarrow HC_{p}(B(M'_{2})) \rightarrow HC_{p}(B(M'_{3})) \rightarrow HC_{p-1}(B(M'_{1})) \rightarrow \dots     \]

If $M' = A$, the total complex $\mathrm{Tot}(P(M(n))_{\bullet} \otimes_{S(V(n))} P'_{\bullet})$ computes 
$\Tor_{\bullet}^{\text{\begin{tiny}$S(V(n))$\end{tiny}}}(M(n),A)$,
which can also be computed from any of the complexes $M(n) \otimes_{S(V(n))} P'_{\bullet}$ or $P(M(n))_{\bullet} \otimes_{S(V(n))} A$.
Moreover, the canonical projections from the total complex to either $M(n) \otimes_{S(V(n))} P'_{\bullet}$ or $P(M(n))_{\bullet} \otimes_{S(V(n))} A$
are quasi-isomorphisms.
By diagram chasing arguments, the component of the quasi-isomorphism from $M(n) \otimes_{S(V(n))} P'_{\bullet}$
to $P(M(n))_{\bullet} \otimes_{S(V(n))} A$ passing through $\mathrm{Tot}(P(M(n))_{\bullet} \otimes_{S(V(n))} P'_{\bullet})$
in degree $\bullet = 1$ is induced by $(\pi_{A} \otimes \id_{P(M(n))}) \circ s_{1}$
, for $\pi_{A} : S(V(n)) \rightarrow A$ the canonical projection.
This implies that $\times_{0,1}^{\text{\begin{tiny}$A$\end{tiny}}}$ is an isomorphism.

On the other hand, the mixed complex $B(A)_{\bullet}$ has only vertical homology $H_{0}(B(A))$ and $H_{1}(B(A))$ and
the first term of the Connes' spectral sequence $E^{1}_{p,q} = H_{q-p}(B(A))$ for $B(A)_{\bullet}$ satisfies that
$d^{1}_{p,p} : E^{1}_{p,p} \rightarrow E^{1}_{p,p+1}$ coincides with $\times_{0,1}^{\text{\begin{tiny}$A$\end{tiny}}}$, hence it is an isomorphism.
As a consequence, $HC_{\bullet}(B(A))=0$, for $\bullet \geq 1$.

The short exact sequence of graded $A$-modules $0 \rightarrow A_{+} \rightarrow A \rightarrow k \rightarrow 0$ (where $A_{+}$ is the irrelevant ideal
of $A$) implies that $HC_{\bullet}(B(A_{+})) \simeq HC_{\bullet+1}(B(k))$, for $\bullet \geq 1$.
It is direct to check that $HC_{3}(B(k)) \simeq \Lambda^{5} V(n) \oplus \Lambda^{3} V(n) \oplus V(n)$ and 
$HC_{4}(B(k)) \simeq \Lambda^{6} V(n) \oplus \Lambda^{4} V(n) \oplus \Lambda^{2} V(n) \oplus k$, 
since all morphisms of $B(k)$ are zero.
It follows that $HC_{2}(B(A_{+})) \simeq \Lambda^{5} V(n) \oplus \Lambda^{3} V(n) \oplus V(n)$ 
and $HC_{3}(B(A_{+})) \simeq \Lambda^{6} V(n) \oplus \Lambda^{4} V(n) \oplus \Lambda^{2} V(n) \oplus k$.

The long exact sequence of cyclic homology obtained from the short exact sequence of graded $A$-modules given by \eqref{eq:euler},
is given by 
\[     0 \rightarrow HC_{1}(B(\Omega_{A/k})) \rightarrow HC_{0}(B(A[-2])) \rightarrow HC_{0}(B(A[-1]^{n}))
         \rightarrow HC_{0}(B(\Omega_{A/k})) \rightarrow 0.     \]
It is directly verified that $HC_{0}(B(A[-2])) \simeq M(n)[-2]$, $HC_{0}(B(A[-1]^{n})) \simeq M(n)[-1]^{n}$ and, 
under these identifications, the map $HC_{0}(B(A[-2])) \rightarrow HC_{0}(B(A[-1]^{n}))$ coincides with $w \mapsto \sum_{i=1}^{n} w x_{i} \otimes x_{i}$ and it  is injective, since $\mathrm{ann}_{W(n)}(x_{i}) = 0$ for all $i = 1, \dots, n$ (see Prop. \ref{prop:Minyectivo}),
implying that $HC_{1}(B(\Omega_{A/k}))=0$.
Since $HC_{\bullet}(B(A))=0$, for $\bullet \geq 1$, we find that $HC_{\bullet}(B(\Omega_{A/k}))=0$, for $\bullet \geq 2$.
Hence, the long exact sequence of cyclic homology obtained from the short exact sequence of graded $A$-modules given by \eqref{eq:sucW(n)}
assures that $HC_{2}(B(A_{+})) \simeq HC_{1}(B(M(n)))$, so $HC_{1}(B(M(n))) \simeq \Lambda^{5} V(n) \oplus \Lambda^{3} V(n) \oplus V(n)$. 
Hence, the cokernel of the map $\times_{0,1}^{\text{\begin{tiny}$M(n)$\end{tiny}}}$ is isomorphic to $\Lambda^{5} V(n) \oplus \Lambda^{3} V(n) \oplus V(n)$ as $k$-vector spaces.

The previous considerations also imply that $HC_{3}(B(A_{+})) \simeq HC_{2}(B(M(n)))$, and, in consequence, 
$HC_{2}(B(M(n))) \simeq \Lambda^{6} V(n) \oplus \Lambda^{4} V(n) \oplus \Lambda^{2} V(n) \oplus k$. 
Moreover, Theorem \ref{teo:exacta} tells us that $H_{2}(B(M(n))) \simeq H_{2}(V(n),M(n)^{\otimes 2}) \simeq \Lambda^{2} V(n)$, 
so, from the proof of Proposition \ref{prop:b1b2} in the next section, we obtain that the homology classes of the cycles 
\[     \Big\{ \sum_{\sigma \in \SSS_{6}} \epsilon(\sigma) [x_{i_{\sigma(1)}}, x_{i_{\sigma(2)}}] \otimes [x_{i_{\sigma(3)}}, x_{i_{\sigma(4)}}]
       \otimes x_{i_{\sigma(5)}} \wedge x_{i_{\sigma(6)}}
       : 1 \leq i_{1} < i_{2} < i_{3} < i_{4} < i_{5} < i_{6} \leq n \Big\}     \]
provide a basis composed of elements of degree $2$ of $H_{2}(V(n),M(n)^{\otimes 2})$. 
Since the second map of the low degree exact sequence \eqref{eq:connes} is induced by the inclusion of the complex 
$(B(M(n))_{\bullet},d_{\bullet})$ in $\mathrm{Tot}(B(M(n))_{\bullet,\bullet})$ and the components of the boundaries of  $\mathrm{Tot}(B(M(n))_{\bullet,\bullet})$ in $B(M(n))_{0,2}$ have degree strictly greater than $2$, we deduce that the previous list of 
homology classes is also linearly independent when considered in $\mathrm{Tot}(B(M(n))_{\bullet,\bullet})$, and therefore, 
the second map of long exact sequence \eqref{eq:connes} is injective and the kernel of the map $\times_{0,1}^{\text{\begin{tiny}$M(n)$\end{tiny}}}$ is isomorphic to $\Lambda^{4} V(n) \oplus \Lambda^{2} V(n) \oplus k$ 
as $k$-vector spaces.

We shall first give an explicit description of the cokernel of the external product.
\begin{proposition}
\label{prop:ciclostor1}
Let $n \geq 3$. 
There is a sequence of $k$-vector space isomorphisms
\[     \Coker(\times_{0,1}^{\text{\begin{tiny}$M(n)$\end{tiny}}}) \simeq \Tor^{\text{\begin{tiny}$S(V(n))$\end{tiny}}}_{1,1}(M(n),M(n))
       \simeq \Lambda^{5} V(n) \oplus \Lambda^{3} V(n) \oplus V(n),     \]
where we consider for each direct summand of $\Tor^{\text{\begin{tiny}$S(V(n))$\end{tiny}}}_{1,1}(M(n),M(n))$
the following bases given by the homology classes of the cycles of the Chevalley-Eilenberg
complex $C_{1}(V(n),M(n) \otimes_{S(V(n))} M(n))$:
\begin{itemize}
\item[(i)] for $V(n)$:
\[     \Big\{ \sum_{1 \leq i < j \leq n} [x_{i},x_{j}] \otimes [x_{i},x_{j}] \otimes x_{l}
           + \sum_{1 \leq i, j \leq n} \big([x_{l},x_{i}] \otimes [x_{i},x_{j}] \otimes x_{j} + [x_{i},x_{j}] \otimes [x_{l},x_{i}] \otimes x_{j}\big)  
           \Big\}_{1 \leq l \leq n},
\]
\item[(ii)] for $\Lambda^{3} V(n)$:
\[     \Big\{ \sum_{l=1}^{n} \sum_{\sigma \in \SSS_{3}} \epsilon(\sigma)
           \big([x_{i_{\sigma(1)}},x_{i_{\sigma(2)}}] \wedge [x_{i_{\sigma(3)}},x_{l}] \otimes x_{l}
          + [x_{i_{\sigma(1)}},x_{l}] \wedge [x_{i_{\sigma(2)}},x_{l}] \otimes x_{i_{\sigma(3)}}\big)
           \Big\}_{1 \leq i_{1} < i_{2} < i_{3} \leq n},
\]
\item[(iii)] for $\Lambda^{5} V(n)$:
\[     \Big\{ \sum_{\sigma \in \SSS_{5}} \epsilon(\sigma) [x_{i_{\sigma(1)}},x_{i_{\sigma(2)}}] \otimes
        [x_{i_{\sigma(3)}},x_{i_{\sigma(4)}}] \otimes x_{i_{\sigma(5)}}
        \Big\}_{1 \leq i_{1} < i_{2} < i_{3} < i_{4} < i_{5} \leq n}.
\]
\end{itemize}
\end{proposition}
\begin{proof}
The fact that these classes are cycles of the Chevalley-Eilenberg complex is direct but rather tedious to check.
Also, it is easy to see that they all belong to
\[     (C_{1}(V(n),M(n) \otimes_{S(V(n))} M(n)))_{1} = M(n)_{0} \otimes M(n)_{0} \otimes V(n)
= \Lambda^{2} V(n) \otimes \Lambda^{2} V(n) \otimes V(n),     \]
where there are no boundaries by degree reasons.
We leave the simple but rather long proof that they all form a linearly independent set, which mostly depends on degree reasons. 

On the other hand, since
\[     \Lambda^{5} V(n) \oplus \Lambda^{3} V(n) \oplus V(n)
       = \Coker(\times_{0,1}^{\text{\begin{tiny}$M(n)$\end{tiny}}}) \supseteq \Tor^{\text{\begin{tiny}$S(V(n))$\end{tiny}}}_{1,1}(M(n),M(n))
       \supseteq \Lambda^{5} V(n) \oplus \Lambda^{3} V(n) \oplus V(n),     \]
all previous inclusions must be equalities and the proposition follows.
\end{proof}

Finally, we will provide an explicit description of the kernel of the external product.
\begin{proposition}
\label{prop:ciclosnucleo}
Let $n \geq 3$. 
There is an isomorphism of graded $k$-vector spaces, homogeneous of degree $0$,
\[     \Ker(\times_{0,1}^{\text{\begin{tiny}$M(n)$\end{tiny}}}) \simeq \Lambda^{4} (V(n)[1]) \oplus \Lambda^{2} (V(n)[1]) \oplus k,     \]
where we consider for each direct summand the bases given by the homology classes of the following cycles of the 
Chevalley-Eilenberg complex $C_{0}(V(n),M(n) \otimes_{S(V(n))} M(n))$:
\begin{itemize}
\item[(i)] for $k$:
\[     \Big\{ \sum_{1 \leq i < j \leq n} [x_{i},x_{j}] \otimes [x_{i},x_{j}] \Big\},
\]
\item[(ii)] for $\Lambda^{2} V(n)$:
\[     \Big\{ \sum_{l=1}^{n} [x_{i},x_{l}] \wedge [x_{j},x_{l}] \Big\}_{1 \leq i < j \leq n},
\]
\item[(iii)] for $\Lambda^{4} V(n)$:
\[     \Big\{ \sum_{\sigma \in \SSS_{4}} \epsilon(\sigma) [x_{i_{\sigma(1)}},x_{i_{\sigma(2)}}] \otimes
        [x_{i_{\sigma(3)}},x_{i_{\sigma(4)}}] \Big\}_{1 \leq i_{1} < i_{2} < i_{3} < i_{4} \leq n}.
\]
\end{itemize}
\end{proposition}
\begin{proof}
As in the proof of the previous proposition, it is direct but tedious to verify the fact that these classes are cycles of the Chevalley-Eilenberg complex.
Also, it is easy to see that they all belong to
\[     (C_{0}(V(n),M(n) \otimes_{S(V(n))} M(n)))_{0} = M(n)_{0} \otimes M(n)_{0}
       = \Lambda^{2} V(n) \otimes \Lambda^{2} V(n),     \]
that they belong to $\Ker(\times_{0,1}^{\text{\begin{tiny}$M(n)$\end{tiny}}})$ and that
they all form a linearly independent set, always using degree considerations. 
Finally, since
\[     \Lambda^{4} (V(n)[1]) \oplus \Lambda^{2} (V(n)[1]) \oplus k
       = \Ker(\times_{0,1}^{\text{\begin{tiny}$M(n)$\end{tiny}}}) \supseteq \Tor^{\text{\begin{tiny}$S(V(n))$\end{tiny}}}_{0,0}(M(n),M(n)),     \]
       and
\[     \Tor^{\text{\begin{tiny}$S(V(n))$\end{tiny}}}_{0,0}(M(n),M(n)) \supseteq \Lambda^{4} (V(n)[1]) \oplus \Lambda^{2} (V(n)[1]) \oplus k,     \]
all the inclusions must be equalities and the proposition follows.
\end{proof}

\begin{remark}
\label{rem:s2wwedge2w}
It is clear that the following inclusions hold: 
$k \oplus \Lambda^{4} (V(n)[1]) \subseteq H_{0}(V(n),S^{2}M(n))$ and $\Lambda^{2} (V(n)[1]) \subseteq H_{0}(V(n),\Lambda^{2}M(n))$. 
It is not difficult to find exactly which irreducible representations of $\so(n)$ among the ones given for the decomposition of $H_{0,0}(V(n),M(n)^{\otimes 2})$ in Corollary \ref{coro:compisownwn} belong to the component $H_{0,0}(V(n),S^{2}M(n))$ or to $H_{0,0}(V(n),\Lambda^{2}M(n))$. 
On the one hand, we recall that if $\lambda$ is a dominant weight and $\alpha$ a positive root of $\so(n)$ such that $\lambda(H_{\alpha}) \neq 0$, 
then $\Gamma_{2 \lambda} \subseteq S^{2}\Gamma_{\lambda}$ and $\Gamma_{2 \lambda - \alpha} \subseteq \Lambda^{2}\Gamma_{\lambda}$ 
(see \cite{FH1}, Exercise 25.32). 
First, applying the previous fact to $\lambda = L_{1}$ for $n=3$, we see that 
\[     \Gamma_{2 L_{1}} \subseteq H_{0,0}(V(n),S^{2}M(n)),     \]
and secondly, to $\lambda = L_{1} + L_{2}$ for $n \geq 4$, we find that
\[     \Gamma_{2 L_{1} + 2 L_{2}} \subseteq H_{0,0}(V(n),S^{2}M(n)).     \]
Also, for a suitable $\alpha$ in each case, we conclude that 
\begin{align*}
   \Gamma_{2 L_{1} + L_{2}} &\subseteq H_{0,0}(V(n),\Lambda^{2}M(n)), &\text{if $n = 5$,}
   \\
   \Gamma_{2 L_{1} + L_{2} + L_{3}} \oplus \Gamma_{2 L_{1} + L_{2} - L_{3}} &\subseteq H_{0,0}(V(n),\Lambda^{2}M(n)), &\text{if $n = 6$,}
   \\
   \Gamma_{2 L_{1} + L_{2} + L_{3}} &\subseteq H_{0,0}(V(n),\Lambda^{2}M(n)), &\text{if $n \geq 7$.}
\end{align*}
By a dimension argument, it is not difficult to check that, for $n \geq 5$, 
$H_{0,0}(V(n), \Lambda^{2}M(n)) \simeq \Lambda^{2} (V(n)[1]) \wedge \Lambda^{2}(V(n)[1])$ 
is the direct sum of $\Lambda^{2}(V(n)[1])$ and the modules appearing in the previous list, in each case
(see \cite{FH1}, Eq. (24.29) and (24.41)). 
The same argument implies that $H_{0,0}(V(n), \Lambda^{2}M(n)) \simeq \Lambda^{2} (V(n)[1]) \wedge \Lambda^{2}(V(n)[1])$ 
is the direct sum of $\Lambda^{2}(V(n)[1])$ and $\Gamma_{2 L_{1}}$ for $n = 4$. 
Finally, for $n=3$, $H_{0,0}(V(n), \Lambda^{2}M(n))$ is isomorphic to $\Lambda^{2} (V(n)[1])$. 
\qed
\end{remark}

As a consequence of Corollary \ref{coro:anulacion} and Proposition \ref{prop:ciclostor1}, we may obtain the following corollary, which corrects the one stated in \cite{Mov1},  Corollary 21. 
\begin{corollary}
\label{coro:anulacion2}
Let $n \geq 3$. 
The graded $S(V(n))$-module $W(n)^{\otimes i}$ is free if and only if $i > \mathrm{max}\{2,(n-1)/2\}$.
\end{corollary}
\begin{proof}
Since $W(n)^{\otimes i}$ is a bounded below graded $S(V(n))$-module, it is free  
if and only if the homology groups $H_{\bullet}(V(n),W(n)^{\otimes i})$ vanish for all $\bullet > 0$ (see \cite{Ber2}, Prop. 2.3 and Cor. 2.4). 
By Theorem \ref{teo:exacta}, we see that this is never the case for $i = 1$. 
Moreover, Proposition \ref{prop:ciclostor1} implies that $H_{1}(V(n),W(n)^{\otimes 2})$ does not vanish, so we assume that $i \geq 3$. 

Theorem \ref{teo:exacta} tells us that, if $p \geq 2$, then $H_{p} (V(n),W(n)^{\otimes i}) \simeq \Lambda^{p + 2 i} V(n)$ vanishes 
if and only if $p + 2 i > n$, \textit{i.e.} $i > (n - p)/2$. 
Hence, the collection of homology groups $\{ H_{p} (V(n),W(n)^{\otimes i}) \}_{p \geq 2}$ vanish if and only if $i > (n-2)/2$.

Finally, by Corollary \ref{coro:anulacion}, $H_{2}(\ym(n),W(n)^{\otimes i}) = 0$ for $i \geq 2$, which implies that the map  
$S'_{i} : H_{3}(V(n),W(n)^{\otimes i}) \rightarrow H_{1}(V(n),W(n)^{\otimes (i+1)})$ is an isomorphism.
Therefore, there is an isomorphism $H_{1}(V(n),W(n)^{\otimes i}) \simeq \Lambda^{3 + 2 (i-1)} V(n)$, for $i \geq 3$.
Then, $H_{1}(V(n),W(n)^{\otimes i}) = 0$ if and only if $1 + 2i > n$, \textit{i.e.} $i > (n - 1)/2$. 
The corollary is thus proved.
\end{proof}

\begin{remark}
Since $M(2)=W(2)[2]$ is isomorphic to the trivial equivariant $S(V(2))$-module (see Proposition \ref{prop:W(2)}), 
it is not hard to see that Theorem \ref{teo:exacta}, Propositions \ref{prop:annh3yangmills} and \ref{prop:wni}, 
and Corollaries \ref{coro:anulacion} and \ref{coro:anulacion2} are not true if $n = 2$. 
Nevertheless, it may be directly verified that Propositions \ref{prop:s1inyectivo}, \ref{prop:ciclostor1} and \ref{prop:ciclosnucleo} 
are valid even if $n = 2$. 
\qed
\end{remark}

\section
{\texorpdfstring{The cohomology $H^{\bullet}(\ym(n),\Ker(\epsilon_{\tym(n)}))$}
{The homology of the Yang-Mills algebra with coefficients in the augmentation ideal of TYM(n)}}
\label{sec:calculoh1tymn}

In this section we shall compute the intermediate cohomology groups $H^{0}(\ym(n),\Ker(\epsilon_{\tym(n)}))$ and 
$H^{1}(\ym(n),\Ker(\epsilon_{\tym(n)}))$, which 
will be used in the next section in order to determine the first Hochschild cohomology group of the Yang-Mills algebra $\YM(n)$, for $n \geq 3$. 
From now on, we shall assume that $n \geq 3$, unless otherwise stated.
All morphisms will also be $\so(n)$-equivariant, unless we state the contrary.

\subsection{Analysis of the spectral sequence}
\label{subsection:espe}

The aim of this section is to prove the following theorem.
\begin{theorem}
\label{teo:homi}
If $n \geq 3$, $H_{2}(\ym(n),\Ker(\epsilon_{\tym(n)})) = H_{2}(\ym(n),\tym(n))$ is
isomorphic to $V(n)[-4]$ as a graded $\so(n)$-module and $H_{3}(\ym(n),\Ker(\epsilon_{\tym(n)}))$ vanishes.
\end{theorem}
\begin{proof}
To simplify notation, we shall denote the augmentation ideal $\Ker(\epsilon_{\tym(n)})$ simply by $I$.

Let us consider the increasing filtration $F_{\bullet}I$ of equivariant $\YM(n)$-modules of $I$ given by
\[     F_{p}I = \begin{cases}
                   I, &\text{if $p \geq -1$,}
                   \\
                   I^{-p}, &\text{if $p \leq -2$.}
                \end{cases}
                \]
The previous filtration is exhaustive and Hausdorff and it induces an exhaustive and Hausdorff increasing filtration 
$F_{p} C_{\bullet}(\YM(n),I) = C_{\bullet}(\YM(n),F_{p} I)$ on the complex $C_{\bullet}(\YM(n),I)$. 
Therefore, we obtain a spectral sequence whose zeroth term is 
\[     E^{0}_{p,q} = F_{p} C_{p+q}/F_{p-1}C_{p+q} \simeq C_{p+q}(\YM(n),F_{p}I/F_{p-1}I).     \]

Moreover, the isomorphism of graded algebras $\TYM(n) \simeq T W(n)$ yields an isomorphism of equivariant $\ym(n)$-modules
\[     F_{p} I/F_{p-1} I = \begin{cases}
                              I^{-p}/I^{-p+1} \simeq W(n)^{\otimes (-p)}, &\text{if $p \leq -1$,}
                              \\
                              0, &\text{if $p \geq 0$.}
                           \end{cases}
\]
So the initial term of the spectral sequence may be written as 
\[     E^{0}_{p,q} \simeq C_{p+q}(\YM(n),F_{p}I/F_{p-1}I) \simeq \begin{cases}
                                                                    C_{p+q}(\ym(n),W(n)^{\otimes (-p)}),
                                                                    &\text{if $p \leq -1$,}
                                                                    \\
                                                                    0, &\text{if $p \geq 0$,}
                                                                 \end{cases}
\]
As a consequence,
\[     E^{1}_{p,q} \simeq \begin{cases}
                             H_{p+q}(\ym(n),W(n)^{\otimes (-p)}),
                             &\text{if $p \leq -1$,}
                             \\
                             0, &\text{if $p \geq 0$.}
                          \end{cases}
\]

By Corollary \ref{coro:anulacion}, $H_{2}(\ym(n), W(n)^{\otimes i}) = 0$, for $i \geq 2$, so we see that $E^{1}_{p,2-p} = 0$, for all 
$p \in \ZZ \setminus \{ -1 \}$.
Also, Proposition \ref{prop:annh3yangmills} tells us that $H_{3}(\ym(n), W(n)^{\otimes i}) = 0$, for $i \geq 1$,
so we have that $E^{1}_{p,3-p} = 0$, for all $p \in \ZZ$.
Furthermore, $H_{j}(\ym(n), W(n)^{\otimes i}) = 0$ for $j \notin \{ 0, 1, 2, 3 \}$ and $i \in \NN_{0}$, so
by Corollary \ref{coro:anulacion2} the spectral sequence is bounded and therefore convergent. 
Thus, $H_{3}(\ym(n),I) = 0$.

We may present the previous results in a pictorial way. 
\begin{center}
\begin{figure}[H]
\[
\xymatrix@R-10pt
{
&
&
&
&
&
&
\\
\bullet
&
0
&
0
&
0
&
0
&
&
\\
\bullet
&
\bullet
\ar[l]
&
0
&
0
&
0
&
E_{\bullet,\bullet}^{1}
&
\\
0
&
\bullet
&
\bullet
\ar[l]
\ar@{.}[uull]
&
\bullet
\ar[l]^{d^{1}_{-1,3}}
&
0
&
&
\\
0
&
0
&
\bullet
&
\bullet
\ar[l]
&
0
&
&
\\
0
&
0
&
0
&
\bullet
\ar@{.}[uuulll]
\ar@{.}[uu]
&
0
&
&
\\
0
&
0
&
0
&
0
&
0
\ar[uuuuuu]^>{q}
\ar[rr]^>{p}
&
&
}
\]
\caption{First step $E^{1}_{\bullet,\bullet}$ of the spectral sequence.
The dotted lines show the limits wherein the spectral sequence is concentrated}
\end{figure}
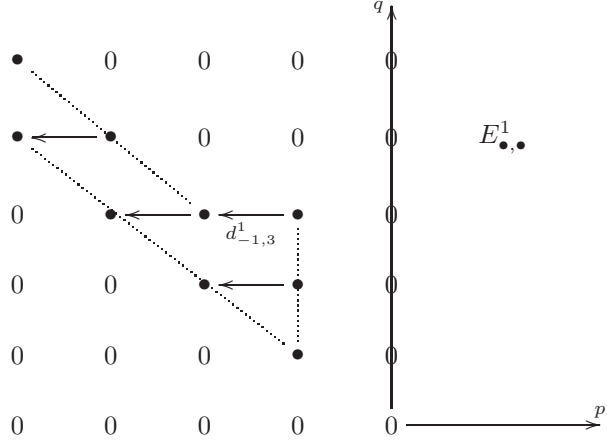
\end{center}

We shall study the differential
\[     d^{1}_{-1,3} : E^{1}_{-1,3} = H_{2}(\ym(n), W(n)) \rightarrow H_{1}(\ym(n), W(n)^{\otimes 2}) = E^{1}_{-2,3}.     \]
We will prove in Proposition \ref{prop:ker} that the kernel of the map $d^{1}_{-1,3}$ is exactly $V(n)[-4]$,
for which a basis was given in Proposition \ref{prop:ciclostor1} and Corollary \ref{coro:ciclostor2}.
This implies that $E^{2}_{-1,3} \simeq V(n)[-4]$. 
Hence, $H_{2}(\ym(n),I)$ is a subquotient of $V(n)[-4]$.
However, since the spectral sequence is defined in the category of $\so(n)$-modules, $H_{2}(\ym(n),I)$ must be a subquotient as an $\so(n)$-module. 
As $V(n)[-4]$ is an irreducible $\so(n)$-module, $H_{2}(\ym(n),I)$ can only be $V(n)[-4]$ or trivial.
This last possibility is impossible, due to the fact that, given $i = 1, \dots, n$, the homology classes of the cycles 
\[     \sum_{l = 1} [x_{i},x_{l}] \otimes x_{l} \in C_{2}(\ym(n),\tym(n)) \subseteq C_{2}(\ym(n),I)     \]
considered in Lemma \ref{lema:morfb1} form a linearly independent set, by internal weight reasons. 
\end{proof}

\subsection{\texorpdfstring
{Computation of the kernel of $d_{-1,3}^{1}$}
{Computation of the kernel of d}
}
\label{subsec:kerneld}

In this subsection we shall prove that $\Ker(d^{1}_{-1,3}) \simeq V(n)[-4]$.

We recall the isomorphism of equivariant $\ym(n)$-modules 
$W(n)^{\otimes 2} \simeq \Lambda^{2} W(n) \oplus S^{2} W(n)$, and the surjective map 
$B_{1}' : H_{1}(V(n), W(n)^{\otimes 2}) \rightarrow H_{2}(\ym(n), W(n))$ given in Theorem \ref{teo:exacta}
(see Proposition \ref{prop:s1inyectivo}). 
If we define the homology groups 
\[     H_{2,s}(\ym(n), W(n)) = \mathrm{Im}(B_{1}'|_{H_{0}(V(n),S^{2} W(n))})     \]
and 
\[     H_{2,a}(\ym(n), W(n)) = \mathrm{Im}(B_{1}'|_{H_{1}(V(n),\Lambda^{2} W(n))}),     \] 
it turns out that $H_{2}(\ym(n), W(n)) = H_{2,s}(\ym(n), W(n)) + H_{2,a}(\ym(n), W(n))$.

The following lemmas will be useful in the sequel. 
\begin{lemma}
\label{lema:morfb1}
The image under $B_{1}'$ of the homology classes of the cycles $x$ in $W(n)^{\otimes 2} \otimes V(n)$ 
considered in Corollary \ref{coro:ciclostor2} and Proposition \ref{prop:ciclostor1} is described as follows 
\begin{description}
\item[($V(n)$-type component)] Given $l$, with $1 \leq l \leq n$, and
\[     x = \sum_{1 \leq i < j \leq n} [x_{i},x_{j}] \otimes [x_{i},x_{j}] \otimes x_{l}
           + \sum_{1 \leq i, j \leq n} \big([x_{l},x_{i}] \otimes [x_{i},x_{j}] \otimes x_{j} + [x_{i},x_{j}] \otimes [x_{l},x_{i}] \otimes x_{j}\big),
\]
$B_{1}'(\bar{x})$ is the homology class of the cycle $\sum_{j = 1}^{n} [x_{l},x_{j}] \otimes x_{j} \in C_{2}(\YM(n),W(n))$.   

\item[($\Lambda^{3} V(n)$-type component)] Given $i_{1}, i_{2}, i_{3}$ such that $1 \leq i_{1} < i_{2} < i_{3} \leq n$, and 
\[     x = \sum_{l=1}^{n} \sum_{\sigma \in \SSS_{3}} \epsilon(\sigma) 
          \big([x_{i_{\sigma(1)}},x_{i_{\sigma(2)}}] \wedge [x_{i_{\sigma(3)}},x_{l}] \otimes x_{l}
          + [x_{i_{\sigma(1)}},x_{l}] \wedge [x_{i_{\sigma(2)}},x_{l}] \otimes x_{i_{\sigma(3)}}\big),
\]
$B_{1}'(\bar{x})$ is the homology class of the cycle 
$\sum_{\sigma \in \mathbb{A}_{3}} [x_{i_{\sigma(1)}},x_{i_{\sigma(2)}}] \otimes x_{i_{\sigma(3)}} \in C_{2}(\YM(n),W(n))$. 

\item[($\Lambda^{5} V(n)$-type component)] Given $i_{1}, i_{2}, i_{3}, i_{4}, i_{5}$ such that $1 \leq i_{1} < i_{2} < i_{3} < i_{4} < i_{5} \leq n$, and
\[     x = \sum_{\sigma \in \SSS_{5}} \epsilon(\sigma) [x_{i_{\sigma(1)}},x_{i_{\sigma(2)}}] \otimes
        [x_{i_{\sigma(3)}},x_{i_{\sigma(4)}}] \otimes x_{i_{\sigma(5)}},
\]
$B_{1}'(\bar{x})$ vanishes. 

\item[(generic-type component)] If $x$ is a generic cycle of the form 
\[     x = \sum_{l=1}^{n} \big(w x_{l} \otimes [x_{i},x_{j}] - w \otimes [x_{i},x_{j}] x_{l}\big) \otimes x_{l},     \]
then $B_{1}'(\bar{x})$ is the homology class of the cycle
$(w x_{j} \otimes x_{i} - w x_{i} \otimes x_{j}) \in C_{2}(\YM(n),W(n))$. 
\end{description}
\end{lemma}
\begin{proof}
Consider the double complex \eqref{eq:complejodoblesucesp}. 
It yields a spectral sequence converging to the homology $H_{\bullet} (\ym(n) , W(n)^{\otimes i})$ and an associated long exact sequence. 
Remark \ref{obs:diagram} indicates how to compute the map $B_{1}'$ for each cycle of the previous list. 
This is a lengthy but straightforward and usual homological computation. 
\end{proof}

Let us define a map 
\begin{align*}
\Delta : W(n) \rightarrow \YM(n)
\\
w \mapsto \sum_{i=1}^{n} [x_{i},[x_{i},w]].
\end{align*}
Since $q.w=0$, it turns out that 
\begin{equation}
\label{eq:delta}
     \Delta(w) = \sum_{i=1}^{n} x_{i} \rho_{i}^{2}(w) + \rho_{i}^{2}(x_{i}w) + \tilde{w},     
\end{equation}
where $\rho_{i}^{2}$ is given by \eqref{eq:acciontotal} and $\tilde{w} \in \oplus_{p\geq 3} \tym(n)^{p}$. 

\begin{lemma}
\label{lema:granconmutadorfinal}
If $p_{2} : \tym(n) \rightarrow \tym(n)^{2}$ denotes the canonical projection, then 
\begin{multline}
\label{eq:granconmutadorfinal}
     p_{2}\big(\sum_{l = 1}^{n} [x_{l},[x_{l},[x_{i_{1}}, \dots, [x_{i_{r}},[x_{p},x_{q}]]\dots]]]\big)
     \\
     = 2\sum_{h = 1}^{r} \sum_{l = 1}^{n} x_{l} x_{i_{1}} \dots x_{i_{h-1}}[[x_{l},x_{i_{h}}], x_{i_{h+1}} \dots x_{i_{r}} [x_{p},x_{q}]]
     + 2\sum_{l = 1}^{n} x_{i_{1}} \dots x_{i_{r}} [[x_{l},x_{p}],[x_{l},x_{q}]],
\end{multline}
for all $r \in \NN_{0}$. 
\end{lemma}
\begin{proof}
We proceed by induction. 
In the first place, using the Yang-Mills relations \eqref{eq:relym}, it is not hard to prove that 
\[     \Delta([x_{i},x_{j}]) = 2 \sum_{l = 1}^{n} [[x_{l},x_{i}],[x_{l},x_{j}]].     \]

We must now show a similar expression for $\Delta(\bar{x}^{\bar{i}}[x_{p},x_{q}])$, for $|\bar{i}| > 0$, 
where $\bar{x}_{\bar{i}} = x_{i_{1}} \dots x_{i_{r}}$ and $|\bar{i}| = r \geq 1$. 
In order to do so we consider the element in $\tym(n)$
\begin{equation}
\label{eq:granconmutador}
     \sum_{l = 1}^{n} [x_{l},[x_{l},[x_{i_{1}}, \dots, [x_{i_{r}},[x_{p},x_{q}]]\dots]]],
\end{equation}
whose component in $\tym(n)^{2}$ is obtained by replacing in the previous expression 
$[x_{t},\place]$ by $x_{t}.(\place) + \rho_{t}^{2}(\place)$, and which gives 
\begin{equation}
\label{eq:reduccion}
     \Delta(\bar{x}_{\bar{i}} [x_{p},x_{q}])
       + \sum_{h = 1}^{r} \sum_{l = 1}^{n} x_{l}^{2} x_{i_{1}} \dots x_{i_{h-1}} \rho_{i_{h}}^{2}(x_{i_{h+1}} \dots x_{i_{r}} [x_{p},x_{q}]).
\end{equation}

On the other hand, using the Jacobi identity and the Yang-Mills relations \eqref{eq:relym}, 
the element \eqref{eq:granconmutador} may be rewritten as 
\begin{equation}
\label{eq:granconmutador2}
\begin{split}
     &\sum_{l = 1}^{n} [x_{l},[x_{l},[x_{i_{1}}, \dots, [x_{i_{r}},[x_{p},x_{q}]]\dots]\dots]]
     \\
     &=\sum_{l = 1}^{n} \Big(\sum_{h = 1}^{r} [x_{l},[x_{i_{1}}, \dots,\big([[x_{l},x_{i_{h}}],\dots[x_{i_{r}},[x_{p},x_{q}]]\dots]
     + [x_{i_{r}},[x_{l},[x_{p},x_{q}]]\dots]\big)\dots]]\Big)
     \\
     &=\begin{aligned}[t]
         \sum_{l = 1}^{n} \Big(\sum_{h = 1}^{r} [x_{l},[x_{i_{1}}, \dots[[x_{l},x_{i_{h}}], \dots[x_{i_{r}},&[x_{p},x_{q}]]\dots]\dots]]
         + 2 [x_{i_{1}}, \dots, [x_{i_{r}},[[x_{l},x_{p}],[x_{l},x_{q}]]]\dots]
         \\
         + &\sum_{h = 1}^{r} [x_{i_{1}}, \dots[[x_{l},x_{i_{h}}], \dots[x_{i_{r}},[x_{l},[x_{p},x_{q}]]]\dots]\dots]\Big).
     \end{aligned}
\end{split}
\end{equation}
Furthermore, the same relations may be used to further simplify as follows 
\begin{align*}
     &\sum_{h = 1}^{r} \sum_{l = 1}^{n} [x_{l},[x_{i_{1}}, \dots,[[x_{l},x_{i_{h}}], \dots[x_{i_{r}},[x_{p},x_{q}]]\dots]\dots]]
       \\
       &= \sum_{h = 1}^{r} \sum_{l = 1}^{n} \Big(\sum_{g = 1}^{h}
          [x_{i_{1}}, \dots,[[x_{l},x_{i_{g}}], \dots[[x_{l},x_{i_{h}}], \dots[x_{i_{r}},[x_{p},x_{q}]]\dots]\dots]\dots]
       \\
       &+  \sum_{g = h+1}^{r}
          [x_{i_{1}}, \dots[[x_{l},x_{i_{h}}], \dots[[x_{l},x_{i_{g}}], \dots[x_{i_{r}},[x_{p},x_{q}]]\dots]\dots]\dots]
       \\
       &+ [x_{i_{1}}, \dots,\big([[x_{l},[x_{l},x_{i_{h}}]], \dots[x_{i_{r}},[x_{p},x_{q}]]\dots] 
       + [[x_{l},x_{i_{h}}], \dots[x_{i_{r}},[x_{l},[x_{p},x_{q}]]]\dots]\big)\dots]\Big),
\end{align*}
which also coincides with
\begin{align*}
       \sum_{h = 1}^{r} \sum_{l = 1}^{n} &\Big(\sum_{g = 1}^{h}
          [x_{i_{1}}, \dots,[[x_{l},x_{i_{g}}], \dots[[x_{l},x_{i_{h}}], \dots[x_{i_{r}},[x_{p},x_{q}]]\dots]\dots]\dots]
       \\
       &+  \sum_{g = h+1}^{r}
          [x_{i_{1}}, \dots[[x_{l},x_{i_{h}}], \dots[[x_{l},x_{i_{g}}], \dots[x_{i_{r}},[x_{p},x_{q}]]\dots]\dots]\dots]
       \\
       &+ [x_{i_{1}}, \dots,[[x_{l},x_{i_{h}}], \dots[x_{i_{r}},[x_{l},[x_{p},x_{q}]]]\dots]\dots]\Big).
\end{align*}
Hence,
\begin{align*}
   &\sum_{h = 1}^{r} \sum_{l = 1}^{n} [x_{l},[x_{i_{1}}, \dots,[[x_{l},x_{i_{h}}], \dots[x_{i_{r}},[x_{p},x_{q}]]\dots]\dots]]
   \\
   &- \sum_{h = 1}^{r} \sum_{l = 1}^{n}
    [x_{i_{1}}, \dots,[[x_{l},x_{i_{h}}], \dots[x_{i_{r}},[x_{l},[x_{p},x_{q}]]]\dots]\dots] \in \bigoplus_{p > 2}\tym(n)^{p}.
\end{align*}
Therefore, we have proved that 
\begin{align*}
     \sum_{l = 1}^{n} [x_{l},[x_{l},[x_{i_{1}}, \dots, [x_{i_{r}},[x_{p},x_{q}]]\dots]]]
     &= 2\sum_{h = 1}^{r} \sum_{l = 1}^{n} [x_{l},[x_{i_{1}}, \dots[[x_{l},x_{i_{h}}], \dots[x_{i_{r}},[x_{p},x_{q}]]\dots]\dots]]
     \\
     &+ 2\sum_{l = 1}^{n} [x_{i_{1}}, \dots, [x_{i_{r}},[[x_{l},x_{p}],[x_{l},x_{q}]]]\dots] +\tilde{w},
\end{align*}
where $\tilde{w} \in \bigoplus_{p > 2} \tym(n)^{p}$.
Then, we see that 
\begin{multline*}
     p_{2}\big(\sum_{l = 1}^{n} [x_{l},[x_{l},[x_{i_{1}}, \dots, [x_{i_{r}},[x_{p},x_{q}]]\dots]]]\big)
     \\
     = 2\sum_{h = 1}^{r} \sum_{l = 1}^{n} x_{l} x_{i_{1}} \dots [[x_{l},x_{i_{h}}], \dots x_{i_{r}} [x_{p},x_{q}]]
     + 2\sum_{l = 1}^{n} \bar{x}_{\bar{i}} [[x_{l},x_{p}],[x_{l},x_{q}]]
\end{multline*}
and the lemma is thus proved. 
\end{proof}

The following proposition is a consequence of the previous lemmas.
\begin{proposition}
\label{prop:proycan}
The restriction
\[     d^{1}_{-1,3}|_{H_{2,a}(\ym(n), W(n))} : H_{2,a}(\ym(n), W(n))
                                                   \rightarrow H_{1}(\ym(n), W(n)^{\otimes 2})     \]
is an injection.
Moreover, $B'_{1} \circ I_{2} \circ d^{1}_{-1,3} : H_{2}(\ym(n), W(n)) \rightarrow H_{2}(\ym(n), W(n))$
coincides with the map $\pi_{a} : H_{2}(\ym(n), W(n)) \rightarrow H_{2,a}(\ym(n), W(n))$ induced by the canonical projection 
$W(n)^{\otimes 2} \rightarrow \Lambda^{2} W(n)$, 
where $I_{2}$ is the map defined in Theorem \ref{teo:exacta}. 
\end{proposition}
\begin{proof}
It suffices to prove the second statement, since the former is a direct consequence of the latter. 

The fact that $B_{1}'$ is surjective implies that $H_{2}(\ym(n),W(n))$ is generated as $k$-module by the images under $B_{1}'$ of 
the cycles given in Lemma \ref{lema:morfb1}. 
First, the cycles corresponding to the $V(n)$-type component belong to the symmetric part $H_{2,s}(\ym(n),W(n))$ 
and they vanish when we apply $d_{-1,3}^{1}$ to them, 
since $d_{-1,3}^{1}$ corresponds to the composition of the map given by
\begin{equation}
\label{eq:b1v}
\begin{split}
       \sum_{j=1}^{n} [x_{l},x_{j}] \otimes x_{j}
       &\mapsto \sum_{j,p=1}^{n} \big([[[x_{l},x_{j}],x_{p}],x_{p}] \otimes x_{j} - 2 [[[x_{l},x_{j}],x_{p}],x_{j}] \otimes x_{p} 
         + [[[x_{l},x_{j}],x_{j}],x_{p}] \otimes x_{p}\big)
       \\  
       &= \sum_{j,p=1}^{n} \big(2 [[x_{l},x_{p}],[x_{j},x_{p}]] \otimes x_{j} - 2 [[x_{l},x_{j}],[x_{p},x_{j}]] \otimes x_{p}\big) 
       = 0 
\end{split}   
\end{equation}
and taking the class in $I^{2}/I^{3} \simeq W(n)^{\otimes 2}$. 
Notice that we have used the Jacobi identity and the Yang-Mills relations. 

Second, the cycles of the $\Lambda^{3} V(n)$-type component belong to $H_{2,a}(\ym(n),W(n))$. 
In~order to obtain the image under $d_{-1,3}^{1}$ of a representative of this component, 
we see that $d_{-1,3}^{1}$ is induced by the class in $I^{2}/I^{3}$ of
\begin{equation}
\label{eq:b1wedge3v}
\begin{split}
\MoveEqLeft
       \sum_{\sigma \in \mathbb{A}_{3}} [x_{i_{\sigma(1)}},x_{i_{\sigma(2)}}] \otimes x_{i_{\sigma(3)}}
       \\
       &\begin{multlined}[t][0.8\displaywidth] 
         \mapsto \sum_{\sigma \in \mathbb{A}_{3}} \sum_{p=1}^{n} 
         \Big([[[x_{i_{\sigma(1)}},x_{i_{\sigma(2)}}],x_{p}],x_{p}] \otimes x_{i_{\sigma(3)}} 
         - 2 [[[x_{i_{\sigma(1)}},x_{i_{\sigma(2)}}],x_{p}],x_{i_{\sigma(3)}}] \otimes x_{p}
         \\
         + [[[x_{i_{\sigma(1)}},x_{i_{\sigma(2)}}],x_{i_{\sigma(3)}}],x_{p}] \otimes x_{p}\Big) 
         \end{multlined}
       \\
       &= \sum_{\sigma \in \mathbb{A}_{3}} \sum_{p=1}^{n} 
       \big( 2 [[x_{i_{\sigma(1)}},x_{p}],[x_{i_{\sigma(2)}},x_{p}]] \otimes x_{i_{\sigma(3)}} 
       - 2 [[x_{i_{\sigma(1)}},x_{i_{\sigma(2)}}],[x_{p},x_{i_{\sigma(3)}}]] \otimes x_{p}\big)
       \\
       &= \sum_{\sigma \in \SSS_{3}} \sum_{p=1}^{n} \epsilon(\sigma)
           \Big( [x_{i_{\sigma(1)}},x_{i_{\sigma(2)}}] \wedge [x_{i_{\sigma(3)}},x_{p}] \otimes x_{p}
        + [x_{i_{\sigma(1)}},x_{p}] \wedge [x_{i_{\sigma(2)}},x_{p}] \otimes x_{i_{\sigma(3)}} \Big),
\end{split}
\end{equation}
where we have again used the Jacobi identity and the Yang-Mills relations. 
By Lemma \ref{lema:morfb1}, we conclude that $B'_{1} \circ I_{2} \circ d_{-1,3}^{1}$ applied to 
$\sum_{\sigma \in \mathbb{A}_{3}} [x_{i_{\sigma(1)}},x_{i_{\sigma(2)}}] \otimes x_{i_{\sigma(3)}}$ is the identity. 

Since the image under $B_{1}'$ of the representatives of the $\Lambda^{5} V(n)$-type component vanishes, 
it is not necessary to consider it.

Finally, we shall prove that $B'_{1} \circ I_{2} \circ d_{-1,3}^{1} \circ B_{1}' \circ \times^{\text{\begin{tiny}$M(n)$\end{tiny}}}_{0,1} 
= \pi_{a} \circ B_{1}' \circ \times^{\text{\begin{tiny}$M(n)$\end{tiny}}}_{0,1}$. 
In order to do so, it suffices to prove the previous identity only for the case of cycles 
$w \otimes [x_{i},x_{j}] \in H_{0}(V(n),W(n)^{\otimes 2})$, 
with $w = x_{i_{1}} \dots x_{i_{r}} [x_{p},x_{q}]$. 
Taking into account that $d_{-1,3}^{1} \circ B_{1}' \circ \times^{\text{\begin{tiny}$M(n)$\end{tiny}}}_{0,1}$ 
is given by the class in $I^{2}/I^{3}$ of the map that sends $w \otimes [x_{i},x_{j}]$ to
\begin{align*}
     &\sum_{l=1}^{n} \Big(2([x_{j},[x_{l},[x_{i},[x_{i_{1}}, \dots,[x_{i_{r}},[x_{p},x_{q}]]\dots]]]] \otimes x_{l} 
     \\
     &\phantom{\sum_{l=1}^{n}}- [x_{i},[x_{l},[x_{i},[x_{i_{1}}, \dots,[x_{i_{r}},[x_{p},x_{q}]]\dots]]]] \otimes x_{l})
     \\
     &\phantom{\sum_{l=1}^{n}}+ [x_{l},[x_{l},[x_{j},[x_{i_{1}}, \dots,[x_{i_{r}},[x_{p},x_{q}]]\dots]]]] \otimes x_{i}  
     - [x_{l},[x_{l},[x_{i},[x_{i_{1}}, \dots,[x_{i_{r}},[x_{p},x_{q}]]\dots]]]] \otimes x_{j}   
     \\
     &\phantom{\sum_{l=1}^{n}}- [x_{l},\big([x_{j},[x_{i},[x_{i_{1}}, \dots,[x_{i_{r}},[x_{p},x_{q}]]\dots]]]
     + [x_{i},[x_{i},[x_{i_{1}}, \dots,[x_{i_{r}},[x_{p},x_{q}]]\dots]]]\big)] \otimes x_{l} \Big),     
\end{align*}
we see that $d_{-1,3}^{1} \circ B_{1}' \circ \times^{\text{\begin{tiny}$M(n)$\end{tiny}}}_{0,1} (w \otimes [x_{i},x_{j}])$ is 
\begin{equation}
\label{eq:conmutaparte}
\begin{split}
      &\sum_{l=1}^{n} 
      \Big(2 \sum_{h=0}^{r} 
      \big(x_{l} \underset{i_{0}=j}{\underbrace{x_{i_{0}}}} \dots [[x_{l},x_{i_{h}}], x_{i_{h+1}}\dots x_{i_{r}} [x_{p},x_{q}]] \otimes x_{i}
      \\
      &\phantom{\sum_{l=1}^{n} \Big(2 \sum_{h=0}^{r}} 
      - x_{l} \underset{i_{0}=i}{\underbrace{x_{i_{0}}}} \dots [[x_{l},x_{i_{h}}], x_{i_{h+1}}\dots x_{i_{r}} [x_{p},x_{q}]] \otimes x_{j} \big)
      \\
      &\phantom{\sum_{l=1}^{n}}- 2 x_{i} x_{i_{1}} \dots x_{i_{r}} [[x_{l},x_{p}],[x_{l},x_{p}]] \otimes x_{j}
      + 2 x_{j} x_{i_{1}} \dots x_{i_{r}} [[x_{l},x_{p}],[x_{l},x_{p}]] \otimes x_{i}
      \\
      &\phantom{\sum_{l=1}^{n}}+ 2 [[x_{j},x_{l}], x_{i} x_{i_{1}} \dots x_{i_{r}} [x_{p},x_{q}]] \otimes x_{l}
      - 2 [[x_{i},x_{l}],x_{j} x_{i_{1}} \dots x_{i_{r}}[x_{p},x_{q}]] \otimes x _{l} 
      \\
      &\phantom{\sum_{l=1}^{n}}- x_{l}[[x_{i},x_{j}], x_{i_{1}} \dots x_{i_{r}}[x_{p},x_{q}]] \otimes x _{l} \Big),         
\end{split}
\end{equation}
where we have used the identity \eqref{eq:granconmutadorfinal} and the following equality 
\begin{gather*}
      2 [x_{j},[x_{l},[x_{i},[x_{i_{1}}, \dots,[x_{i_{r}},[x_{p},x_{q}]]\dots]]]] \otimes x_{l} 
      -  [x_{l},[x_{j},[x_{i},[x_{i_{1}}, \dots,[x_{i_{r}},[x_{p},x_{q}]]\dots]]]] \otimes x_{l}  
      \\
      -2 [x_{i},[x_{l},[x_{j},[x_{i_{1}}, \dots,[x_{i_{r}},[x_{p},x_{q}]]\dots]]]] \otimes x_{l} 
      +  [x_{l},[x_{i},[x_{j},[x_{i_{1}}, \dots,[x_{i_{r}},[x_{p},x_{q}]]\dots]]]] \otimes x_{l}
      \\
      =2 [[x_{j},x_{l}],[x_{i},[x_{i_{1}}, \dots,[x_{i_{r}},[x_{p},x_{q}]]\dots]]] \otimes x_{l}
      -2 [[x_{i},x_{l}],[x_{j},[x_{i_{1}}, \dots,[x_{i_{r}},[x_{p},x_{q}]]\dots]]] \otimes x_{l} 
      \\
      -  [x_{l},[[x_{i},x_{j}],[x_{i_{1}}, \dots,[x_{i_{r}},[x_{p},x_{q}]]\dots]]] \otimes x_{l}.
\end{gather*}
We may rewrite $d_{-1,3}^{1} \circ B_{1}' \circ \times^{\text{\begin{tiny}$M(n)$\end{tiny}}}_{0,1} (w \otimes [x_{i},x_{j}])$ 
given in \eqref{eq:conmutaparte} as the sum of a coboundary 
\begin{align*}
     &\sum_{l=1}^{n} \Big(
     d_{2}^{\mathrm{CE}}
     \big(2 [[x_{l},x_{i}], x_{i_{1}} \dots x_{i_{r}} [x_{p},x_{q}]] \otimes x_{l} \wedge x_{j} 
     - 2 [[x_{l},x_{j}], x_{i_{1}} \dots x_{i_{r}} [x_{p},x_{q}]] \otimes x_{l} \wedge x_{i}\big) 
     \\
     &\phantom{\sum_{l=1}^{n}}
     + 2 d_{2}^{\mathrm{CE}}
     \big(\sum_{h=1}^{r} x_{l} x_{i_{1}} \dots x_{i_{h-1}} [[x_{l},x_{i_{h}}],x_{i_{h+1}} \dots x_{i_{r}} [x_{p},x_{q}]] \otimes x_{i} \wedge x_{j}\big)
     \\
     &\phantom{\sum_{l=1}^{n}}+ 2 d_{2}^{\mathrm{CE}}\big(x_{i_{1}} \dots x_{i_{r}} [[x_{l},x_{p}],[x_{l},x_{p}]] \otimes x_{i} \wedge x_{j}\big)\Big)
\end{align*}
in $C_{1}(V(n),W(n)^{\otimes 2})$ and a cycle in $C_{1}(\ym(n),W(n)^{\otimes 2})$
\begin{align*}
     &\sum_{l=1}^{n} \Big(
     - 2 [x_{j}[x_{l},x_{i}], x_{i_{1}} \dots x_{i_{r}} [x_{p},x_{q}]] \otimes x_{l} 
     + 2 [x_{i}[x_{l},x_{j}], x_{i_{1}} \dots x_{i_{r}} [x_{p},x_{q}]] \otimes x_{l}
     \\
     &- x_{l} [[x_{i},x_{j}], x_{i_{1}} \dots x_{i_{r}} [x_{p},x_{q}]] \otimes x_{l} 
     + 2 [[x_{l},x_{i}], x_{i_{1}} \dots x_{i_{r}} [x_{p},x_{q}]] \otimes [x_{l}, x_{j}]
     \\
     &+ 2 \sum_{h=1}^{r} x_{l} x_{i_{1}} \dots x_{i_{h-1}} [[x_{l},x_{i_{h}}],x_{i_{h+1}} \dots x_{i_{r}} [x_{p},x_{q}]] \otimes [x_{i}, x_{j}]
     \\
     &- 2 [[x_{l},x_{j}], x_{i_{1}} \dots x_{i_{r}} [x_{p},x_{q}]] \otimes [x_{l}, x_{i}] 
     + 2 x_{i_{1}} \dots x _{i_{r}} [[x_{l},x_{p}],[x_{l},x_{q}]] \otimes [x_{i},x_{j}] 
     \Big).
\end{align*}     
Using the description of $I_{2}$ given in Remark \ref{obs:morfismossucesp}, we get that $I_{2} \circ d_{-1,3}^{1} \circ B_{1}' \circ \times^{\text{\begin{tiny}$M(n)$\end{tiny}}}_{0,1} (w \otimes [x_{i},x_{j}])$ is the class of the cycle in $C_{1}(V(n),W(n)^{\otimes 2})$ 
given by 
\begin{equation}
\label{eq:imporcom*}
\begin{split}
     \sum_{l=1}^{n}
     \Big(&
     2 [x_{i}[x_{l},x_{j}], x_{i_{1}} \dots x_{i_{r}} [x_{p},x_{q}]] \otimes x_{l} 
     - 2 [x_{j}[x_{l},x_{i}], x_{i_{1}} \dots x_{i_{r}} [x_{p},x_{q}]] \otimes x_{l} 
     \\
     - &x_{l} [[x_{i},x_{j}], x_{i_{1}} \dots x_{i_{r}} [x_{p},x_{q}]] \otimes x_{l} \Big),
\end{split}
\end{equation}
which is equal to 
\begin{equation}
\label{eq:imporcom2*}
     \sum_{l=1}^{n} \Big( 
       ([x_{l} [x_{i},x_{j}], x_{i_{1}} \dots x_{i_{r}} [x_{p},x_{q}]] 
          - [[x_{i},x_{j}], x_{l} x_{i_{1}} \dots x_{i_{r}} [x_{p},x_{q}]] ) \otimes x_{l} \Big)
\end{equation}   
and may be simplified to give the following expression for 
$I_{2} \circ d_{-1,3}^{1} \circ B_{1}' \circ \times^{\text{\begin{tiny}$M(n)$\end{tiny}}}_{0,1} 
(w \otimes [x_{i},x_{j}])$:
\begin{equation}
\label{eq:imporcom}
     \sum_{l=1}^{n} \Big(
        (w \wedge [x_{i},x_{j}] x_{l} - w x_{l} \wedge [x_{i},x_{j}]) \otimes x_{l} \Big). 
\end{equation}   
The proposition is thus proved.
\end{proof}

Let us consider the Lie algebra $\h(n) = \tym(n)/\C^{2}(\tym(n))$, where $\C^{2}(\tym(n))$ is the second step 
of the lower central series of $\tym(n)$. 
There is an isomorphism $\h(n) \simeq W(n) \oplus \Lambda^{2} W(n)$ as graded vector spaces. 
Besides, $\Lambda^{2} W(n) = \Z(\h(n))$ and $[\hskip 0.6mm,] : W(n) \wedge W(n) \rightarrow \Lambda^{2} W(n)$ 
is a homogeneous isomorphism of degree $0$. 
The Lie algebra $\tym(n)$ being free, $\h(n)$ is a free nilpotent Lie algebra of nilpotency index equal to $2$.

Since the adjoint action of $\ym(n)$ on $\tym(n)$ induces an action on the quotient $\h(n)$, 
the graded vector space $S^{2} \h(n)$ has a natural graded action of $\ym(n)$. 
We will denote $D(\h(n))$ the graded vector space $(S^{2} \h(n))_{\tym(n)}$.
Hence, $D(\h(n))$ is provided with a graded action of $\ym(n)$ such that $\tym(n)$ vanishes, and, in consequence,
the graded action of $\ym(n)$ on $D(\h(n))$ in turn induces a graded action of $V(n) = \ym(n)/\tym(n)$ on $D(\h(n))$.

If $a, b \in \h(n)$, we shall denote  by $a \circ b$ the class of $a \otimes_{s} b = (a \otimes b + b \otimes a)/2 \in S^{2}\h(n)$ 
in $D(\h(n))$. 
In this case,
\[     x_{i}.(a \circ b) = [x_{i},a] \circ b + a \circ [x_{i},b].     \]

\begin{proposition}
\label{prop:sucexactcort}
There is a short exact sequence of graded $S(V(n))$-modules
\[     0 \rightarrow \Lambda^{3} W(n) \overset{\alpha}{\rightarrow} D(\h(n))
                              \overset{\beta}{\rightarrow} S^{2} W(n) \rightarrow 0,     \]
where $\beta$ is induced by the natural projection of $S^{2} \h(n) \rightarrow S^{2} W(n)$ and $\alpha$ is given by
\[     w_{1} \wedge w_{2} \wedge w_{3} \mapsto \bar{w}_{1} \circ \overline{[w_{2},w_{3}]},     \]
for $w_{1}, w_{2}, w_{3} \in W(n)$.
\end{proposition}
\begin{proof}
We know that $\h(n) \simeq W(n) \oplus \Lambda^{2} W(n)$ as graded vector spaces.
Therefore, $D(\h(n))$ is a quotient of $S^{2} \h(n) \simeq S^{2} W(n) \oplus (W(n) \otimes \Lambda^{2} W(n)) \oplus S^{2} \Lambda^{2} W(n)$
(isomorphism of graded vector spaces).

However, the last direct summand belongs to $\tym(n).S^{2}\h(n)$.
This follows from the fact that, given $v, v', w, w' \in W(n)$, then
\[     v.(\bar{w} \circ \overline{[v' , w']}) = \overline{[v,w]} \circ \overline{[v',w']} + \bar{w} \circ \overline{[v,[v',w']]}
                                                  = \overline{[v,w]} \circ \overline{[v',w']},     \]
since $[v,[v',w']] \in \C^{2}(\tym(n))$.
Also, by degree reasons, 
we see that the subset of $D(\h(n))$ defined by the classes of the elements of the graded vector space 
$W(n) \otimes \Lambda^{2} W(n) \subseteq S^{2} \h(n)$ is in fact a graded $S(V(n))$-submodule of $D(\h(n))$.

Since $\tym(n)$ is a free Lie algebra generated by $W(n)$, it is clear by an internal weight argument that 
$S^{2} W(n) \cap \tym(n).S^{2} \h(n) = 0$.
As a consequence, the projection $p : D(\h(n)) \rightarrow S^{2} W(n)$ is a $k$-linear epimorphism.
It is also homogeneous $S(V(n))$-linear of degree $0$, 
for the classes in $D(\h(n))$ of the elements of $W(n) \otimes \Lambda^{2} W(n) \subseteq S^{2} \h(n)$
form a graded $S(V(n))$-module and $p$ is the natural projection given by the quotient by this submodule.

On the other hand, given $v \otimes_{s} [w,w']$ in the component $W(n) \otimes \Lambda^{2} W(n) \subseteq S^{2} \h(n)$, it turns out that 
\[     v \otimes_{s} [w,w'] = - v \otimes_{s} [w',w] = - w' \otimes_{s} [w,v] + w.(v \otimes_{s} w').     \]
Therefore, the map $\alpha$ is well-defined, and it is readily verified to be $S(V(n))$-linear.
Furthermore, $\mathrm{Im}(\alpha)$ coincides with the collection of classes in $D(\h(n))$ of the elements in 
$W(n) \otimes \Lambda^{2} W(n) \subseteq S^{2} \h(n)$.
The injectivity of $\alpha$ follows from the fact that $\tym(n)$ is a free Lie algebra generated by $W(n)$. 
\end{proof}

From now on, we shall identify $\Lambda^{3} W(n)$ with the image of $\alpha$ in $D(\h(n))$.
Moreover, by the previous proposition, we will not write the bars denoting class for the elements of $W(n)$, and hence,
we shall often write $w_{1} \circ [w_{2},w_{3}] \in \Lambda^{3} W(n)$ instead of $w_{1} \wedge w_{2} \wedge w_{3}$.

The previous short exact sequence implies that there exists a map in homology of the form  
\begin{equation}
\label{eq:mordelta}
     \delta : H_{1} (V(n),S^{2} W(n)) \rightarrow H_{0}(V(n),\Lambda^{3} W(n)),     
\end{equation}
which by the Snake Lemma is induced by 
\begin{equation}
\label{eq:morfdelta}
   \sum_{i=1}^{n} w_{i} \otimes_{s} w'_{i} \otimes x_{i} \mapsto
    \sum_{i=1}^{n} (\rho_{i}^{2}(w_{i}) \circ w'_{i} + \rho_{i}^{2}(w'_{i}) \circ w_{i}),
\end{equation}
where we have used the notation 
%
given in \eqref{eq:acciontotal}, then 

\begin{remark}
\label{rem:obsdelta}
The map $\delta$ is naturally obtained in the following way. 
From the short exact sequence of $\ym(n)$-modules given by 
\begin{equation}
\label{eq:sucasoc}
     0 \rightarrow W(n)^{\otimes 2} \rightarrow I/I^{3} \rightarrow W(n) \rightarrow 0,     
\end{equation}
where $I = \mathrm{ker}(\epsilon_{\tym(n)})$, we obtain the long exact sequence in homology 
\[     \dots \rightarrow H_{p}(\ym(n), W(n)^{\otimes 2}) \rightarrow H_{p}(\ym(n), I/I^{3}) \rightarrow H_{p}(\ym(n), W(n)) \rightarrow 
                         H_{p-1}(\ym(n), W(n)^{\otimes 2}) \rightarrow \dots     \]
It is clear that $d^{1}_{-1,3}$ coincides with the map $H_{2}(\ym(n), W(n)) \rightarrow H_{1}(\ym(n), W(n)^{\otimes 2})$ in the long exact sequence. 
By Remark \ref{rem:accionsvntor} and the functoriality of the Hochschild-Serre spectral sequence, there are maps between the long exact sequences 
\eqref{eq:longexacseq} for $X = W(n)^{2}$, $X = I/I^{3}$ and $X = W(n)$ in the obvious way, which could be represented as a bicomplex. 
In particular, there is a map $H_{1}(V(n),W(n)^{\otimes 2}) \rightarrow H_{0}(V(n),W(n)^{\otimes 3})$. 
It induces thus a morphism $H_{1}(V(n),S^{2}W(n)) \rightarrow H_{0}(V(n),\Lambda^{3}W(n))$, which coincides with $\delta$. 
\qed
\end{remark}

On the other hand, the restriction of the map $B_{2} : H_{0}(V(n),W(n)^{\otimes 3}) \rightarrow H_{1}(\ym(n),W(n)^{\otimes 2})$ 
of the exact sequence of Theorem \ref{teo:exacta} to $H_{0}(V(n),\Lambda^{3}W(n))$ is induced by (see Remark \ref{obs:morfismossucesp})
\begin{equation}
\label{eq:morfb2}
     w_{1} \wedge w_{2} \wedge w_{3} \mapsto w_{1} \wedge w_{2} \otimes w_{3} + w_{2} \wedge w_{3} \otimes w_{1} + w_{3} \wedge w_{1} \otimes w_{2}.
\end{equation}

\begin{proposition}
\label{prop:diagramaconmutativo}
If $\delta : H_{1} (V(n),S^{2} W(n)) \rightarrow H_{0}(V(n),\Lambda^{3} W(n))$ is the map \eqref{eq:mordelta}, 
the following diagram is commutative 
\[
\xymatrix
{
H_{1}(V(n),S^{2}W(n))
\ar[r]^{\delta}
\ar[d]^{B'_{1}}
&
H_{0}(V(n),\Lambda^{3} W(n))
\ar[d]^{B_{2}}
\\
H_{2}(\ym(n),W(n))
\ar[r]^{d^{1}_{-1,3}}
&
H_{1}(\ym(n), W(n)^{\otimes 2})
}
\]
\end{proposition}
\begin{proof}
It is clear from the expressions of the maps given by \eqref{eq:morfdelta}, \eqref{eq:morfb2}, \eqref{eq:imporcom}, 
\eqref{eq:b1v}, \eqref{eq:b1wedge3v} and Lemma~\ref{lema:morfb1}. 
\end{proof}

\begin{proposition}
\label{prop:b1b2}
The restriction of $B'_{1}$ to $H_{1}(V(n),S^{2}W(n))$ has kernel isomorphic to $\Lambda^{5} V(n)$ and the 
restriction of $B_{2}$ to $H_{0} (V(n),\Lambda^{3} W(n))$ is injective.
\end{proposition}
\begin{proof}
Let us first show that $\Ker(B'_{1}|_{H_{1}(V(n),S^{2}V(n))}) = \Lambda^{5} V(n)$.
Lemma \ref{lema:morfb1} tells us that $\Lambda^{5} V(n)$ is included in $\Ker(B'_{1})$. 
Furthermore, the expression of the cycles in Proposition \ref{prop:ciclostor1} implies that 
$\Lambda^{5} V(n) \subseteq H_{1}(V(n),S^{2}W(n))$.

On the other hand, by the long exact sequence of Theorem \ref{teo:exacta}, $\Ker(B'_{1}) = \mathrm{Im}(S'_{1})$.
Using the same theorem, we derive that $S'_{1}$ is injective and $H_{3}(V(n),W(n)) \simeq \Lambda^{5} V(n)$, so 
$\Ker(B'_{1}) \simeq \Lambda^{5} V(n)$.
Therefore, the restriction of the morphism $B'_{1}$ to $H_{1}(V(n),S^{2}W(n))$ has kernel $\Lambda^{5} V(n)$.

Let us now prove that the map $B_{2}|_{H_{0} (V(n),\Lambda^{3} W(n))}$ is injective.
On the one hand, the exact sequence of Theorem \ref{teo:exacta} tells us that 
$\Ker(B_{2}) = \mathrm{Im}(S_{2}) = S_{2}(H_{2}(V(n),W(n)^{\otimes 2}))$.
By the same theorem, $H_{2}(V(n),W(n)^{\otimes 2}) \simeq \Lambda^{6} V(n)$. 
In fact, using the same ideas explained in Proposition \ref{prop:s1inyectivo}, the cycles 
\[     \Big\{ \sum_{\sigma \in \SSS_{6}} \epsilon(\sigma) [x_{i_{\sigma(1)}}, x_{i_{\sigma(2)}}] \otimes [x_{i_{\sigma(3)}}, x_{i_{\sigma(4)}}]
       \otimes x_{i_{\sigma(5)}} \wedge x_{i_{\sigma(6)}}
       : 1 \leq i_{1} < i_{2} < i_{3} < i_{4} < i_{5} < i_{6} \leq n \Big\}     \]
form a basis of the homology group $H_{2}(V(n),W(n)^{\otimes 2})$. 
This is due to the fact that, by the double complex \eqref{eq:complejodoblesucesp}, Remark \ref{obs:morfismossucesp} 
and standard computations on the second term of a spectral sequence, 
the map $\Lambda^{6} V(n) = H_{6}(V(n) , k) \hookrightarrow H_{2}(V(n),W(n)^{\otimes 2})$ is induced by 
\[     x_{i_{1}} \wedge x_{i_{2}} \wedge x_{i_{3}} \wedge x_{i_{4}} \wedge x_{i_{5}} \wedge x_{i_{6}} \mapsto
       \sum_{\sigma \in \SSS_{6}} \epsilon(\sigma) [x_{i_{\sigma(1)}}, x_{i_{\sigma(2)}}] \otimes [x_{i_{\sigma(3)}}, x_{i_{\sigma(4)}}]
       \otimes x_{i_{\sigma(5)}} \wedge x_{i_{\sigma(6)}} ,      \]
where the last element is a cycle in $W(n)^{\otimes 2} \otimes \Lambda^{2} V(n)$. 
Similar arguments as those used for the map $S_{1}$ in Proposition \ref{prop:s1inyectivo} assure that 
the image under the morphism $S_{2}$ of an element of this basis is the class of the cycle 
\[     \sum_{\sigma \in \SSS_{6}}
       \epsilon(\sigma) [x_{i_{\sigma(1)}}, x_{i_{\sigma(2)}}] \otimes [x_{i_{\sigma(3)}}, x_{i_{\sigma(4)}}]
       \otimes [x_{i_{\sigma(5)}},x_{i_{\sigma(6)}}].     \]
Hence, $\mathrm{Im}(S_{2}) \subseteq H_{0}(V(n),S^{3}W(n))$. 
Taking into account that $S^{3}W(n) \cap \Lambda^{3} W(n) = \{ 0 \}$, it turns out that the kernel of the map $B_{2}|_{H_{0}(V(n),\Lambda^{3}W(n))}$ vanishes.
\end{proof}

\begin{proposition}
\label{prop:vnker}
The kernel of the map $\delta : H_{1}(V(n),S^{2}W(n)) \rightarrow H_{0}(V(n),\Lambda^{3}W(n))$ given in \eqref{eq:morfdelta}
is $V(n)[-4] \oplus \Lambda^{5} V(n)$.
\end{proposition}
\begin{proof}
It is clear from Propositions \ref{prop:diagramaconmutativo} and \ref{prop:b1b2} that $\Lambda^{5}V(n) \subseteq \Ker(\delta)$. 
Also, the expression for the cycles of $V(n)$ given in Lemma \ref{lema:morfb1} 
and the expression for $\delta$ given in \eqref{eq:morfdelta} tell us that $V(n)[-4] \subseteq \Ker(\delta)$. 
The change of degree comes from the fact that, in Lemma \ref{lema:morfb1}, we have studied the homology of $M(n)$. 
The elements given by the cycles of $(\Lambda^{3} V(n))[-2]$ belong to $H_{1}(V(n),\Lambda^{2}W(n))$ and therefore are not in the kernel 
of the map $\delta$.

In consequence, it suffices to prove that $\delta$ is injective if restricted to the subspace spanned by the generic elements of 
$H_{1}(V(n),S^{2}W(n))$, which is rather tedious.
We shall anyway include the proof of this fact because it is quite non evident.

Since $\delta$ is a homogeneous morphism of degree $0$, it suffices to restrict to homogeneous generic elements. 

Now, we may consider a generic element of $H_{1}(V(n),S^{2}W(n))$ which we shall assume to be 
the homology class of a $k$-linear combination of cycles of the form 
\begin{equation}
\label{eq:ciclossimgen}
\begin{split}
     \sum_{l=1}^{n} \Big( &\big([x_{i},x_{j}] x_{l} \otimes \bar{x}_{\bar{i}} [x_{p},x_{q}]
                    - [x_{i},x_{j}] \otimes \bar{x}_{\bar{i}} [x_{p},x_{q}]x_{l}\big) \otimes x_{l}
     \\
                    + &\big(\bar{x}_{\bar{i}} [x_{p},x_{q}] \otimes  [x_{i},x_{j}] x_{l}
                    - \bar{x}_{\bar{i}} [x_{p},x_{q}] x_{l} \otimes [x_{i},x_{j}]\big) \otimes x_{l}  \Big),
\end{split}
\end{equation}
for $|\bar{i}| \geq 0$, and which is obtained from the map $\times^{\text{\begin{tiny}$M(n)$\end{tiny}}}_{0,1}$ applied to the element 
\begin{equation}
\label{eq:ciclossimgen0}
\begin{split}
     \big([x_{i},x_{j}] \otimes \bar{x}_{\bar{i}} [x_{p},x_{q}] - \bar{x}_{\bar{i}} [x_{p},x_{q}] \otimes [x_{i},x_{j}]\big)
       \otimes \big(\sum_{l=1}^{n} (x_{l} \otimes 1 - 1 \otimes x_{l}) \otimes x_{l}\big)
\end{split}
\end{equation}
in $C_{0}(V(n),\Lambda^{2} W(n)) \otimes C_{1}(V(n),(S(V(n))/\cl{q})^{\otimes 2})$. 
We shall denote $\tilde{c} = [x_{i},x_{j}] \otimes \bar{x}_{\bar{i}} [x_{p},x_{q}] - \bar{x}_{\bar{i}} [x_{p},x_{q}] \otimes [x_{i},x_{j}]$. 
We notice that the generic elements \eqref{eq:ciclossimgen} have degree greater than or equal to $6$. 

Note that $\times^{\text{\begin{tiny}$M(n)$\end{tiny}}}_{0,1}$ ``interchanges parity'' between $H_{0}(V(n),W(n)^{\otimes 2})$ and $H_{1}(V(n),W(n)^{\otimes 2})$,
\textit{i.e.} it sends $H_{0}(V(n),\Lambda^{2} W(n))$ to $H_{1}(V(n),S^{2} W(n))$
and $H_{0}(V(n),S^{2} W(n))$ to $H_{1}(V(n),\Lambda^{2} W(n))$.

Let $c$ be the cycle given by \eqref{eq:ciclossimgen} and $\bar{c}$ its homology class. 
If we use the identity given in \eqref{eq:morfdelta}, then $\delta(\bar{c})$ is given by the cycle 
\begin{equation}
\begin{split}
     \sum_{l=1}^{n} \Big( &
                                               \rho_{l}^{2}([x_{i},x_{j}] x_{l}) \circ \bar{x}_{\bar{i}} [x_{p},x_{q}]
                          + 
                                               [x_{i},x_{j}] x_{l} \circ \rho_{l}^{2}((\bar{x}_{\bar{i}} [x_{p},x_{q}])
     \\
                    - & 
                                              \rho_{l}^{2}([x_{i},x_{j}]) \circ \bar{x}_{\bar{i}} [x_{p},x_{q}] x_{l}
                    - 
                                              [x_{i},x_{j}] \circ \rho_{l}^{2}(\bar{x}_{\bar{i}} [x_{p},x_{q}] x_{l})
                    \label{eq:ciclossimgendelta}
     \\
                    + &
                                              \rho_{l}^{2}(\bar{x}_{\bar{i}} [x_{p},x_{q}]) \circ [x_{i},x_{j}] x_{l}
                    + 
                                              \bar{x}_{\bar{i}}[x_{p},x_{q}] \circ \rho_{l}^{2}([x_{i},x_{j}] x_{l})
     \\
                    - &
                                              \rho_{l}^{2}(\bar{x}_{\bar{i}} [x_{p},x_{q}] x_{l}) \circ [x_{i},x_{j}]
                    - 
                                              \bar{x}_{\bar{i}} [x_{p},x_{q}] x_{l} \circ \rho_{l}^{2}([x_{i},x_{j}])
                                              \Big),
\end{split}
\end{equation}
and we denote by $a_{j}^{l}$ and $b_{j}^{l}$ the first and second summand (without signs), respectively, of the $j$-th line in the previous equation. 
Since
\begin{align*}
  a_{1}^{l} - a_{2}^{l} &= \big(\rho_{l}^{2}([x_{i},x_{j}] x_{l}) + \rho_{l}^{2}([x_{i},x_{j}]) x_{l}\big) \circ \bar{x}_{\bar{i}} [x_{p},x_{q}]
                         - d_{1}^{\mathrm{CE}}\big(\rho_{l}^{2}([x_{i},x_{j}]) \circ \bar{x}_{\bar{i}} [x_{p},x_{q}] \otimes x_{l}\big),
  \\
  b_{1}^{l} - b_{2}^{l} &= -[x_{i},x_{j}] \circ \big(\rho_{l}^{2}(\bar{x}_{\bar{i}} [x_{p},x_{q}]) x_{l}
                         + \rho_{l}^{2}(\bar{x}_{\bar{i}} [x_{p},x_{q}] x_{l})\big)
                         + d_{1}^{\mathrm{CE}}\big([x_{i},x_{j}] \circ \rho_{l}^{2}(\bar{x}_{\bar{i}} [x_{p},x_{q}]) \otimes x_{l}\big),
  \\
  a_{3}^{l} - a_{4}^{l} &= -\big(\rho_{l}^{2}(\bar{x}_{\bar{i}}[x_{p},x_{q}] x_{l})
                         + \rho_{l}^{2}(\bar{x}_{\bar{i}}[x_{p},x_{q}] x_{l})\big) \circ [x_{i},x_{j}]
                         + d_{1}^{\mathrm{CE}}\big(\rho_{l}^{2}( \bar{x}_{\bar{i}} [x_{p},x_{q}]) \circ [x_{i},x_{j}] \otimes x_{l}\big),
  \\
  b_{3}^{l} - b_{4}^{l} &=  \bar{x}_{\bar{i}} [x_{p},x_{q}] \circ \big(\rho_{l}^{2}([x_{i},x_{j}] x_{l})
                         + \rho_{l}^{2}([x_{i},x_{j}]) x_{l}\big)
                         - d_{1}^{\mathrm{CE}}\big( \bar{x}_{\bar{i}} [x_{p},x_{q}] \circ \rho_{l}^{2}([x_{i},x_{j}]) \otimes x_{l}\big),
\end{align*}
it turns out that $\delta(\bar{c})$ is the class of the cycle
\begin{equation}
\begin{split}
     \sum_{l=1}^{n} \Big( &\big(\rho_{l}^{2}([x_{i},x_{j}] x_{l}) + \rho_{l}^{2}([x_{i},x_{j}]) x_{l}\big) \circ  \bar{x}_{\bar{i}}[x_{p},x_{q}]
                          -[x_{i},x_{j}] \circ \big(\rho_{l}^{2}(\bar{x}_{\bar{i}} [x_{p},x_{q}]) x_{l} 
                          + \rho_{l}^{2}(\bar{x}_{\bar{i}} [x_{p},x_{q}] x_{l})\big)
                    \label{eq:ciclossimgendelta2}
     \\
                    -&\big(\rho_{l}^{2}(\bar{x}_{\bar{i}} [x_{p},x_{q}] x_{l}) + \rho_{l}^{2}(\bar{x}_{\bar{i}} [x_{p},x_{q}]) x_{l}\big) 
                    \circ [x_{i},x_{j}]
                    + \bar{x}_{\bar{i}} [x_{p},x_{q}] \circ \big(\rho_{l}^{2}([x_{i},x_{j}] x_{l}) + \rho_{l}^{2}([x_{i},x_{j}]) x_{l}\big) \Big).
\end{split}
\end{equation}

Using the identity \eqref{eq:delta}, we may rewrite the cycle \eqref{eq:ciclossimgendelta2} as follows
\begin{equation}
\label{eq:ciclossimgendelta3}
\begin{split}
               \Delta([x_{i},x_{j}]) \circ \bar{x}_{\bar{i}} [x_{p},x_{q}]
               -[x_{i},x_{j}] \circ \Delta(\bar{x}_{\bar{i}} [x_{p},x_{q}])
               -\Delta(\bar{x}_{\bar{i}} [x_{p},x_{q}]) \circ [x_{i},x_{j}]
               + \bar{x}_{\bar{i}} [x_{p},x_{q}] \circ \Delta([x_{i},x_{j}])
               \\
               = 2\Big(\Delta([x_{i},x_{j}]) \circ \bar{x}_{\bar{i}} [x_{p},x_{q}]
               -[x_{i},x_{j}] \circ \Delta(\bar{x}_{\bar{i}} [x_{p},x_{q}])\Big).
\end{split}
\end{equation}

Equality \eqref{eq:reduccion} tells us that 
\begin{equation}
\label{eq:id}
\begin{split}
     &[x_{i},x_{j}] \circ p_{2}\big(\sum_{l = 1}^{n} [x_{l},[x_{l},[x_{i_{1}}, \dots, [x_{i_{r}},[x_{p},x_{q}]]\dots]]]\big)
     \\
     &= [x_{i},x_{j}] \circ \Big(\Delta(\bar{x}_{\bar{i}} [x_{p},x_{q}])
     + \sum_{h = 1}^{r} \sum_{l = 1}^{n} x_{l}^{2} x_{i_{1}} \dots x_{i_{h-1}} \rho_{i_{h}}^{2}(x_{i_{h+1}} \dots x_{i_{r}} [x_{p},x_{q}])\Big)
     \\
     &= [x_{i},x_{j}] \circ \Delta(\bar{x}_{\bar{i}} [x_{p},x_{q}])
     - d_{1}^{\mathrm{CE}}\Big([x_{i},x_{j}]x_{l} \circ \sum_{h = 1}^{r} \sum_{l = 1}^{n}
     x_{i_{1}} \dots x_{i_{h-1}} \rho_{i_{h}}^{2}(x_{i_{h+1}} \dots x_{i_{r}} [x_{p},x_{q}]) \otimes x_{l}\Big)
     \\
     &-  d_{1}^{\mathrm{CE}}\Big([x_{i},x_{j}] \circ \sum_{h = 1}^{r} \sum_{l = 1}^{n}
     x_{l} x_{i_{1}} \dots x_{i_{h-1}} \rho_{i_{h}}^{2}(x_{i_{h+1}} \dots x_{i_{r}} [x_{p},x_{q}]) \otimes x_{l}\Big),
\end{split}
\end{equation}
so $\delta(\bar{c})$ is given by the cycle 
\begin{equation}
\label{eq:ciclossimgendelta4}
               2\Big(2 \sum_{l = 1}^{n} [[x_{l},x_{i}],[x_{l},x_{j}]] \circ  \bar{x}_{\bar{i}} [x_{p},x_{q}]
               -[x_{i},x_{j}] \circ p_{2}\big(\sum_{l = 1}^{n} [x_{l},[x_{l},[x_{i_{1}}, \dots, [x_{i_{r}},[x_{p},x_{q}]]\dots]]]\big)\Big).
\end{equation}

Making use of \eqref{eq:granconmutadorfinal} of Lemma \ref{lema:granconmutadorfinal} 
in the previous equation
, we obtain that   
\begin{equation}
\label{eq:ciclossimgendelta5}
\begin{split}
               4\sum_{l = 1}^{n} &\Big((\overset{*_{1}}{\overbrace{[[x_{l},x_{i}],[x_{l},x_{j}]] \circ  \bar{x}_{\bar{i}} [x_{p},x_{q}]}}
               - \overset{*_{2}}{\overbrace{[x_{i},x_{j}] \circ \bar{x}_{\bar{i}} [[x_{l},x_{p}],[x_{l},x_{q}]]}})
               \\
               &- \overset{*_{3}}
               {\overbrace{\sum_{h = 1}^{r} [x_{i},x_{j}] \circ x_{l} x_{i_{1}} \dots [[x_{l},x_{i_{h}}], \dots x_{i_{r}} [x_{p},x_{q}]]}}\Big).
\end{split}
\end{equation}
We shall denote this cycle by $c'$.

Let us now consider the $k$-linear map
\[     \xi : \Lambda^{3} W(n) \rightarrow
             \Lambda^{2} V(n) \otimes \Lambda^{2} W(n),     \]
given by
\begin{align*}
   \xi([x_{i_{1}},x_{j_{1}}] z_{1} \wedge [x_{i_{2}},x_{j_{2}}] z_{2} \wedge [x_{i_{3}},x_{j_{3}}] z_{3})
   &= z_{1}(0) x_{i_{1}} \wedge x_{j_{1}} \otimes [x_{i_{2}},x_{j_{2}}] z_{2} \wedge [x_{i_{3}},x_{j_{3}}] z_{3}
   \\
   &+ z_{2}(0) x_{i_{2}} \wedge x_{j_{2}} \otimes [x_{i_{3}},x_{j_{3}}] z_{3} \wedge [x_{i_{1}},x_{j_{1}}] z_{1}
   \\
   &+ z_{3}(0) x_{i_{3}} \wedge x_{j_{3}} \otimes [x_{i_{1}},x_{j_{1}}] z_{1} \wedge [x_{i_{2}},x_{j_{2}}] z_{2},
\end{align*}
for $z_{i} \in S(V(n))$, $i= 1,2,3$.

It is readily verified that, if $\Lambda^{2} V(n)$ is provided with the trivial action of $S(V(n))$, 
$\xi$ is an $S(V(n))$-linear and $\so(n)$-equivariant map, so it induces a morphism between the homology groups 
\[     \bar{\xi} : H_{0}(V(n),\Lambda^{3} W(n)) \rightarrow
                      H_{0}(V(n),\Lambda^{2} V(n) \otimes \Lambda^{2} W(n))
                      = \Lambda^{2} V(n) \otimes H_{0}(V(n),\Lambda^{2} W(n)).     \]

We shall prove that $\bar{\xi}(\delta(\bar{c}))$ is the class of the cycle 
\begin{equation}
\label{eq:xi}
   2 \sum_{1 \leq s < t \leq n} x_{s} \wedge x_{t} \otimes (x_{s} \wedge x_{t}).\tilde{c} 
   + \sum_{l=1}^{n} \Big( \delta_{|\bar{i}|,0} x_{p} \wedge x_{q} \otimes [x_{l},x_{i}] \wedge [x_{l},x_{j}] 
   - \delta_{|\bar{i}|,0} x_{i} \wedge x_{j} \otimes [x_{l},x_{p}] \wedge [x_{l},x_{q}] \Big),
\end{equation}
where 
$(x_{p} \wedge x_{q}).\tilde{c}$ denotes the action of $x_{p} \wedge x_{q} \in \Lambda^{2} V(n) \simeq \so(n)$ on $\tilde{c}$
(see \cite{FH1}, \S 20.1, (20.4)).
The expression \eqref{eq:xi} does not depend on the choice of $\tilde{c}$ since the differential of the Chevalley-Eilenberg complex is $\so(n)$-equivariant.

Let us compute $\bar{\xi}(\delta(\bar{c}))$.
In order to do so, it suffices to apply $\xi$ to the cycle $c'$ given in \eqref{eq:ciclossimgendelta5}.
Then $\xi(c')$ is four times the class of the cycle 
\begin{align*}
               &\sum_{l=1}^{n} \Big(
               \overset{\xi(*_{1})}{\overbrace{x_{l} \wedge x_{i} \otimes [x_{l},x_{j}] \wedge \bar{x}_{\bar{i}} [x_{p},x_{q}]
               + x_{l} \wedge x_{j} \otimes \bar{x}_{\bar{i}} [x_{p},x_{q}] \wedge [x_{l},x_{i}]+ \delta_{|\bar{i}|,0} x_{p} \wedge x_{q} \otimes [x_{l},x_{i}] \wedge [x_{l},x_{j}]}}
               \\
               &- \overset{\xi(*_{2})}{\overbrace{(\underset{\star_{1}}{\underbrace{x_{i} \wedge x_{j} \otimes \bar{x}_{\bar{i}} ([x_{l},x_{p}] \wedge [x_{l},x_{q}])}}
               + x_{l} \wedge x_{q} \otimes [x_{i},x_{j}] \wedge \bar{x}_{\bar{i}} [x_{l},x_{p}]
               + x_{l} \wedge x_{p} \otimes \bar{x}_{\bar{i}} [x_{l},x_{q}] \wedge [x_{i},x_{j}])}}
               \\
               &- \overset{\xi(*_{3}), \hskip 0.1cm \text{\begin{scriptsize}first part\end{scriptsize}}}
               {\overbrace{\sum_{h=1}^{r}\underset{\star_{2}}{\underbrace{x_{i} \wedge x_{j}
               \otimes x_{l} x_{i_{1}} \dots ([x_{l},x_{i_{h}}] \wedge x_{i_{h+1}} \dots x_{r} [x_{p},x_{q}])}}}}
               \\
               &\phantom{2 \sum_{l=1}^{n}} - \overset{\xi(*_{3}), \hskip 0.1cm \text{\begin{scriptsize}second part\end{scriptsize}}}
               {\overbrace{\sum_{h=1}^{r} x_{l} \wedge x_{i_{h}} \otimes x_{l} x_{i_{1}} \dots \hat{x}_{i_{h}} \dots x_{r} [x_{p},x_{q}] 
               \wedge [x_{i},x_{j}]}} \Big),
\end{align*}
where we have used that 
\[     \sum_{l=1}^{n} x_{p} \wedge x_{q} \otimes [x_{i},x_{j}] \wedge x_{l} x_{i_{1}} \dots x_{i_{r-1}} [x_{l},x_{i_{r}}]
        = \sum_{l=1}^{n} x_{p} \wedge x_{q} \otimes [x_{i},x_{j}] \wedge x_{i_{1}} \dots x_{i_{r-1}} [x_{l},[x_{l},x_{i_{r}}]] = 0.     \]

Since $\star_{2}$ is evidently a boundary and $\star_{1}$ is a boundary if $|\bar{i}| > 0$, thus $\xi(c')$ is equivalent to
\begin{align*}
          &2 \sum_{l=1}^{n} \Big(
               x_{l} \wedge x_{i} \otimes [x_{l},x_{j}] \wedge \bar{x}_{\bar{i}} [x_{p},x_{q}]
               - x_{l} \wedge x_{j} \otimes [x_{l},x_{i}] \wedge \bar{x}_{\bar{i}} [x_{p},x_{q}] 
               \\
               &\phantom{2 \sum_{l=1}^{n}} + \delta_{|\bar{i}|,0} x_{p} \wedge x_{q} \otimes [x_{l},x_{i}] \wedge [x_{l},x_{j}]
               - \delta_{|\bar{i}|,0} x_{i} \wedge x_{j} \otimes [x_{l},x_{p}] \wedge [x_{l},x_{q}]
               \\
               &\phantom{2 \sum_{l=1}^{n}}+ x_{l} \wedge x_{p} \otimes [x_{i},x_{j}] \wedge \bar{x}_{\bar{i}} [x_{l},x_{q}]
               - x_{l} \wedge x_{q} \otimes [x_{i},x_{j}] \wedge \bar{x}_{\bar{i}} [x_{l},x_{p}]
               \\
               &\phantom{2 \sum_{l=1}^{n}}+ \sum_{h=1}^{r} x_{l} \wedge x_{i_{h}}
               \otimes [x_{i},x_{j}] \wedge x_{l} x_{i_{1}} \dots \hat{x}_{i_{h}} \dots x_{r} [x_{p},x_{q}]  \Big)
\end{align*}
or also to
\begin{align*}
               &2 \sum_{l=1}^{n} \Big(
               x_{l} \wedge x_{i} \otimes (x_{l} \wedge x_{i}).[x_{i},x_{j}] \wedge \bar{x}_{\bar{i}} [x_{p},x_{q}]
               + x_{l} \wedge x_{j} \otimes (x_{l} \wedge x_{j}).[x_{i},x_{j}] \wedge \bar{x}_{\bar{i}} [x_{p},x_{q}]
               \\
               &\phantom{= 2 \sum_{l=1}^{n}}+ x_{l} \wedge x_{p} \otimes [x_{i},x_{j}] \wedge \bar{x}_{\bar{i}} (x_{l} \wedge x_{p}).[x_{p},x_{q}]
               + x_{l} \wedge x_{q} \otimes [x_{i},x_{j}] \wedge \bar{x}_{\bar{i}} (x_{l} \wedge x_{q}).[x_{p},x_{q}]
               \\
               &\phantom{= 2 \sum_{l=1}^{n}}+ \delta_{|\bar{i}|,0} x_{p} \wedge x_{q} \otimes [x_{l},x_{i}] \wedge [x_{l},x_{j}]
               - \delta_{|\bar{i}|,0} x_{i} \wedge x_{j} \otimes [x_{l},x_{p}] \wedge [x_{l},x_{q}]
               \\
               &\phantom{= 2 \sum_{l=1}^{n}}+ \sum_{h=1}^{r} x_{l} \wedge x_{i_{h}}
               \otimes [x_{i},x_{j}] \wedge (x_{l} \wedge x_{i_{h}})(\bar{x}_{\bar{i}}) [x_{p},x_{q}]  \Big)
\end{align*}
which may be further simplified to give 
\[     2 \sum_{1 \leq s < t \leq n} x_{s} \wedge x_{t} \otimes (x_{s} \wedge x_{t}).\tilde{c} 
       + \sum_{l=1}^{n} \delta_{|\bar{i}|,0} \Big( x_{p} \wedge x_{q} \otimes [x_{l},x_{i}] \wedge [x_{l},x_{j}] 
       - x_{i} \wedge x_{j} \otimes [x_{l},x_{p}] \wedge [x_{l},x_{q}] \Big).     \]
     
If $|\bar{i}| > 0$, \textit{i.e.} $\bar{\tilde{c}}$ has degree strictly greater than $4$, then  
\[     \bar{\xi}(\delta(\bar{c})) = 2 \sum_{1 \leq s < t \leq n} x_{s} \wedge x_{t} \otimes (x_{s} \wedge x_{t}).\bar{\tilde{c}}.     \]
In this case, if $\bar{c} \in \Ker(\delta)$, then, since $\{ x_{s} \wedge x_{t} \}_{1 \leq s < t \leq n}$ is a basis of $\Lambda^{2} V(n)$,
it must be that $(x_{s} \wedge x_{t}).\bar{\tilde{c}} = 0$, for all $1 \leq s < t \leq n$.
This in turn implies that $\bar{\tilde{c}}$ belongs to the trivial representation of $\so(n)$ in $H_{0}(V(n),W(n)^{\otimes 2})$. 
However, Corollary \ref{coro:compisownwn} tells us that this is not possible.

If $|\bar{i}| = 0$, $\bar{\tilde{c}}$ has degree $4$ and 
\begin{equation}
\label{eq:xi0}
     \bar{\xi}(\delta(\bar{c})) = 2 \sum_{1 \leq s < t \leq n} x_{s} \wedge x_{t} \otimes (x_{s} \wedge x_{t}).\bar{\tilde{c}}
       + \sum_{l=1}^{n} \Big( x_{p} \wedge x_{q} \otimes [x_{l},x_{i}] \wedge [x_{l},x_{j}] 
       - x_{i} \wedge x_{j} \otimes [x_{l},x_{p}] \wedge [x_{l},x_{q}] \Big).     
\end{equation}
If $n = 3$, from $H_{0,4}(V(n),\Lambda^{2}W(n)) = \Ker(\times^{\text{\begin{tiny}$M(n)$\end{tiny}}}_{0,1})$, we need not consider this case (see Remark \ref{rem:s2wwedge2w}). 

Let us thus assume that $n \geq 4$ and that $\bar{\tilde{c}}$ is any non trivial element of an isotypic component in 
$H_{0}(V(n),\Lambda^{2}W(n))$ different from (the ones appearing in) $\Lambda^{2} V(n)$. 
In this case, since both $\times^{\text{\begin{tiny}$M(n)$\end{tiny}}}_{0,1}$ and $\delta$ are $\so(n)$-equivariant, $\delta(\times^{\text{\begin{tiny}$M(n)$\end{tiny}}}_{0,1}(\bar{\tilde{c}}))$ vanishes if and only if 
$\delta \circ \times^{\text{\begin{tiny}$M(n)$\end{tiny}}}_{0,1}$ vanishes on the complete isotypic component to which 
$\bar{\tilde{c}}$ belongs. 
In this case, we fix 
\[     
\tilde{c} = \begin{cases}
            [e_{1},e_{2}] \wedge [e_{1},e_{4}] \in \Gamma_{2 L_{1}},     &\text{if $n = 4$,}
            \\
            [e_{1},e_{2}] \wedge [e_{1},e_{5}] \in \Gamma_{2 L_{1} + L_{2}},     &\text{if $n = 5$,}
            \\
            [e_{1},e_{2}] \wedge [e_{1},e_{3}] \in \Gamma_{2 L_{1} + L_{2} + L_{3}},     &\text{if $n \geq 7$,}
            \end{cases}     \]
where $\{ e_{1}, \dots, e_{n} \}$ is a basis of $V(n)$ for which the quadratic form of $V(n)$ is polarized (see \cite{FH1}, \S 18.1). 
We may choose this basis as follows. 
Let $m = [n/2]$ be the integral part of $n/2$. 
If $n$ is even, we define $e_{j} = (x_{j} + i x_{j+m})/\sqrt{2}$ and $e_{j+m} = (x_{j} - i x_{j+m})/\sqrt{2}$, for $1 \leq j \leq m$; 
whereas, if $n$ is odd, we also define $e_{n} = x_{n}$. 

We have intentionally omitted the case $n=6$ in the previous list since we should consider two different isotypic components: 
$\tilde{c}=[e_{1},e_{2}] \wedge [e_{1},e_{3}] \in \Gamma_{2 L_{1} + L_{2} + L_{3}}$ and 
$\tilde{c}=[e_{1},e_{2}] \wedge [e_{1},e_{6}] \in \Gamma_{2 L_{1} + L_{2} - L_{3}}$. 

We shall see that $\bar{\xi}(\delta(\times^{\text{\begin{tiny}$M(n)$\end{tiny}}}_{0,1}(\bar{\tilde{c}})))$ does not vanish in any case. 
We start recalling the following elementary fact: 
Consider a $k$-vector space $V$ of finite dimension $n$, $\phi \in (V^{*} \otimes V^{*})^{*}$ a bilinear map on $V^{*}$, 
and a basis $\{ v_{1}, \dots, v_{n} \} \subseteq V$, with dual basis $\{ v_{1}^{*}, \dots, v_{n}^{*} \} \subseteq V^{*}$. 
By the canonical identification $(V^{*} \otimes V^{*})^{*} \simeq V \otimes V$, we see that the expression 
\[     \sum_{i,j=1}^{n} \phi(v_{i}^{*},v_{j}^{*}) v_{i} \otimes v_{j} \in V^{\otimes 2}     \]
identifies with $\phi$, so it is independent of the choice of the basis. 
When $V$ is provided with a nondegenerate symmetric bilinear form $Q : V^{\otimes 2} \rightarrow k$, we may consider $\phi = Q^{-1}$, 
the inverse form of $Q$ (\textit{i.e.} the one obtained by the condition that the map $v \mapsto Q(v,\place)$ from $V$ to $V^{*}$ is an isometry). 

We shall apply the previous fact in order to rewrite equation \eqref{eq:xi0} as follows. 
First, 
\[     \sum_{l=1}^{n} x_{l} \otimes x_{l} = \sum_{l=1}^{m} (e_{l} \otimes e_{l+m} + e_{l+m} \otimes e_{l}) + \delta_{n-2m,1} e_{n} \otimes e_{n},     \]
where we have used $\phi$ equal to the inverse of the form on $V(n)$. 
Also, the nondegenerate symmetric form $K(x_{s} \wedge x_{t}, x_{s'} \wedge x_{t'}) = \delta_{s,s'} \delta_{t,t'} - \delta_{s,t'} \delta_{t,s'}$ is invariant, so it is a Killing form on $\Lambda^{2} V(n)$, coinciding with $-\frac{1}{8}\tr$ under the canonical identification $\Lambda^{2} V(n) \simeq \so(n)$ (see \cite{FH1}, \S 20.1, (20.4)). 
Hence,  
\[     \sum_{1 \leq s < t \leq n} (x_{s} \wedge x_{t}) \otimes (x_{s} \wedge x_{t}) = 
       \underset{\text{\begin{tiny}    $\begin{matrix}
                                       1 \leq s < t \leq n
                                       \\
                                       1 \leq s' < t' \leq n
                                       \end{matrix}$
                                       \end{tiny}}}
                                       {\sum} K^{-1}((e_{s} \wedge e_{t})^{*},(e_{s'} \wedge e_{t'})^{*}) 
                                       (e_{s} \wedge e_{t}) \otimes (e_{s'} \wedge e_{t'}).      \]
By the previous considerations, we may rewrite the expression \eqref{eq:xi0} for  $\bar{\xi}(\delta(\times^{\text{\begin{tiny}$M(n)$\end{tiny}}}_{0,1}(\bar{\tilde{c}})))$, 
where $\tilde{c} = [e_{1},e_{2}] \wedge [e_{1},e_{h}]$ and $h \in \{ 3, 4, 5, 6 \}$ is given according to the previous choices of cycles. 
There is one term of the form $(e_{1} \wedge e_{2}) \otimes a_{1,2}$ in the cycle representing the homology class 
$\bar{\xi}(\delta(\times^{\text{\begin{tiny}$M(n)$\end{tiny}}}_{0,1}(\bar{\tilde{c}})))$, where 
\begin{align*}
       a_{1,2} &= (e_{1+m} \wedge e_{2+m}) ([e_{1},e_{2}] \wedge [e_{1},e_{h}]) 
       - \sum_{l=1}^{m} ([e_{l},e_{1}] \wedge [e_{l+m},e_{h}] + [e_{l+m},e_{1}] \wedge [e_{l},e_{h}]) 
       \\
       &- \delta_{n-2m,1} [e_{n},e_{1}] \wedge [e_{n},e_{h}]
       \\
       &= [e_{1},e_{1+m}] \wedge [e_{1}, e_{h}] - [e_{1}, e_{2}] \wedge [e_{2+m}, e_{h}] - [e_{2+m},e_{2}] \wedge [e_{1},e_{h}]
       \\
       &- \sum_{l=1}^{m} ([e_{l},e_{1}] \wedge [e_{l+m},e_{h}] + [e_{l+m},e_{1}] \wedge [e_{l},e_{h}]) 
       - \delta_{n-2m,1} [e_{n},e_{1}] \wedge [e_{n},e_{h}],     
\end{align*}
so it does not vanish. 
Since there are no boundaries in this degree, $\bar{\xi}(\delta(\times^{\text{\begin{tiny}$M(n)$\end{tiny}}}_{0,1}(\bar{\tilde{c}})))$ is not zero, 
and \textit{a fortiori} $\delta(\times^{\text{\begin{tiny}$M(n)$\end{tiny}}}_{0,1}(\bar{\tilde{c}}))$ does not vanish. 
The proposition is thus proved.
\end{proof}

Let us now come back to the study of the spectral sequence of the previous section. 
\begin{proposition}
\label{prop:ker}
The kernel of the differential $d^{1}_{-1,3}$ is isomorphic to $V(n)[-4]$, with basis given in Lemma \ref{lema:morfb1}.
\end{proposition}
\begin{proof}
By Propositions \ref{prop:diagramaconmutativo}, \ref{prop:b1b2} and \ref{prop:vnker}, we see that 
$\Ker(d^{1}_{-1,3}) \simeq V(n)[-4]$.
\end{proof}

\section{\texorpdfstring
{Computation of $HH^{1}(\YM(n))$}
{Computation of the group of outer derivations of the Yang-Mills algebra}
}
\label{sec:hh1}

In this section we shall finally compute the group of outer derivations $HH^{1}(\YM(n))$, for $n \geq 3$. 
We recall that, since the beginning of the previous section, we have assumed that $n \geq 3$, unless we say the contrary. 

We begin by describing some derivations of $\YM(n)$.
\begin{proposition}
\label{prop:algder}
There is a homogeneous $k$-linear monomorphism of degree $0$
\[     k \oplus V(n)[2] \oplus \Lambda^{2} (V(n)[1]) \hookrightarrow HH^{1}(\YM(n)).     \]
\end{proposition}
\begin{proof}
We shall consider the following collection of plain derivations of $\YM(n)$ (in the non graded sense). 
They are homogeneous $k$-linear maps certain of degree, but satisfying the usual Jacobi identity (not the graded version).

In the first place, the homogeneous morphism of degree $0$ given by
\begin{align*}
   d_{\mathrm{eu}} : \YM(n) &\rightarrow \YM(n)
   \\
   z &\mapsto |z| z,
\end{align*}
where $z \in \YM(n)$ is homogeneous of usual degree $|z|$,
is a derivation of $\YM(n)$, which we call the Eulerian derivation.

Next, we define the derivations $d_{i}$, $i=1,\dots,n$ of degree $-1$ 
induced by the morphisms of the same name 
\begin{align*}
     d_{i} : V(n) &\rightarrow T(V(n))
     \\
     x_{j} &\mapsto \delta_{i,j}.
\end{align*}

Finally, since $\so(n) \simeq \Lambda^{2} V(n)$ acts on $\YM(n)$ by derivations of degree $0$,
we immediately obtain a homogeneous map of degree $0$ from $\Lambda^{2} (V(n)[1])$ to $\Der(\YM(n))$.

It is clear that the previous derivations induce a homogeneous $k$-linear monomorphism of degree $0$ 
\begin{equation}
\label{eq:premorfder}
     k \oplus V(n)[2] \oplus \Lambda^{2} (V(n)[1]) \hookrightarrow \Der(\YM(n)).
\end{equation}
To prove this, by degree reasons, it is only necessary to show that the set of derivations induced by the the standard basis of $\so(n)$ 
and the Eulerian derivation is linearly independent, which is direct. 

The map of the proposition is then the composition of the morphism \eqref{eq:premorfder} with the canonical projection 
\[     \Der(\YM(n)) \rightarrow \Der(\YM(n))/\mathrm{Innder}(\YM(n)) \simeq HH^{1}(\YM(n)).     \]
Since the inner derivations have degree greater than or equal to $1$, except for the zero derivation, 
it turns out that the previous composition is also injective.
The proposition is thus proved. 
\end{proof}

We devote the rest of this section to prove that the monomorphism of Proposition \ref{prop:algder} is in fact an isomorphism. 
This will be achieved by making use of the spectral sequence associated to a filtration of $S(\ym(n))$.

First, the isomorphism of equivariant $\YM(n)$-modules $\YM(n)^{\mathrm{ad}} \simeq S(\ym(n))$ implies that
\[     HH^{1} (\YM(n)) \simeq H^{1}(\ym(n),\YM(n)) \simeq H^{1}(\ym(n),S(\ym(n)))
                                = \bigoplus_{i \in \NN_{0}} H^{1}(\ym(n),S^{i}(\ym(n))).     \]
Let $I$ be the ideal of $S(\ym(n))$ generated by $\tym(n)$. 
We deduce that $I$ is also an equivariant $\YM(n)$-module, for $\tym(n)$ is an equivariant $\YM(n)$-module.
Therefore, we may consider the decreasing filtration $\{F^{\bullet} S(\ym(n))\}_{\bullet \in \ZZ}$
of equivariant $\YM(n)$-modules of $S(\ym(n))$ given by 
\[     F^{p} S(\ym(n)) = \begin{cases}
                            I^{p}, &\text{if $p \geq 1$,}
                            \\
                            S(\ym(n)), &\text{if $p \leq 0$.}
                         \end{cases}
                         \]
We see that $F^{\bullet} S(\ym(n))$ is exhaustive and Hausdorff.
Given $i \in \NN$, it induces a decreasing filtration 
$\{F^{\bullet} S^{i}(\ym(n))\}_{\bullet \in \ZZ}$ of $\ym(n)$-modules on $S^{i}(\ym(n))$, which also becomes exhaustive and Hausdorff.
For each $p \geq 0$, there is a natural isomorphism of equivariant $\YM(n)$-modules
\[     F^{p} S^{i}(\ym(n))/F^{p+1} S^{i}(\ym(n)) \simeq S^{i-p}V(n) \otimes S^{p}(\tym(n)),     \]
where the action of $\ym(n)$ on $S^{i-p}V(n)$ is trivial and the action of $\so(n)$ on each factor is the obvious one. 

This filtration provides a collection of spectral sequences ${}^{i}E_{\bullet}^{\bullet,\bullet}$ with 
\begin{equation}
\label{eq:sucesphh1}
   {}^{i}E_{1}^{p,q} = H^{p+q} (\ym(n),S^{i-p}V(n) \otimes S^{p}(\tym(n))),
\end{equation}
for all $i \in \NN_{0}$, where by definition $S^{q}(\place) = 0$, whenever $q < 0$.
Each of these spectral sequences is bounded and, hence, convergent.
The complete spectral sequence 
$E_{\bullet}^{\bullet,\bullet} = \oplus_{i \in \NN_{0}} {}^{i}E_{\bullet}^{\bullet,\bullet}$ is therefore convergent.
We remark that the collection of spectral sequences ${}^{i}E_{\bullet}^{\bullet,\bullet}$, $i \in \NN_{0}$, and its direct sum $E_{\bullet}^{\bullet,\bullet}$ are considered in the category of equivariant $\YM(n)$-modules. 

Since the isomorphism of equivariant $\YM(n)$-modules $S(\tym(n)) \simeq \U(\tym(n))^{\mathrm{ad}}$ preserves the internal 
degree, it induces an isomorphism $S^{+}(\tym(n)) \simeq \Ker(\epsilon_{\tym(n)})$, and, as a consequence,
\[     H_{\bullet}(\ym(n),\Ker(\epsilon_{\tym(n)})) \simeq H_{\bullet}(\ym(n),S^{+}(\tym(n)))
            = \bigoplus_{i \in \NN} H_{\bullet}(\ym(n),S^{i}(\tym(n))).     \]
By Theorem \ref{teo:homi}, $H_{3}(\ym(n),\Ker(\epsilon_{\tym(n)})) = 0$, which yields that 
$H_{3}(\ym(n),S^{i}(\tym(n))) = 0$ for all $i \in \NN$.
On the other hand, by the same theorem, $H_{2}(\ym(n),\Ker(\epsilon_{\tym(n)})) \simeq H_{2}(\ym(n),\tym(n)) \simeq V(n)[-4]$,
so it turns out that $H_{2}(\ym(n),S^{i}(\tym(n))) = 0$ for all $i \geq 2$.

Taking into account the Poincar\'e duality of the Yang-Mills algebra, we obtain the following result. 
\begin{proposition}
\label{prop:homo}
The homology groups $H^{0}(\ym(n),S^{i}(\tym(n)))$ and $H^{1}(\ym(n),S^{i}(\tym(n)))$ vanish for $i \geq 1$ and for $i \geq 2$, respectively. 
There is a homogeneous isomorphism of $\so(n)$-modules of degree $0$ of the form $H^{1}(\ym(n),\tym(n)) \simeq H^{1}(\ym(n),k) \simeq V(n)$.
\end{proposition}

From the previous proposition and recalling that 
\[     {}^{i}E_{\infty}^{p,q} \simeq F^{p}H^{1}(\ym(n),S^{i}(\ym(n)))/F^{p+1}H^{1}(\ym(n),S^{i}(\ym(n))),     \]
we conclude that $H^{1}(\ym(n),S^{i}(\ym(n)))$ is a direct sum (as a graded $\so(n)$-module) 
of a subquotient of ${}^{i}E_{2}^{1,0}$ and a subquotient of ${}^{i}E_{2}^{0,1}$.

\begin{center}
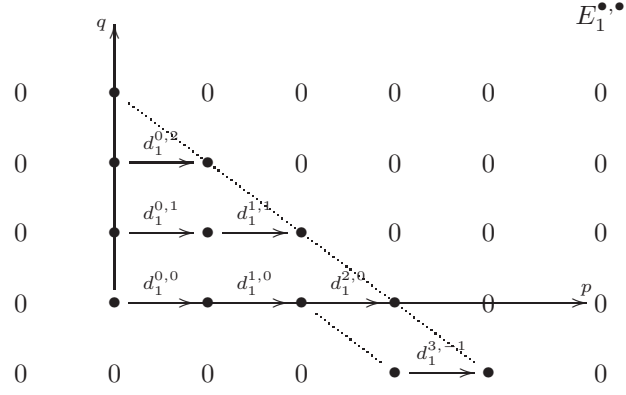
\begin{figure}[H]
\[
\xymatrix@R-10pt
{
&
&
&
&
&
&
E^{\bullet,\bullet}_{1}
&
\\
0
&
\bullet
\ar@{.}[rrrrdddd]
&
0
&
0
&
0
&
0
&
0
&
\\
0
&
\bullet
\ar[r]^{d^{0,2}_{1}}
&
\bullet
&
0
&
0
&
0
&
0
&
\\
0
&
\bullet
\ar[r]^{d_{1}^{0,1}}
&
\bullet
\ar[r]^{d_{1}^{1,1}}
&
\bullet
&
0
&
0
&
0
&
\\
0
&
\bullet
\ar[uuuu]^>{q}
\ar[rrrrr]^>{p}
\ar[r]^{d^{0,0}_{1}}
\ar@{.}[rr]
&
\bullet
\ar[r]^{d^{1,0}_{1}}
&
\bullet
\ar[r]^{d^{2,0}_{1}}
\ar@{.}[rd]
&
\bullet
&
0
&
0
&
\\
0
&
0
&
0
&
0
&
\bullet
\ar[r]^{d_{1}^{3,-1}}
&
\bullet
&
0
&
}
\]
\caption{First step $E_{1}^{\bullet,\bullet}$ of the spectral sequence.
The dotted lines indicate the limits wherein the spectral sequence is concentrated.}
\end{figure}
\end{center}

We shall begin by analyzing ${}^{i}E_{2}^{1,0}$.
In order to do so, it is necessary to compute the kernel of ${}^{i}d_{1}^{1,0}$ and the image of ${}^{i}d_{1}^{0,0}$, for $i \in \NN_{0}$. 

First, making use of the identifications
\[     {}^{i}E^{0,0}_{1} = H^{0} (\ym(n),S^{i}V(n)) \simeq S^{i}V(n)     \]
and
\[     {}^{i}E^{1,0}_{1} = H^{1}(\ym(n),S^{i-1}V(n) \otimes \tym(n)) \simeq S^{i-1}V(n) \otimes V(n)     \]
given in Proposition \ref{prop:homo}, it is clear that ${}^{i}d_{1}^{0,0}$ may be identified with the de Rham differential $d_{\mathrm{dR}}^{0}$ 
restricted to the $i$-th component of the symmetric algebra $S(V(n))$, which is given by 
\begin{equation}
\begin{split}
\label{eq:derha}
   d_{\mathrm{dR}}^{p} : S^{i}V(n) \otimes \Lambda^{p} V(n) &\rightarrow S^{i-1}V(n) \otimes \Lambda^{p+1} V(n)
   \\
   z \otimes x_{i_{1}} \wedge \dots \wedge x_{i_{p}} &\mapsto \sum_{j=1}^{n} \partial_{j} (z) \otimes x_{j} \wedge x_{i_{1}} \wedge \dots \wedge x_{i_{p}}.
\end{split}
\end{equation}


The following lemmas will be used in the forthcoming study of the spectral sequence we are dealing with. 
\begin{lemma}
\label{lema:d10}
The image of the differential 
\[     {}^{2}d_{1}^{1,0} : {}^{2}E^{1,0}_{1} = H^{1}(\ym(n),V(n) \otimes \tym(n)) \rightarrow {}^{2}E^{2,0}_{1} = H^{2}(\ym(n), S^{2}(\tym(n)))     \]
is canonically isomorphic to $\Lambda^{2} V(n)$.
In fact, the image of ${}^{2}d_{1}^{1,0}$ coincides with the image of the linear monomorphism 
\begin{equation}
\label{eq:iota}
   \bar{\iota} : \Lambda^{2} V(n) \rightarrow H^{2}(\ym(n), S^{2}(\tym(n)))
\end{equation}
given by the composition of the map $\iota : \Lambda^{2} V(n) \rightarrow Z^{2}(\YM(n), S^{2}(\tym(n)))$ defined by 
\begin{align*}
   \iota(x_{i} \wedge x_{j}) &= \sum_{p,l=1}^{n} \Big(4[x_{i},x_{l}] [[x_{j},x_{p}],x_{l}] \otimes x_{p}
   - 4 [[x_{i},x_{p}],x_{l}] [x_{j},x_{l}] \otimes x_{p}
   \\
   &\phantom{= \sum_{p,l=1}^{n}} - 2[x_{i},x_{l}] [[x_{j},x_{l}],x_{p}] \otimes x_{p}
   + 2[[x_{i},x_{l}],x_{p}] [x_{j},x_{l}] \otimes x_{p}\Big).
\end{align*}
and the canonical projection $Z^{2}(\YM(n), S^{2}(\tym(n))) \rightarrow H^{2}(\ym(n), S^{2}(\tym(n)))$. 
\end{lemma}
\begin{proof}
By Proposition \ref{prop:homo}, $H^{1}(\ym(n),V(n) \otimes \tym(n)) \simeq V(n)^{\otimes 2}$. 
Furthermore, by the description of a basis of cocycles of the cohomology group $H^{1}(\ym(n),\tym(n))$ given in the paragraph preceding 
Theorem \ref{teo:homi}, we see that the homology class in $H^{1}(\ym(n),V(n) \otimes \tym(n))$ corresponding to 
$x_{i} \otimes x_{j} \in V(n)^{\otimes 2}$ is that of the cocycle 
\[     \sum_{l=1}^{n} x_{i} \otimes [x_{j},x_{l}] \otimes x_{l} \in V(n) \otimes C^{1}(\YM(n),\tym(n)),     \]
which will be denoted by $\bar{x}_{i,j}$.
Therefore, we have that ${}^{2}d_{1}^{1,0}(\bar{x}_{i,j})$ is the class of the cocycle
\begin{align*}
   &\sum_{p,l=1}^{n} \Big(2[x_{i},x_{p}] [[x_{j},x_{l}],x_{p}] \otimes x_{l}
   - 2 [[x_{i},x_{p}],x_{l}] [x_{j},x_{l}] \otimes x_{p}
   \\
   &\phantom{\sum_{p,l=1}^{n}}-2[x_{i},x_{l}] [[x_{j},x_{l}],x_{p}] \otimes x_{p}
   + [x_{i},x_{l}] [[x_{j},x_{l}],x_{p}] \otimes x_{p}
   + [[x_{i},x_{l}],x_{p}] [x_{j},x_{l}] \otimes x_{p}\Big)
   \\
   &=\sum_{p,l=1}^{n} \Big(2[x_{i},x_{p}] [[x_{j},x_{l}],x_{p}] \otimes x_{l}
   - 2 [[x_{i},x_{p}],x_{l}] [x_{j},x_{l}] \otimes x_{p}
   \\
   &\phantom{=\sum_{p,l=1}^{n}}-[x_{i},x_{l}] [[x_{j},x_{l}],x_{p}] \otimes x_{p}
   + [[x_{i},x_{l}],x_{p}] [x_{j},x_{l}] \otimes x_{p}\Big)
   \\
   &=\sum_{p,l=1}^{n} \Big(2[x_{i},x_{l}] [[x_{j},x_{p}],x_{l}] \otimes x_{p}
   - 2 [[x_{i},x_{p}],x_{l}] [x_{j},x_{l}] \otimes x_{p}
   \\
   &\phantom{=\sum_{p,l=1}^{n}}-[x_{i},x_{l}] [[x_{j},x_{l}],x_{p}] \otimes x_{p}
   + [[x_{i},x_{l}],x_{p}] [x_{j},x_{l}] \otimes x_{p}\Big).
\end{align*}
It is readily verified that the previous identity vanishes if we take $\bar{x}^{s}_{i,j} = \bar{x}_{i,j} + \bar{x}_{j,i}$, 
for all $i, j = 1, \dots, n$.
On the other hand, if $\bar{x}_{i,j}^{a} = \bar{x}_{i,j} - \bar{x}_{j,i}$, 
${}^{2}d_{1}^{1,0}(\bar{x}^{a}_{i,j})$ is given by the cocycle 
\begin{align*}
   &\sum_{p,l=1}^{n} \Big(4[x_{i},x_{l}] [[x_{j},x_{p}],x_{l}] \otimes x_{p}
   - 4 [[x_{i},x_{p}],x_{l}] [x_{j},x_{l}] \otimes x_{p}
   \\
   &\phantom{\sum_{p,l=1}^{n}}- 2[x_{i},x_{l}] [[x_{j},x_{l}],x_{p}] \otimes x_{p}
   + 2[[x_{i},x_{l}],x_{p}] [x_{j},x_{l}] \otimes x_{p}\Big).
\end{align*}
Hence, ${}^{2}d_{1} ^{1,0}(\bar{x}^{a}_{i,j})$ is the class of a cocycle of degree $8$.

Also, $\iota$ is injective,  and we can see this as follows. 
Let us first consider the case $n \neq 4$. 
Indeed, the fact that the map $\iota$ is a non trivial $\so(n)$-equivariant and that $\Lambda^{2}V(n)$ is irreducible 
implies that $\iota$ is monomorphic. 
For the case $n = 4$ we proceed analogously, but taking into account that 
$\Lambda^{2} V(n) \simeq \Gamma_{L_{1} + L_{2}} \oplus \Gamma_{L_{1} - L_{2}}$ 
and $\iota$ does not vanish in any direct summand. 
 
From the form of the complex $C^{\bullet}(\YM(n),S^{2}(\tym(n)))$, we see that the space 
$B^{2}(\YM(n), S^{2}(\tym(n)))_{8}$ of coboundaries of degree $8$ 
is an epimorphic image of $S^{2}(\tym(n))_{4} \simeq S^{2}(\Lambda^{2}V(n))$ under the $\so(n)$-equivariant map 
$d_{3}$. 
It is not hard to prove that the intersection between $B^{2}(\YM(n), S^{2}(\tym(n)))_{8}$ and the image of $\iota$ 
is trivial, so $\bar{\iota}$ is also injective. 
This can be deduced from arguments on isotypic components, which we now explain. 
If $n \neq 6$, Remark \ref{rem:s2wwedge2w} tells us that $\mathrm{Im}(\iota) \simeq \Lambda^{2}V(n)$ is not an 
isotypic component of $S^{2}(\Lambda^{2}V(n))$, which yields that the previous intersection is trivial. 
The case $n = 6$ is analogous. 
\end{proof}

\begin{lemma}
Let $p : V(n)^{\otimes 2} \rightarrow \Lambda^{2} V(n)$ be the canonical projection. 
The following diagram is commutative
\[
\xymatrix@C-10pt
{
{}^{i}E^{1,0}_{1}
\ar@{=}[d]
\ar[rr]^{{}^{i}d^{1,0}_{1}}
&
&
{}^{i}E^{2,0}_{1}
\ar@{=}[d]
\\
H^{1}(\ym(n),S^{i-1}V(n) \otimes \tym(n))
\ar[d]^{\simeq}
&
&
H^{2}(\ym(n),S^{i-2}V(n) \otimes S^{2}\tym(n))
\\
S^{i-1}V(n) \otimes V(n)
\ar[r]^{d_{\mathrm{dR}}^{0} \otimes \id_{V(n)}}
&
S^{i-2}V(n) \otimes V(n)^{\otimes 2}
\ar[r]^{\id_{S^{i-2}V(n)} \otimes p}
&
S^{i-2}V(n) \otimes \Lambda^{2} V(n)
\ar@{^{(}->}[u]^{\id_{S^{i-2}V(n)} \otimes \bar{\iota}}
}
\]

Also, note that $(\id_{S^{i-2}V(n)} \otimes p) \circ (d^{0}_{\mathrm{dR}} \otimes \id_{V(n)}) = d^{1}_{\mathrm{dR}}$. 
\end{lemma}
\begin{proof}
The second statement is direct. 
In order to prove the first statement of the lemma, it suffices to restrict to the case that $\bar{c} \in H^{1}(\ym(n),S^{i-1}V(n)\otimes \tym(n))$ is represented by
\[     \sum_{l=1}^{n} z \otimes [x_{j},x_{l}] \otimes x_{l},     \]
where $z = x_{j_{1}} \dots x_{j_{i-1}} \in S^{i-1}V(n)$. 
Notice that we have used the description of a basis of cocycles of the cohomology group $H^{1}(\ym(n),\tym(n))$ given in the paragraph preceding Theorem \ref{teo:homi}. 
Thus, ${}^{i}d_{1}^{1,0}(\bar{c})$ is the cohomology class of the cocycle
\begin{align*}
   &\sum_{l,g=1}^{n} \sum_{h=1}^{i-1}
   \Big(2 x_{j_{1}} \dots [x_{j_{h}},x_{g}] \dots x_{j_{i-1}} \otimes [[x_{j},x_{l}],x_{g}] \otimes x_{l}
   \\
   &\phantom{\sum_{l,g=1}^{n} \sum_{h=1}^{i-1}}
   - 2 x_{j_{1}} \dots [[x_{j_{h}},x_{g}],x_{l}] \dots x_{j_{i-1}} \otimes [x_{j},x_{l}] \otimes x_{g}
   \\
   &\phantom{\sum_{l,g=1}^{n} \sum_{h=1}^{i-1}}
   -2 x_{j_{1}} \dots [x_{j_{h}},x_{l}] \dots x_{j_{i-1}} \otimes [[x_{j},x_{l}],x_{g}] \otimes x_{g}
   \\
   &\phantom{\sum_{l,g=1}^{n} \sum_{h=1}^{i-1}}
   + x_{j_{1}} \dots [x_{j_{h}},x_{l}] \dots x_{j_{i-1}} \otimes [[x_{j},x_{l}],x_{g}] \otimes x_{g}
   \\
   &\phantom{\sum_{l,g=1}^{n} \sum_{h=1}^{i-1}}
   + x_{j_{1}} \dots [[x_{j_{h}},x_{l}],x_{g}] \dots x_{j_{i-1}} \otimes [x_{j},x_{l}] \otimes x_{g}\Big),
\end{align*}
which can also be rewritten as
\begin{align*}
   &\sum_{l,g=1}^{n} \sum_{r=1}^{n} \partial_{r} (z)
   \Big(2 [x_{r},x_{g}] \otimes [[x_{j},x_{l}],x_{g}] \otimes x_{l}
   - 2 [[x_{r},x_{g}],x_{l}] \otimes [x_{j},x_{l}] \otimes x_{g}
   \\
   &\phantom{\sum_{l,g=1}^{n} \sum_{r=1}^{n} \partial_{r} (z)} 
   -2 [x_{r},x_{l}] \otimes [[x_{j},x_{l}],x_{g}] \otimes x_{g}
   + [x_{r},x_{l}] \otimes [[x_{j},x_{l}],x_{g}] \otimes x_{g}
   \\
   &\phantom{\sum_{l,g=1}^{n} \sum_{r=1}^{n} \partial_{r} (z)} 
   + [[x_{r},x_{l}],x_{g}] \otimes [x_{j},x_{l}] \otimes x_{g}\Big) 
   = \sum_{r=1}^{n} \partial_{r} (z) \otimes \iota(x_{r} \wedge x_{j}).
\end{align*}
The lemma is then proved. 
\end{proof}

As a direct consequence of the previous lemmas and the fact that $H^{1}_{\mathrm{dR}}(S(V(n))) = 0$ 
(see \cite{Wei1}, Cor. 9.9.3), we obtain the following proposition. 
\begin{proposition}
The following diagram 
\[
\xymatrix
{
{}^{i}E^{0,0}_{1}
\ar[r]^{{}^{i}d^{0,0}_{1}}
\ar[d]^{\simeq}
&
{}^{i}E^{1,0}_{1}
\ar[r]^{{}^{i}d^{1,0}_{1}}
\ar[d]^{\simeq}
&
{}^{i}E^{2,0}_{1}
\\
S^{i}V(n)
\ar[r]^(.36){d_{\mathrm{dR}}^{0}}
&
S^{i-1}V(n) \otimes V(n)
\ar[r]^(.48){d_{\mathrm{dR}}^{1}}
&
S^{i-2}V(n) \otimes \Lambda^{2} V(n)
\ar@{^{(}->}[u]^{\id_{S^{i-2}V(n)} \otimes \bar{\iota}}
}
\]
is commutative.
Since $H^{1}_{\mathrm{dR}}(S(V(n))) = 0$, this implies that ${}^{i}E^{1,0}_{2} = 0$, for all $i \in \NN_{0}$.
\end{proposition}

By the proposition we conclude that ${}^{i}E^{1,0}_{2} = 0$ for $i \in \NN_{0}$. 
Hence, $H^{1}(\ym(n),S^{i}(\ym(n)))$ is isomorphic
(as graded a $\so(n)$-module) to a subquotient of ${}^{i}E_{2}^{0,1}$.
Since ${}^{i}E_{2}^{0,1} = \Ker({}^{i}d_{1}^{0,1})$, it will be convenient to make this map explicit. 

\begin{lemma}
\label{lema:d01}
The image of the differential
\[     {}^{1}d_{1}^{0,1} : {}^{1}E^{0,1}_{1} = H^{1}(\ym(n),V(n)) \rightarrow {}^{1}E^{1,1}_{1} = H^{2}(\ym(n), \tym(n))     \]
is naturally isomorphic to $S^{2}_{\mathrm{irr}} V(n) \simeq S^{2} V(n)/k.\bar{q}$, where $\bar{q} = \sum_{i=1}^{n} x_{i} \otimes x_{i}$.
In fact, the image of ${}^{1}d_{1}^{0,1}$ coincides with the image of the linear monomorphism 
\begin{equation}
\label{eq:iota'}
   \bar{\iota}' : S^{2}_{\mathrm{irr}} V(n) \rightarrow H^{2}(\ym(n), \tym(n))
\end{equation}
given by the composition of the map $\iota' : S^{2}_{\mathrm{irr}} V(n) \rightarrow Z^{2}(\YM(n), \tym(n))$ defined as 
\[     \iota'(\overline{x_{i} \otimes_{s} x_{j}}) = \sum_{p,l=1}^{n} (-2[[x_{i},x_{p}],x_{j}] \otimes x_{p} - 2 [[x_{j},x_{p}],x_{i}] \otimes x_{p})
\]
and the canonical projection $Z^{2}(\YM(n), \tym(n)) \rightarrow H^{2}(\ym(n), \tym(n))$. 
\end{lemma}
\begin{proof}
By Proposition \ref{prop:homo}, $H^{1}(\ym(n),V(n)) \simeq V(n)^{\otimes 2}$. 
Analogously to Lemma \ref{lema:d10}, we find that, if $\bar{x}_{i,j}$ is the cohomology class of the cocycle 
$x_{i} \otimes x_{j} \in V(n) \otimes C^{1}(\YM(n), k)$, then ${}^{1}d_{1}^{0,1}(\bar{x}_{i,j})$ is the cohomology class of the cocycle 
\begin{equation}
\label{eq:cociclos2}
      \sum_{p=1}^{n} (-2[[x_{i},x_{p}],x_{j}] \otimes x_{p} + [[x_{i},x_{j}],x_{p}] \otimes x_{p}).
\end{equation}
The Jacobi identity tells us that ${}^{1}d_{1}^{0,1}(\bar{x}_{i,j}^{a})$ vanishes if $\bar{x}_{i,j}^{a}$ is 
the cohomology class of the cocycle $x_{i} \otimes x_{j} - x_{j} \otimes x_{i} \in V(n) \otimes C^{1}(\YM(n), k)$, for all $i, j = 1, \dots, n$.
Moreover, by the Yang-Mills relations \eqref{eq:relym}, ${}^{1}d_{1}^{0,1}(\bar{q}) = 0$, with $\bar{q} = \sum_{i=1}^{n} \bar{x}_{i,i}$. 

On the other hand, if $\bar{x}_{i,j}^{s}$ is the cohomology class of $x_{i} \otimes x_{j} + x_{j} \otimes x_{i}$ (for $1 \leq i \leq j \leq n$), 
${}^{1}d_{1}^{0,1}(\bar{x}_{i,j}^{s})$ is given by the cocycle
\begin{equation}
\label{eq:cocicloef}
     \sum_{p=1}^{n} (-2[[x_{i},x_{p}],x_{j}] \otimes x_{p} - 2 [[x_{j},x_{p}],x_{i}] \otimes x_{p}).
\end{equation}
Therefore, ${}^{1}d_{1}^{0,1}(\bar{x}_{i,j}^{s})$ is the class of a cocycle of degree $6$.
However, there are also coboundaries in this degree, which are of the form 
\begin{equation}
\label{eq:cobordes}
     \sum_{1 \leq a < b \leq n} \sum_{p=1}^{n} c_{a,b} [[x_{a},x_{b}],x_{p}] \otimes x_{p}.
\end{equation}
Hence, we see that the cocycle \eqref{eq:cocicloef} is equivalent to
\begin{equation}
\label{eq:cociclofinal}
     -4 \sum_{p=1}^{n} [[x_{i},x_{p}],x_{j}] \otimes x_{p}.
\end{equation}

It is easily checked that $\iota'$ is injective, which follows from the fact that the map $\iota'$ is a non trivial $\so(n)$-equivariant and that $S^{2}_{\mathrm{irr}}V(n)$ is an irreducible $\so(n)$-module. 
 
Considering the complex $C^{\bullet}(\YM(n),\tym(n))$, we see that the subspace 
$B^{2}(\YM(n), \tym(n))_{6}$ spanned by the coboundaries of degree $6$ 
is an epimorphic image of $(\tym(n))_{4} \simeq \Lambda^{2}V(n)$ under the $\so(n)$-equivariant map 
$d_{3}$. 
It is not hard to prove that the intersection between $B^{2}(\YM(n), \tym(n))_{6}$ and the image of $\iota'$ 
is trivial, since $\mathrm{Im}(\iota') \simeq S^{2}_{\mathrm{irr}}V(n)$ is not an 
isotypic component of $\Lambda^{2}V(n)$. 
This implies that $\bar{\iota}'$ is monomorphic. 
\end{proof}

\begin{lemma}
\label{lema:p'}
If we denote by $p' : V(n)^{\otimes 2} \rightarrow S^{2}_{\mathrm{irr}} V(n)$ the canonical projection, 
the following diagram is commutative
\[
\xymatrix
{
{}^{i}E^{0,1}_{1}
\ar@{=}[d]
\ar[rr]^{{}^{i}d^{0,1}_{1}}
&
&
{}^{i}E^{1,1}_{1}
\ar@{=}[d]
\\
H^{1}(\ym(n),S^{i}V(n))
\ar[d]^{\simeq}
&
&
H^{2}(\ym(n),S^{i-1}V(n) \otimes \tym(n))
\\
S^{i}V(n) \otimes V(n)
\ar[r]^(.46){d_{\mathrm{dR}}^{0} \otimes \id_{V(n)}}
&
S^{i-1}V(n) \otimes V(n)^{\otimes 2}
\ar[r]^{\id_{S^{i-1}V(n)} \otimes p'}
&
S^{i-1}V(n) \otimes S^{2}_{\mathrm{irr}}V(n)
\ar@{^{(}->}[u]^{\id_{S^{i-1}V(n)} \otimes \bar{\iota}'}
}
\]
\end{lemma}
\begin{proof}
It is enough to prove the lemma when $\bar{c} \in H^{1}(\ym(n),S^{i}V(n))$ is represented by 
$z \otimes x_{j} \in S^{i}V(n) \otimes C^{1}(\YM(n),k)$,
where $z = x_{j_{1}} \dots x_{j_{i}} \in S^{i}V(n)$.
If this is the case, ${}^{i}d_{1}^{0,1}(\bar{c})$ is the cohomology class of the cocycle
\begin{align*}
   &\sum_{l=1}^{n} \sum_{h=1}^{i} (-2 x_{j_{1}} \dots [[x_{j_{h}},x_{l}],x_{j}] \dots x_{j_{i}} \otimes x_{l}
   + x_{j_{1}} \dots [[x_{j_{h}},x_{j}],x_{l}] \dots x_{j_{i}} \otimes x_{l})
   \\
   &= \sum_{l=1}^{n} \sum_{r=1}^{n} \partial_{r}(z) \otimes (-2[[x_{r},x_{l}],x_{j}] + [[x_{r},x_{j}],x_{l}]) \otimes x_{l}.
\end{align*}
\end{proof}

The previous lemma implies the following result.
\begin{proposition}
\label{prop:e2}
The space ${}^{i}E_{2}^{0,1}$ vanishes for $i \geq 3$.
Furthermore, 
\begin{itemize}
\item[(1)] ${}^{0}E_{2}^{0,1}$ is the vector space with basis given by the cohomology class of the cocycles $\{ x_{i} : i = 1, \dots, n \}$,
where $x_{i} \in V(n) = C^{1}(\YM(n),k)$,

\item[(2)] ${}^{1}E_{2}^{0,1}$ is the vector space with basis given by the cohomology class of the cocycles
\[     \Big\{ x_{i} \otimes x_{j} - x_{j} \otimes x_{i} : 1 \leq i < j \leq n \Big\} \cup \Big\{\sum_{i=1}^{n} x_{i} \otimes x_{i} \Big\} 
       \subseteq C^{1}(\YM(n),V(n)),     \]

\item[(3)] ${}^{2}E_{2}^{0,1}$ is the vector space with basis given by the cohomology class of the cocycles
\[     \Big\{ \sum_{i=1}^{n} x_{j}x_{i} \otimes x_{i} - \frac{1}{2} x_{i}^{2} \otimes x_{j} : i = 1, \dots, n \Big\} 
       \subseteq C^{1}(\YM(n),S^{2}V(n)).     \]
\end{itemize}
\end{proposition}
\begin{proof}
First, it is direct to prove that the collection of elements of $C^{1}(\YM(n),k)$ given in item (1) 
is indeed a basis of cocycles, since $H^{1}(\YM(n),k) \simeq V(n)$.
On the other hand, Lemma \ref{lema:d01} says that, $\Ker({}^{1}d_{1}^{0,1})$ is generated by the cocycles given in item (2). 

Let $i \geq 2$. 
We consider
\[     z = \sum_{j=1}^{n} z_{j} \otimes x_{j} \in S^{i} V(n) \otimes V(n)     \]
a representative of a cohomology class $\bar{z}$ in ${}^{i}E_{1}^{0,1}$.
We notice that 
\[     (\id_{S^{i-1}V(n)} \otimes p') \circ (d_{\mathrm{dR}}^{0} \otimes \id_{V(n)}) (\sum_{j=1}^{n} z_{j} \otimes x_{j})
       = \sum_{j,h=1}^{n} \partial_{h}(z_{j}) \otimes \overline{x_{h} \otimes_{s} x_{j}} \in S^{i-1} V(n) \otimes S^{2}_{\mathrm{irr}} V(n).
\]
Therefore, by Lemma \ref{lema:p'}, $\bar{z} \in \Ker({}^{i}d_{1}^{0,1})$ if and only if the following conditions are satisfied:
\begin{itemize}
\label{item:condicion1}
\item[(i)] $\partial_{h} z_{j} = - \partial_{j} z_{h}$, for all $h, j = 1, \dots, n$ such that $h \neq j$,
\label{item:condicion2}
\item[(ii)] $\partial_{h} z_{h} = \partial_{j} z_{j}$, for all $h, j = 1, \dots, n$.
\end{itemize}

We shall first analyze the case $i = 2$.
In order to do so, we shall assume that 
\[     z_{j} = \sum_{m,l=1}^{n} a_{l,m}^{j} x_{l} x_{m} \in S^{2} V(n),     \]
where $a^{j}_{l,m} = a^{j}_{m,l} \in k$, for all $l,m = 1, \dots, n$.
The previous conditions are respectively equivalent to 
\begin{itemize}
\item[(a)] $a^{j}_{l,m} = - a^{l}_{j,m}$, for all $j, l, m = 1, \dots, n$ such that $j \neq l$,

\item[(b)] $a^{j}_{j,m} = a^{l}_{l,m}$, for all $j, l, m = 1, \dots, n$.
\end{itemize}
The first condition implies that, if $j,l,m$ are all different, then
\[     - a^{m}_{l,j} = a^{j}_{l,m} = - a^{l}_{j,m} = a^{m}_{j,l},      \]
so, it must be $a^{j}_{l,m} = 0$.
Also, both conditions yield that, given $j \neq l$,
\[     - a^{l}_{j,j} = a^{j}_{l,j} = a^{l}_{l,l}.      \]
We shall denote $\alpha_{l} = a^{l}_{l,l}$. 

Applying these considerations we may simplify the expression of $z$ as follows 
\begin{align*}
     z  = \sum_{j,l,m=1}^{n} a_{l,m}^{j} x_{l} x_{m} \otimes x_{j}
       &= \sum_{j=1}^{n} \Big(2 \underset{\text{\begin{tiny}
                                       $\begin{matrix}
                                       1 \leq m \leq n
                                       \\
                                       m \neq j
                                       \end{matrix}$
                                       \end{tiny}}}
                                       {\sum} a_{j,m}^{j} x_{j} x_{m} \otimes x_{j}
                         + \underset{\text{\begin{tiny}
                                       $\begin{matrix}
                                       1 \leq m \leq n
                                       \\
                                       m \neq j
                                       \end{matrix}$
                                       \end{tiny}}}
                                       {\sum} a_{m,m}^{j} x_{m}^{2} \otimes x_{j}
                         + a^{j}_{j,j} x_{j}^{2} \otimes x_{j}\Big)
        \\
        &= \sum_{j=1}^{n} \Big(\underset{\text{\begin{tiny}$
                                     \begin{matrix}
                                       1 \leq m \leq n
                                       \\
                                       m \neq j
                                       \end{matrix}$
                                       \end{tiny}}}
                                       {\sum} 2 \alpha_{m} x_{j} x_{m} \otimes x_{j}
                         - \underset{\text{\begin{tiny}$
                                       \begin{matrix}
                                       1 \leq m \leq n
                                       \\
                                       m \neq j
                                       \end{matrix}$
                                       \end{tiny}}
                                       }{\sum} \alpha_{j} x_{m}^{2} \otimes x_{j}
                         + \alpha_{j} x_{j}^{2} \otimes x_{j}\Big)
        \\
        &= \sum_{m=1}^{n} 2 \alpha_{m} \Big(\sum_{j=1}^{n} x_{j} x_{m} \otimes x_{j} - \frac{1}{2} x_{j}^{2} \otimes x_{m}\Big),
\end{align*}
where we have omitted the terms with $a^{j}_{l,m}$ ($j,l,m$ all different) in the last member of the first line, since they vanish.
In consequence, we have proved that ${}^{2}E_{1}^{0,1}$ is spanned by the basis given in item (3). 

We shall now show that ${}^{i}E_{1}^{0,1} = 0$ for $i \geq 3$.
This is a direct consequence of the following auxiliary lemma.
\begin{lemma}
\label{lema:tec}
Let $n \geq 3$ and let $p_{1}, \dots, p_{n}$ be homogeneous polynomials of degree $i \geq 3$ in $k[x_{1},\dots ,x_{n}]$ 
which satisfy that 
\begin{itemize}
\item[(I)] $\partial_{h} p_{j} = - \partial_{j} p_{h}$, for all $h, j = 1, \dots, n$ such that $h \neq j$,

\item[(II)] $\partial_{h} p_{h} = \partial_{j} p_{j}$, for all $h, j = 1, \dots, n$.
\end{itemize}
Then, $p_{1} = \dots = p_{n} = 0$.
\end{lemma}
\begin{proof}
We choose different elements $j_{1}, j_{2}, j_{3} \in \{ 1, \dots, n \}$.
Applying condition (I), it turns out that
\[     \partial_{j_{2}} \partial_{j_{3}} p_{j_{1}}  = \partial_{j_{3}} \partial_{j_{2}} p_{j_{1}}
                                                    = - \partial_{j_{3}} \partial_{j_{1}} p_{j_{2}}
                                                    = - \partial_{j_{1}} \partial_{j_{3}} p_{j_{2}}
                                                    = \partial_{j_{1}} \partial_{j_{2}} p_{j_{3}}
                                                    = \partial_{j_{2}} \partial_{j_{1}} p_{j_{3}}
                                                    = - \partial_{j_{2}} \partial_{j_{3}} p_{j_{1}}.
\]
Therefore, $\partial_{j_{2}} \partial_{j_{3}} p_{j_{1}} = 0$, if $j_{1}, j_{2}, j _{3}$ are all different. 
This in turn implies that 
\begin{equation}
\label{eq:forma}
     p_{j} = a_{j} x_{j}^{i}
                  + \underset{\text{\begin{tiny}$\begin{matrix} h = 1, \dots, n \\ h \neq j \end{matrix}$\end{tiny}}}{\sum} \sum_{d=1}^{i}
                  a_{h,d}^{j} x_{j}^{i-d} x_{h}^{d}
\end{equation}
for $j \in \{ 1, \dots, n \}$.
Also, if $j_{1}, j_{2} \in \{ 1, \dots, n \}$ are two different elements, conditions (I) and (II) tell us that 
\[     \partial_{j_{1}}^{2} p_{j_{1}}  = \partial_{j_{1}} \partial_{j_{1}} p_{j_{1}}
                                       = \partial_{j_{1}} \partial_{j_{2}} p_{j_{2}}
                                       = \partial_{j_{2}} \partial_{j_{1}} p_{j_{2}}
                                       = - \partial_{j_{2}} \partial_{j_{2}} p_{j_{1}}.
\]
Hence, $\partial_{j_{2}}^{2}p_{j_{1}} = \partial_{j_{3}}^{2}p_{j_{1}}$, for all $j_{2},j_{3} \neq j_{1}$. 
Using this identity in equation \eqref{eq:forma}, we conclude that 
\[     p_{j} = a_{j} x_{j}^{i}
                  + \underset{\text{\begin{tiny}$\begin{matrix} h = 1, \dots, n \\ h \neq j \end{matrix}$\end{tiny}}}{\sum} \sum_{d=1}^{2}
                  a_{h,d}^{j} x_{j}^{i-d} x_{h}^{d}     \]
for all $j \in \{ 1, \dots, n \}$.
If $j \neq j'$, 
\[     \partial_{j'} p_{j} =  a_{j',1}^{j} x_{j}^{i-1} + 2 a_{j',2}^{j} x_{j}^{i-2} x_{j'}.     \]
In particular, condition (I) says that $a_{j',1}^{j} = 0$ for all $j, j' = 1, \dots, n$ such that $j \neq j'$, so 
\begin{equation}
\label{eq:forma2}
     p_{j} = a_{j} x_{j}^{i}
                  + \underset{\text{\begin{tiny}$\begin{matrix} h = 1, \dots, n \\ h \neq j \end{matrix}$\end{tiny}}}{\sum} 
                  a_{h,2}^{j} x_{j}^{i-2} x_{h}^{2}.
\end{equation}
 
We need to consider two cases: $i > 3$ and $i=3$. 
If $i > 3$, condition (I) also tells us that $a_{j',2}^{j} = 0$, for all $j, j' = 1, \dots, n$ such that $j \neq j'$. 
In this case, $p_{j} = a_{j} x_{j}^{i}$, for all $j \in \{ 1, \dots, n \}$, and condition (II) implies that $a_{j} = 0$, for all 
$j \in \{ 1, \dots, n \}$, so $p_{j}$ vanishes for all $j \in \{ 1, \dots, n \}$. 

If $i=3$, we recall that 
\[     \partial_{j} p_{j} = 3 a_{j} x_{j}^{2}
                  + \underset{\text{\begin{tiny}$\begin{matrix} h = 1, \dots, n \\ h \neq j \end{matrix}$\end{tiny}}}{\sum}
                  a_{h,2}^{j} x_{h}^{2}.     \]
Identity $\partial_{j}p_{j} = \partial_{j'}p_{j'}$ implies that $a_{j} = 0$, for all $j \in \{ 1, \dots, n \}$, and $a_{h,2}^{j} = 0$, for 
$h \neq j,j'$. 
Since $n \geq 3$, $p_{j}$ vanishes for all $j \in \{ 1, \dots, n \}$. 
The lemma is thus proved. 
\end{proof}

The proof of the proposition is then complete. 
\end{proof}


\begin{proposition}
\label{prop:e3}
The kernel of ${}^{2}d_{2}^{0,1}$ vanishes.
In consequence, ${}^{2}E^{0,1}_{3} = 0$.
\end{proposition}
\begin{proof}
We first observe that ${}^{2}E_{2}^{0,1}$ is isomorphic to $V(n)$ as $\so(n)$-modules, so it is an irreducible $\so(n)$-module.
If we apply the differential ${}^{2}d_{2}^{0,1}$ to the cohomology class represented by a cocycle of the form 
\[     \sum_{l=1}^{n} \Big(x_{j}x_{l} \otimes x_{l} - \frac{1}{2} x_{l}^{2} \otimes x_{j}\Big),     \]
we obtain the cohomology class of the cocycle in $C^{2}(\YM(n),S^{2}(\tym(n)))$ given by 
\begin{equation}
\label{eq:coe2}
     \sum_{l,m=1}^{n} (2[x_{j},x_{m}] \otimes_{s} [x_{l},x_{m}] \otimes x_{l} - 2[x_{j},x_{l}] \otimes_{s} [x_{l},x_{m}] \otimes x_{m}
                         - [x_{l},x_{m}] \otimes_{s} [x_{l},x_{m}] \otimes x_{j}).     
\end{equation}
We point out that the cohomology classes of the previous cocycles are linearly independent, which implies that $\Ker({}^{2}d^{0,1}_{2}) = 0$. 
This can be deduced as follows. 
Taking into account ${}^{2}E_{2}^{0,1}$ is an irreducible $\so(n)$-module and ${}^{2}d_{2}^{0,1}$ is $\so(n)$-equivariant, 
the latter is an isomorphism if it does not vanish. 
Since there are no coboundaries of the same internal degree and the cocycles \eqref{eq:coe2} are nonzero, 
we conclude that $\Ker({}^{2}d^{0,1}_{2}) = 0$. 
\end{proof}

By Propositions \ref{prop:algder}, \ref{prop:e2} and \ref{prop:e3}, we derive the main result of this section. 
\begin{theorem}
\label{teo:hh1}
The morphism given in Proposition \ref{prop:algder} is bijective. 
Furthermore, there is an isomorphism of Lie algebras   
\[     HH^{1}(\YM(n)) \simeq V(n) \rtimes (\so(n) \times k),     \]
where $HH^{1}(\YM(n))$ is provided with the Gerstenhaber bracket, 
$k$ and $V(n)$ are considered as abelian subalgebras and $\so(n)$ acts on $V(n)$ by the standard action. 
\end{theorem}
\begin{proof}
We only need to prove the second statement. 
By Proposition \ref{prop:algder}, a basis of representatives of outer derivations is given by the derivations $d_{\mathrm{eu}}$, 
$d_{i}$ ($i = 1, \dots, n$), and the collection of derivations $d_{i,j}$ ($1 \leq i < j \leq n$) coming from the canonical basis of $\so(n)$ 
(when identified with $\Lambda^{2} V(n)$), which act on $\YM(n)$ as 
\[     d_{i,j} (x_{k}) = 2 (\delta_{j,k} x_{i} -  \delta_{i,k} x_{j}).     \]
From this it is easy to prove that $[d_{\mathrm{eu}},d_{i,j}] = [d_{i},d_{j}] = 0$, 
$[d_{i,j}, d_{k}] = 2 (\delta_{j,k} d_{i} -  \delta_{i,k} d_{j})$ and 
$[d_{\mathrm{eu}},d_{i}] = - d_{i}$. 
Hence, the Lie algebra $HH^{1}(\YM(n))$ with the Gerstenhaber bracket is isomorphic to $V(n) \rtimes (\so(n) \times k)$, 
where $k$ and $V(n)$ are considered as abelian subalgebras and $\so(n)$ acts on $V(n)$ by the standard action. 
\end{proof}
\section{\texorpdfstring
{Hochschild and cyclic homology of $\YM(n)$}
{Hochschild and cyclic homology of the Yang-Mills algebra}
}
\label{sec:hhhc}

\subsection{Generalities}

In this subsection, $A$ shall denote a connected graded $k$-algebra (\textit{i.e.} $A_{0}=k$).
We shall denote by $HC_{\bullet}(A)$ the $\bullet$-th cyclic homology group of $A$
and $\overline{HC}_{\bullet}(A) = HC_{\bullet}(A)/HC_{\bullet}(k)$ the reduced $\bullet$-th cyclic homology group.
Also, $\overline{HH}_{\bullet}(A) = HH_{\bullet}(A)/HH_{\bullet}(k)$
shall denote the reduced $\bullet$-th Hochschild homology group.
We recall that $\overline{HH}_{\bullet}(A) = HH_{\bullet}(A)$, for $\bullet \geq 1$, 
$\overline{HH}_{0}(A) = HH_{0}(A)/k$, $HH_{0}(A) = HC_{0}(A)$ and 
$\overline{HH}_{0}(A) = \overline{HC}_{0}(A)$.
In fact, all of them are not only abelian groups but also $k$-vector spaces. 

As it is usual, if $A$ is provided with an $\NN_{0}$-grading, then the cyclic homology has two gradings: the homological grading and the internal one. 
So, we shall denote by $HC_{i,j}(A)$ and $\overline{HC}_{i,j}(A)$ the components of internal degree $j$ of 
$HC_{i}(A)$ and $\overline{HC}_{i}(A)$, respectively.
As a consequence, these groups are graded vector spaces with respect to the internal grading. 

If $A$ is an $\NN_{0}$-graded algebra, the relation between the previous homologies is 
provided by the following collection of short exact sequence of graded vector spaces (see \cite{Wei1}, Thm. 9.9.1)
\begin{equation}
\label{eq:hhhc}
   0 \rightarrow \overline{HC}_{i-1} (A) \rightarrow \overline{HH}_{i} (A) \rightarrow \overline{HC}_{i} (A) \rightarrow 0, 
\end{equation}
for all $i \geq 0$, derived from Connes' long exact sequence.
This is a corollary of Goodwillie's Theorem proved by M. Vigu\'e-Poirrier (see \cite{Good1} and \cite{VP1}). 


We recall that the Euler-Poincar\'e characteristic for the Hilbert series is given by 
\begin{equation}
\label{eq:caraceuler}
     \chi_{\overline{HC}_{\bullet}(A)(t)} = \sum_{p \in \ZZ} (-1)^{p} \overline{HC}_{p}(A)(t).
\end{equation}

The following proposition is a particular case of Theorem 3.5, Eq. (16), in \cite{I1}.  
\begin{proposition}
\label{prop:seriehilbertalggrad}
If $A$ is an $\NN_{0}$-graded algebra, then 
\[     \chi_{\overline{HC}_{\bullet}(A)(t)} = \sum_{l \geq 1} \frac{\varphi(l)}{l} \log(A(t^{l})),     \]
where $\varphi$ denotes the Euler function. 
\end{proposition}

Corollary 3 of \cite{CD1} and Proposition \ref{prop:seriehilbertalggrad} yield the following result. 
\begin{proposition}
\label{prop:xc}
If $\YM(n)$ denotes the Yang-Mills algebra with $n$ generators provided with the usual grading, then 
\[     \chi_{\overline{HC}_{\bullet}(\YM(n))(t)} = \sum_{l \geq 1} \frac{\varphi(l)}{l} \log(\YM(n)(t^{l}))
                                                 = - \sum_{l \geq 1} \frac{\varphi(l)}{l} \log(1 - n t^{l} + n t^{3l} - t^{4l}).     \]
\end{proposition}

\subsection{Hochschild and cyclic homology of the Yang-Mills algebra}

Corollary 3 of \cite{CD1}, Proposition \ref{prop:centroym}, Theorem \ref{teo:hh1} and Proposition \ref{prop:xc} tell us that 
\begin{align*}
   \overline{HH}_{3}(\YM(n))(t) &= t^{4},
   \\
   \overline{HH}_{2}(\YM(n))(t) &= \Big(\frac{n(n-1)}{2} + 1 \Big) t^{4} + n t^{3},
   \\
   \chi_{\overline{HC}_{\bullet}(\YM(n))(t)} &= - \sum_{l \geq 1} \frac{\varphi(l)}{l} \log(1 - n t^{l} + n t^{3l} - t^{4l}),
\end{align*}
where we have used the Poincar\'e duality of the Yang-Mills algebra.
We shall find the Hilbert series of the other homology $k$-vector spaces by putting together the following facts. 

First, taking into account the short exact sequence \eqref{eq:hhhc} and that $HH_{\bullet}(\YM(n)) = 0$ for $\bullet \geq 4$,
we conclude that $\overline{HC}_{\bullet}(\YM(n)) = 0$ for $\bullet \geq 3$.
Moreover,
\begin{equation}
\label{eq:hhhcrel}
\begin{split}
     \overline{HC}_{2}(\YM(n))(t) &= \overline{HH}_{3}(\YM(n))(t),
     \\
     \overline{HC}_{1}(\YM(n))(t) &= \overline{HH}_{2}(\YM(n))(t) - \overline{HC}_{2}(\YM(n))(t)
     \\
                                  &= \overline{HH}_{2}(\YM(n))(t) - \overline{HH}_{3}(\YM(n))(t),
     \\
     \overline{HC}_{0}(\YM(n))(t) &= \overline{HH}_{0}(\YM(n))(t).
\end{split}
\end{equation}

Second, as noted in Eq. (1.22) of \cite{CD2}, the Koszul property of $\YM(n)$ implies that 
\begin{equation}
\label{eq:koszulalgym}
     \sum_{i=0}^{3} (-1)^{i} \overline{HH}_{i}(\YM(n))(t) = 0.
\end{equation}
Finally, it may be directly checked from \eqref{eq:hhhc} that
\begin{equation}
\label{eq:hhhcrel2}
     \chi_{\overline{HC}_{\bullet}(\YM(n))(t)} = \sum_{i=0}^{3} (-1)^{i} \overline{HC}_{i}(\YM(n))(t) = \sum_{i=0}^{3} (-1)^{i} (3-i) \overline{HH}_{i}(\YM(n))(t).
\end{equation}

These two last identities constitute a linear system 
\begin{align*}
     3 \overline{HH}_{0}(\YM(n))(t) - 2 \overline{HH}_{1}(\YM(n))(t) &= \chi_{\overline{HC}_{\bullet}(\YM(n))(t)} - \overline{HH}_{2}(\YM(n))(t),
     \\
     \overline{HH}_{0}(\YM(n))(t) - \overline{HH}_{1}(\YM(n))(t) &= \overline{HH}_{3}(\YM(n))(t) - \overline{HH}_{2}(\YM(n))(t),
\end{align*}
with unique solution
\begin{align*}
     \overline{HH}_{0}(\YM(n))(t) &= \chi_{\overline{HC}_{\bullet}(\YM(n))(t)} - 2 \overline{HH}_{3}(\YM(n))(t) + \overline{HH}_{2}(\YM(n))(t),
     \\
     \overline{HH}_{1}(\YM(n))(t) &= \chi_{\overline{HC}_{\bullet}(\YM(n))(t)} - 3 \overline{HH}_{3}(\YM(n))(t) + 2 \overline{HH}_{2}(\YM(n))(t),
\end{align*}
Hence, we have proved the main Theorem \ref{teo:hhhcym>3}.

\section{\texorpdfstring{Appendix: Hochschild homology of $\YM(2)$}{Appendix: Hochschild homology of YM(2)}}
\label{sec:Heishom}

As a simple application of the Koszul complex \eqref{eq:complejohomologiayangmills},
we shall compute the Hochschild homology and cohomology of $\YM(2)$. 
This result is known in the literature (see \cite{Nu1}, Chap. III, Thm. 3.2), but we provide more explicit computations. 

Since $\ym(2) \simeq \h_{1}$ (see \cite{HS1}, Example 2.1), $\ym(2)$ has a basis $\{ x,y,z \}$ as $k$-vector space, 
such that $[x,y] = z$ and $z \in \Z(\ym(2))$. 
Notice that when $\ym(2)$ is provided with the usual grading, $x$ and $y$ have degree $1$, whereas $z$ has degree $2$.
We shall write $k[x,y,z]$ instead of $S(\ym(2))$.

It can be easily proved that the right action of $\ym(2)$ on $k[x,y,z]$ is as follows: $p.x = - z \partial p/\partial y$, 
$p.y = z \partial p/\partial x$ and $p.z = 0$, for $p \in k[x,y,z]$. 

Given $p = \sum_{(i,j,l) \in \NN_{0}^{3}} a_{i,j,l} x^{i} y^{j} z^{l} \in k[x,y,z]$, 
define $\int p dx = \sum_{(i,j,l) \in \NN_{0}^{3}} a_{i,j,l} (i+1)^{-1} x^{i+1} y^{j} z^{l}$. 
One can check that
\begin{equation}
\label{eq:conmutacion}
     \frac{\partial}{\partial x} \int p dx = p,  \hskip 0.5cm  \int \frac{\partial p}{\partial x} dx = p-p(0,y,z)
     \hskip 0.5cm \text{and} \hskip 0.5cm \frac{\partial}{\partial y} \int p dx = \int \frac{\partial p}{\partial y} dx.
\end{equation}
Analogous results hold when considering variables $y$ and $z$.

The Koszul complex of $\YM(2)$ is
\begin{equation}
\label{eq:complejohomologiayangmills2}
    0 \longrightarrow k[x,y,z][-4] \overset{d_{3}}{\longrightarrow} k[x,y,z] \otimes V(2)[-2] \overset{d_{2}}{\longrightarrow} k[x,y,z] \otimes V(2)
    \overset{d_{1}}{\longrightarrow} k[x,y,z] \longrightarrow 0,
\end{equation}
with differential
\begin{align*}
   d_{1} (p \otimes x + q \otimes y) &= z \Big(\frac{\partial q}{\partial x} - \frac{\partial p}{\partial y}\Big), \hskip 0.87cm
   d_{3} (r) = - z \frac{\partial r}{\partial y} \otimes x +  z \frac{\partial r}{\partial x} \otimes y,
   \\
   d_{2}(p \otimes x + q \otimes y) &= z^{2} \Big(\frac{\partial^{2}p}{\partial x^{2}} \otimes x + \frac{\partial^{2}p}{\partial x \partial y} \otimes y
                                             + \frac{\partial^{2}q}{\partial y^{2}} \otimes y 
                                             + \frac{\partial^{2}q}{\partial x \partial y} \otimes x\Big),
\end{align*}
where $p,q,r \in k[x,y,z]$.

We see that $H_{3}(\ym(2),\YM(2)^{\mathrm{ad}}) \simeq \Ker(d_{3})$ and that 
$r \in \Ker(d_{3})$ if and only if its partial derivatives with respect to $x$ and $y$ vanish, \textit{i.e.}
if $r \in k[z]$.
As a consequence we get an homogeneous isomorphism $HH_{3}(\YM(2)) \simeq k[z][-4]$ of degree $0$.

Moreover, by Poincar\'e duality, we immediately have that $HH^{0}(\YM(2)) \simeq \Z(\YM(2)) \simeq k[z]$.
Since the image of $d_{1}$ is the set of polynomials of the form $z p$, where $p \in k[x,y,z]$
we see that $HH_{0}(\YM(2)) \simeq k[x,y]$ of degree $0$,

Let us now compute $HH_{2}(\YM(2))$.
Let $\omega = p \otimes x + q \otimes y \in \Ker(d_{2})$.
This is equivalent to the following conditions:
\[
   \frac{\partial}{\partial x} \left( \frac{\partial p}{\partial x} + \frac{\partial q}{\partial y} \right) = 0,
   \hskip 5mm
   \frac{\partial}{\partial y} \left( \frac{\partial p}{\partial x} + \frac{\partial q}{\partial y} \right) = 0.
\]
If we write $p = \sum_{i \in \NN_{0}} p_{i} z^{i}$ and $q = \sum_{i \in \NN_{0}} q_{i} z^{i}$, for $p_{i}, q_{i} \in k[x,y]$, 
for all $i \in \NN_{0}$, the conditions are equivalent to: 
\[ \frac{\partial}{\partial x} \left( \frac{\partial p_{i}}{\partial x} + \frac{\partial q_{i}}{\partial y} \right) = 0,
   \hskip 5mm
   \frac{\partial}{\partial y} \left( \frac{\partial p_{i}}{\partial x} + \frac{\partial q_{i}}{\partial y} \right) = 0, 
   \quad \forall \hskip 0.6mm i \in \NN_{0}.
\]
Then, for all $i \in \NN_{0}$,
\begin{equation}
\label{eq:c}
     \frac{\partial p_{i}}{\partial x} + \frac{\partial q_{i}}{\partial y} = c_{i} \in k.
\end{equation}

We may choose $r = \sum_{i \in \NN_{0}} r_{i} z^{i} \in k[x,y,z]$, with $r_{i} \in k[x,y]$ such that 
\begin{equation}
\label{eq:eleccion}
     r_{i} = \int q_{i+1} dx - \int p_{i+1}(0,y) dy, \forall \hskip 0.5mm i \in \NN_{0}.
\end{equation}

Then,
\[     d_{3} (r) = - \sum_{i \in \NN_{0}} z^{i+1} \frac{\partial r_{i}}{\partial y} \otimes x +  z^{i+1} \frac{\partial r_{i}}{\partial x} \otimes y.     \]

As a consequence, the cycle $\omega$ is homologous to
\begin{align*}
     p \otimes x + q \otimes y - d_{3}(r) 
     &= p_{0} \otimes x + q_{0} \otimes y
       + \sum_{i \in \NN} \Big( z^{i} \Big(p_{i} + \frac{\partial r_{i-1}}{\partial y}\Big) \otimes x 
       + z^{i} \Big(q_{i} - \frac{\partial r_{i-1}}{\partial x}\Big) \otimes y \Big)
     \\
       &= p_{0} \otimes x + q_{0} \otimes y + \sum_{i \in \NN} z^{i}\Big( p_{i} + \frac{\partial r_{i-1}}{\partial y}\Big) \otimes x.     
\end{align*}
From \eqref{eq:c} and \eqref{eq:eleccion} we see that
\begin{align*}
     p_{i} + \frac{\partial r_{i-1}}{\partial y} 
     &= p_{i} + \frac{\partial}{\partial y} \Big(\int q_{i} dx - \int p_{i}(0,y) dy \Big)
     = p_{i} + \int \frac{\partial q_{i}}{\partial y} dx - p_{i}(0,y)
     \\
       &= p_{i} + \int \Big(c_{i} - \frac{\partial p_{i}}{\partial x}\Big) dx - p_{i}(0,y)
       = c_{i} x.    
\end{align*}
Hence, $\omega$ is homologous to the cycle $\omega '=p_{0} \otimes x + q_{0} \otimes y + \sum_{i \in \NN} z^{i} c_{i} x \otimes x$. 
If $c_{i} \neq 0$, $z^{i} c_{i} x \otimes x$ cannot be a boundary because all boundaries have $c_{i} = 0$.
Since $\omega '$ is a cycle, it we must have $\partial q_{0}/\partial y = c_{0} - \partial p_{0}/\partial x$,
and then
\[     q_{0} = c_{0} y - \int \frac{\partial p_{0}}{\partial x} dy + h,     \]
where $h \in k[x]$ is some polynomial.
Therefore the cycle $\omega$ is homologous to
\[     p_{0} \otimes x + c_{0} y \otimes y - \int \frac{\partial p_{0}}{\partial x} dy \otimes y + h \otimes y
        + \sum_{i \in \NN} z^{i} c_{i} x \otimes x.     \]
From this we can conclude that the set given by 
\[     \{ y \otimes y, 
          x^{i_{1}} \otimes y \hskip 0.6mm (i_{1} \in \NN_{0}), 
          x^{i_{2}} y^{i_{3}} \otimes x - \frac{i_{2}}{i_{3}+1} x^{i_{2}-1} y^{i_{3} +1} \otimes y \hskip 0.6mm (i_{2}, i_{3} \in \NN_{0}), 
          z^{i_{4}} x \otimes x \hskip 0.6mm (i_{4} \in \NN_{0}) \}     \]
is a basis of $HH_{2}(\YM(2))$.

In the same way we may compute the homology $HH_{1}(\YM(2))$.
If $\omega = p \otimes x + q \otimes y$ is a $1$-cycle, then
$\partial q_{i}/\partial x - \partial p_{i}/\partial y = 0$,
for all $i \in \NN_{0}$.
Hence there is a polynomial $r_{i} \in k[x,y]$ such that $p_{i} = \partial r_{i}/\partial x$ and
$q_{i} = \partial r_{i}/\partial y$, for all $i \in \NN_{0}$.

If we choose $p'_{i}$ and $q'_{i}$ such that $\partial p'_{i-2}/\partial x + \partial q'_{i-2}/\partial y = r_{i}$,
for all $i \geq 2$, then $\omega$ is homologous to
\[     p_{0} \otimes x + q_{0} \otimes y + z p_{1} \otimes x + z q_{1} \otimes y
       = \frac{\partial r_{0}}{\partial x} \otimes x + \frac{\partial r_{0}}{\partial y} \otimes y
       + z \frac{\partial r_{1}}{\partial x} \otimes x + z \frac{\partial r_{1}}{\partial y} \otimes y.     \]
Moreover, we immediately see that the collection of cycles with $r_{0} = x^{i_{1}}y^{i_{2}} \in k[x,y]$ and
$r_{1} = x^{i_{3}} y^{i_{4}} \in k[x,y]$, with $i_{1}, i_{2}, i_{3}, i_{4} \in \NN_{0}$ and $(i_{1}, i_{2}) \neq (0,0)$, $(i_{3},i_{4}) \neq (0,0)$,
gives a basis of $HH_{1}(\YM(2))$.

Using the previous computations is clear to compute the Hilbert series for the Hochschild homology of $\YM(2)$. 
For completeness, we state the Hilbert series for the Hochschild and cyclic homology of $\YM(2)$, where the Hilbert series for the cyclic homology was obtained from that of the Hochschild homology using relations \eqref{eq:hhhcrel} for $n=2$. 
\begin{theorem}
\label{teo:hhhcym2}
If $n = 2$, then the Hilbert series for the Hochschild homology are 
\begin{align*}
   HH_{\bullet}(\YM(2))(t) &= 0, &\text{if $\bullet \geq 4$,}
   \\
   HH_{3}(\YM(2))(t) &= \frac{t^{4}}{1-t^{2}},
   \\
   HH_{2}(\YM(2))(t) &= 2 t^{3} \frac{1+t-t^{2}}{(1-t^{2})(1-t)},
   \\
   HH_{1}(\YM(2))(t) &= t \frac{(2-t)(1+t^{2})}{(1-t)^{2}},
   \\
   HH_{0}(\YM(2))(t) &= \frac{1}{(1-t)^{2}}.
\end{align*}
The Hochschild cohomology is given by Poincar\'e duality: $HH^{\bullet}(\YM(2)) = HH_{3-\bullet}(\YM(2))[4]$, for $0 \leq \bullet \leq 3$, 
and $HH^{\bullet}(\YM(2)) = 0$, for $\bullet > 3$. 

Also, the Hilbert series for the cyclic homology are given by
\begin{align*}
   &HC_{4+2\bullet}(\YM(2))(t) = 1, &\text{if $\bullet \geq 0$,}
   \\
   &HC_{3+2\bullet}(\YM(2))(t) = 0, &\text{if $\bullet \geq 0$,}
   \\
   &HC_{2}(\YM(2))(t) = 1 + \frac{t^{4}}{1-t^{2}},
   \\
   &HC_{1}(\YM(2))(t) = \frac{(2-t)t^{3}}{(1-t)^{2}},
   \\
   &HC_{0}(\YM(2))(t) = \frac{1}{(1-t)^{2}}.
\end{align*}
\end{theorem}


\begin{thebibliography}{99}
%
%

\bibitem{Ber1} Berger, R. \textit{Koszulity for nonquadratic algebras}. J. Algebra \textbf{239}, (2001), no. 2,
pp. 705--734.

\bibitem{Ber2} Berger, R. \textit{La cat\'egorie des modules gradu\'es sur une alg\`ebre gradu\'ee}. Preprint.\\
\protect{\url{http://webperso.univ-st-etienne.fr/~rberger/mes-textes.html}}.

\bibitem{BDVW1} Berger, R.; Dubois-Violette, M.; Wambst, M. \textit{Homogeneous algebras}. J. Algebra \textbf{261},
(2003), no. 1, pp. 172--185.

\bibitem{BM1} Berger, R.; Marconnet, N. \textit{Koszul and Gorenstein Properties for Homogeneous Algebras}. Algebr. Represent. Theory \textbf{9}, (2006), no. 1, pp. 67--97. 

\bibitem{BT1} Berger, R.; Taillefer, R. \textit{Poincar\'e-Birkhoff-Witt deformations of Calabi-Yau algebras}. 
J. of Noncommutative Geometry \textbf{1}, (2007), no. 2, pp. 241--270.

\bibitem{CE1} Cartan, H.; Eilenberg, S. \textit{Homological algebra}. Princeton University Press, Princeton, N.
J., 1956.

\bibitem{Ch-L1} Chambert-Loir, A. \textit{Alg\`ebre commutative}. Cours enseign\'e \`a l'UPMC, Paris, 2000-2001.\\
\protect{\url{http://perso.univ-rennes1.fr/antoine.chambert-loir/publications/teach/algcom.djvu}}.

\bibitem{CD1} Connes, A.; Dubois-Violette, M. \textit{Yang-Mills Algebra}. Lett. Math. Phys. \textbf{61}, (2002), no. 2, pp. 149--158.

\bibitem{CD2} Connes, A.; Dubois-Violette, M. \textit{Yang-Mills and some related algebras}. Rigorous quantum field
theory, pp. 65--78, Progr. Math., \textbf{251}, Birkh\"auser, Basel, 2007.

\bibitem{Dix1} Dixmier, J. \textit{Enveloping algebras}. Revised reprint of the 1977 translation.
Graduate Studies in Mathematics, \textbf{11}, American Mathematical Society, Providence, RI, 1996.

\bibitem{Doug1} Douglas, M. \textit{Report on the status of the Yang-Mills Millennium Prize Problem}. Preprint.\\
\protect{\url{http://www.claymath.org/millennium/Yang-Mills\_Theory/ym2.pdf}}.

\bibitem{FH1} Fulton, W.; Harris, J. \textit{Representation theory. A first Course}. Graduate Texts in Mathematics, \textbf{129}.
Readings in Mathematics. Springer-Verlag, New York, 1991.

\bibitem{Ge1} Gerstenhaber, M. \textit{On the deformation of rings and algebras}. Ann. of Math. \textbf{79}, (1964), no. 2, pp. 59--103. 

\bibitem{Good1} Goodwillie, T. \textit{Cyclic homology, derivations and the free loop space}. Topology \textbf{24}, (1985), no. 2, pp. 187--215. 

\bibitem{Hart1} Hartshorne, R. \textit{Algebraic geometry}. Graduate Texts in Mathematics, \textbf{52}, Springer-Verlag, New York-Heidelberg, 1977.

\bibitem{Hers1} Herscovich, E. \textit{Teor\'{\i}a de representaciones y homolog\'{\i}a de \'algebras de Yang-Mills}.
PhD. thesis in Mathematics, Universidad de Buenos Aires, Argentina, April 2008.\\
\protect{\url{http://cms.dm.uba.ar/doc/tesisherscovich.pdf}}.

\bibitem{HS1} Herscovich, E.; Solotar, A. \textit{Representation theory of Yang-Mills algebras}.
Accepted for publication in Annals of Mathematics. 
Preprint available at \protect{\url{http://pjm.math.berkeley.edu/scripts/coming.php?jpath=annals}}.

\bibitem{Hum2} Humphreys, J. \textit{Linear algebraic groups}. Graduate Texts in Mathematics, No. 21. Springer-Verlag, New York-Heidelberg, 1975.

\bibitem{Huy1} Huybrechts, D. \textit{Complex geometry. An introduction}. Universitext, Springer-Verlag, Berlin, 2005.

\bibitem{I1} Igusa, K. \textit{Cyclic homology and the determinant of the Cartan matrix}. J. of Pure and Applied Algebra \textbf{83}, 
(1992), pp. 101--119.

\bibitem{Ka1} Kassel, C. \textit{Cyclic homology, comodules and mixed complexes}. J. Algebra \textbf{107}, (1987), no. 1, pp. 195--216. 

\bibitem{Ko1} Kontsevich, M. \textit{Formal (non)commutative symplectic geometry}. The Gelfand Mathematical Seminars, 1990--1992,  173--187, Birkh\"auser Boston, Boston, MA, 1993.

\bibitem{KR1} Kontsevich, M.; Rosenberg, A. \textit{Noncommutative smooth spaces}.  The Gelfand Mathematical Seminars, 1996--1999,  85--108, Gelfand Math. Sem., Birkh\"auser Boston, Boston, MA, 2000.

\bibitem{La1} Lang, S. \textit{Algebra}. Revised third edition. Graduate Texts in Mathematics, \textbf{211}, Springer-Verlag, New York, 2002.

\bibitem{Mi1} Miyanishi, M. \textit{Algebraic geometry}.
Translated from the 1990 Japanese original by the author.
Translations of Mathematical Monographs, \textbf{136}, American Mathematical Society, Providence, RI, 1994.

\bibitem{Mov1} Movshev, M. \textit{On deformations of Yang-Mills algebras}. Preprint.\\
\protect{\url{http://arxiv.org/abs/hep-th/0509119}}.

\bibitem{Mov2} Movshev, M. \textit{Cohomology of Yang-Mills algebras}. J. Noncommut. Geom. {\bf 2}, (2008), pp. 353--404.

\bibitem{MS1} Movshev, M.; Schwarz, A. \textit{Algebraic structure of Yang-Mills theory}. The unity of mathematics, pp. 473--523, Progr. Math.,
\textbf{244}, Birkh\"auser Boston, Boston, MA, 2006.

\bibitem{NO1} N{\u{a}}st{\u{a}}sescu, C.; Van Oystaeyen, F. \textit{Methods of graded rings}. Lecture Notes in Mathematics, \textbf{1836}, Springer-Verlag, Berlin, 2004. 

\bibitem{Ne1} Nekrasov, N. \textit{Lectures on open strings and noncommutative gauge fields}.
Unity from duality: gravity, gauge theory and strings (Les Houches, 2001), pp. 477--495,
NATO Adv. Study Inst., EDP Sci., Les Ulis, 2003.

\bibitem{Nu1} Nuss, P. \textit{Homologie des alg\`ebres commutatives et presque commutatives}.
Th\`ese de Doctorat, Sp\'ecialit\'e: Ma\-th\'e\-ma\-ti\-ques, Universit\'e Louis Pasteur, Strasbourg, France, f\'evrier 1990.


\bibitem{Rot1} Rotman, J. \textit{An introduction to homological algebra}. Second edition. Universitext. Springer, New York, 2009.

\bibitem{VP1} Vigu\'e-Poirrier, M. \textit{Cyclic homology of algebraic hypersurfaces}. J. of Pure and Applied Algebra \textbf{72}, (1991), pp. 95--108. 

\bibitem{Wei1} Weibel, C. \textit{An introduction to homological algebra}. Cambridge Studies in Advanced
Mathematics, {\bf 38}. Cambridge University Press, Cambridge, 1994.
Errata: \protect{\url{http://www.math.rutgers.edu/~weibel/Hbook-corrections.html}}

\bibitem{Zel1} \v{Z}elobenko, D. \textit{Compact Lie groups and their representations}.
Translated from the Russian by Israel Program for Scientific Translations.
Translations of Mathematical Monographs, Vol. \textbf{40}, American Mathematical Society, Providence, RI, 1973.

\end{thebibliography}
%


\noindent Estanislao Herscovich
          \\
          Departamento de Matem\'atica,
          \\
          Facultad de Ciencias Exactas y Naturales, 
          \\
          Universidad de Buenos Aires,
          \\
          Pabell\'on I, Ciudad Universitaria, 
          \\
          1428, Buenos Aires, Argentina          
          \\
          \href{mailto:eherscov@dm.uba.ar}{eherscov@dm.uba.ar} 

\smallskip
          
\noindent  Andrea Solotar
           \\
           Departamento de Matem\'atica,
           \\
           Facultad de Ciencias Exactas y Naturales, 
           \\
           Universidad de Buenos Aires,
           \\
           Pabell\'on I, Ciudad Universitaria, 
           \\
           1428, Buenos Aires, Argentina
           \\
           \href{mailto:asolotar@dm.uba.ar}{asolotar@dm.uba.ar}
\end{document}